\documentclass[12pt]{article}
\usepackage{amsmath,amssymb,amsthm,bm}
\usepackage{enumitem}
\usepackage{color}
\usepackage{tikz-cd}
\usepackage{url}

\theoremstyle{plain}
\newtheorem{thm}{Theorem}[section]
\newtheorem{lem}[thm]{Lemma}
\newtheorem{prop}[thm]{Proposition}
\newtheorem{cor}[thm]{Corollary}
\newtheorem{prob}[thm]{Problem}

\theoremstyle{definition}
\newtheorem{rem}[thm]{Remark}
\newtheorem{df}[thm]{Definition}

\makeatletter

\@addtoreset{equation}{section}
\makeatother



\newcommand{\C}{\mathbb{C}}

\newcommand{\Z}{\mathbb{Z}}
\newcommand{\N}{\mathbb{N}}

\newcommand{\Span}{\mathop{\mathrm{Span}}\nolimits}

\newcommand{\MD}{\mathrm{multideg}}
\newcommand{\LT}{\mathrm{LT}}

\title{Imprimitive association schemes and elimination theory}
\author{ Akihiro Higashitani\thanks{Department of Pure and Applied Mathematics, Graduate School of Information Science and Technology, Osaka University, Yamadaoka 1-5, Suita, Osaka 565-0871, Japan}, Hirotake Kurihara\thanks{Department of Applied Science, Yamaguchi University, 2-16-1 Tokiwadai, Ube 755-8611, Japan}}
\date{}
\begin{document}
\maketitle

\begin{abstract}
We prove that a commutative association scheme is imprimitive if and only if it admits a multivariate $P$- or $Q$-polynomial structure with respect to an elimination-type monomial order. This provides a direct bridge between the classical theory of block and quotient schemes for imprimitive association schemes and elimination theory in computational commutative algebra. For an imprimitive multivariate $P$- or $Q$-polynomial association scheme, we determine the induced multivariate polynomial structures on the quotient and block schemes and describe their associated polynomials via explicit specializations, variable deletions, and rescalings of the original associated polynomials. At the level of zero-dimensional ideals, we show that the ideal of the block scheme is exactly an elimination ideal, whereas the ideal of the quotient scheme is obtained by adjoining the valency relations for the eliminated variables and then eliminating. As applications, we study direct products and crested products from the viewpoint of multivariate polynomiality, and we characterize the schemes that are multivariate $P$- or $Q$-polynomial with respect to every monomial order as precisely the direct products of univariate $P$- or $Q$-polynomial schemes. We also discuss formal duality, composition series, and several related open problems.
\end{abstract}
\textbf{Keywords}: 
imprimitive association schemes; multivariate polynomial association schemes; monomial orders; Gr\"obner bases; elimination theory.

\noindent
\textbf{2020 Mathematics Subject Classification}: Primary 05E30; Secondary 13P10, 13P15.

\tableofcontents

\section{Introduction}
\label{sec:intro}

Commutative association schemes form one of the basic frameworks of algebraic combinatorics. They have deep connections to finite geometry, coding theory, special functions, and harmonic analysis on finite spaces. Among the structural notions in the subject, \emph{imprimitivity} is especially fundamental. A nontrivial closed subset of an imprimitive association scheme gives rise to an equivalence relation on the underlying set and hence to two canonical descendants: the quotient scheme and the block scheme. In this way, imprimitivity governs a basic decomposition mechanism in the theory. On the metric side, it encompasses the familiar bipartite and antipodal decompositions of distance-regular graphs. On the cometric (i.e., $Q$-polynomial) side, it underlies the $Q$-bipartite/$Q$-antipodal dichotomy and its exceptional cases; see, for example, \cite{BI1984,BCN1989,godsilMartin1995,martinMuzychukWilliford2007,suzuki1998,vanDamMartinMuzychuk2013}. It also interacts fruitfully with formal duality and product constructions. Examples include Curtin's work on the Bose--Mesner algebras of block and quotient schemes \cite{curtin2008} and Bailey--Cameron's crested products \cite{baileycameron2005}. More recently, subtle feasibility questions for imprimitive parameter sets have continued to motivate new work \cite{vidali2025}.

Parallel to this classical line of research, a multivariate extension of $P$-polynomiality and $Q$-polynomiality has begun to take shape. Long before a general definition was available, multivariate orthogonal-polynomial phenomena were already visible in specific commutative association schemes and related Gelfand-pair situations. Examples include $q$-Krawtchouk-type and $(n+1,m+1)$-hypergeometric families, wreath products, generalized Hamming schemes, and schemes based on attenuated spaces \cite{CST2006,Godsil2010,Kurihara2013,Mizukawa2004,Mizukawa2007,MT2004,Stanton1980}. In the last few years, these scattered examples have been brought under a common conceptual framework. Bernard, Cramp\'{e}, Poulain d'Andecy, Vinet, and Zaimi introduced bivariate $P$-polynomial (and likewise bivariate $Q$-polynomial) association schemes \cite{bi}. Bannai, Kurihara, Zhao, and Zhu then reformulated the notion in terms of arbitrary monomial orders and established a general theory of multivariate $P$- and $Q$-polynomial association schemes \cite{BKZZ}. Since then the subject has developed further. One now has graph-theoretic counterparts in the form of $m$-distance-regular graphs \cite{BCVZZ24}, explicit bivariate $Q$-polynomial structures for nonbinary Johnson schemes and association schemes obtained from attenuated spaces \cite{BKZZ2,biQ}, and further bivariate $P$- and $Q$-polynomial structures for attenuated-space schemes \cite{BernardAttenuated2025}. Thus multivariate polynomiality is currently in an active phase of development. 

An important feature of the multivariate theory is that the definition itself depends on a choice of monomial order. From the viewpoint of computational commutative algebra, this is not merely technical. Monomial orders govern Gr\"obner bases and, crucially for the present paper, elimination ideals and elimination theorems for polynomial systems (see, e.g., \cite{Cox}). 
This perspective has already appeared in the context of association schemes. Indeed, the Bose--Mesner algebra of a commutative association scheme can be realized as a quotient by a zero-dimensional ideal, and Gr\"obner basis methods have been used to study polynomiality and related algorithmic properties of association schemes \cite{martinezMoro2001}. It is therefore natural to ask whether specific monomial orders encode specific combinatorial features of an association scheme.

The main point of this paper is that imprimitivity is exactly the combinatorial phenomenon detected by elimination orders. More precisely, Theorem~\ref{thm:EliminationThm} shows that, for a commutative association scheme $\mathfrak{X}$, the following are equivalent: 
\begin{enumerate}[label=$(\roman*)$]
    \item $\mathfrak{X}$ is imprimitive;
    \item there exist $\mathcal{D}\subset \N^\ell$ and a monomial order of $s$-elimination type such that $\mathfrak{X}$ is an $\ell$-variate $P$-polynomial association scheme on $\mathcal{D}$ with respect to that order;
    \item there exist $\mathcal{D}^\ast\subset \N^{\ell^\ast}$ and a monomial order of $s^\ast$-elimination type such that $\mathfrak{X}$ is an $\ell^\ast$-variate $Q$-polynomial association scheme on $\mathcal{D}^\ast$ with respect to that order.
\end{enumerate}
Thus a classical decomposition property of imprimitive association schemes turns out to be precisely an elimination-theoretic property inside the recent multivariate polynomial framework. 

Once this connection is established, the behavior of quotient and block schemes becomes transparent from the multivariate viewpoint. Imprimitive association schemes come with canonical quotient schemes and block schemes. We prove that their multivariate $P$- or $Q$-polynomial structures are obtained systematically from that of the original scheme. On the level of associated polynomials, this amounts to deleting variables, specializing variables to valencies or multiplicities, and applying the natural rescalings dictated by quotient parameters; see Theorems~\ref{thm:QuotientScheme} and~\ref{thm:block_scheme_Ppoly}. On the ideal-theoretic level, Theorem~\ref{thm:ideal_block_quotient} shows that the ideal of the block scheme is exactly the $s$-elimination ideal. It also shows that the ideal of the quotient scheme is obtained from the original defining ideal by eliminating the last $\ell-s$ variables and then applying a natural rescaling. In Section~\ref{sec:examples} we illustrate these theorems with examples.

These results have several applications. In Section~\ref{sec:multpoly_products} we analyze direct products and crested products from the viewpoint of multivariate polynomiality. In particular, we show that crested products admit natural multivariate $P$- and $Q$-polynomial structures compatible with the elimination/block orders coming from imprimitivity. We also obtain a characterization of multivariate $P$- or $Q$-polynomial association schemes that remain polynomial for every monomial order: they are precisely the direct products of univariate $P$- or $Q$-polynomial association schemes; see Theorems~\ref{thm:directproduct_anyorder} and~\ref{thm:directproduct_anyorder_Q}. Beyond these applications, we discuss further consequences for formal duality, composition series of quotient schemes, and several open problems that seem worth pursuing.

The organization of this paper is as follows. Section~\ref{sec:2} collects the preliminaries on association schemes, imprimitivity, monomial orders, Gr\"obner bases, and multivariate $P$- or $Q$-polynomial association schemes. Section~\ref{sec:eigen_characterization} gives an equivalent characterization of multivariate $P$-polynomiality in terms of the first eigenmatrix and the associated zero-dimensional ideal, and the $Q$-polynomial counterpart is developed there as well. Section~\ref{sec:MainTheorem} contains the main theorems relating imprimitivity and elimination orders, together with the induced multivariate structures on quotient and block schemes. Section~\ref{sec:examples} illustrates the theory on concrete examples. Section~\ref{sec:multpoly_products} studies direct products, crested products, and the characterization by arbitrary monomial orders. Section~\ref{sec:other_topics} discusses additional topics, including formal duality and composition-series aspects, and Section~\ref{sec:open_problems} concludes with open problems and future directions.

\section{Preliminaries}
\label{sec:2}

\subsection{Association schemes}
\label{sec:AS}
In this subsection, we recall the basic definitions for association schemes.
For background, we refer the reader to Bannai--Bannai--Ito--Tanaka~\cite{BBIT2021} and Bannai--Ito~\cite{BI1984}.
Let $X$ and $\mathcal{I}$ be finite sets
and let $\mathcal{R}$ be a surjective map from $X\times X$ to $\mathcal{I}$.
Let $M_X(\C)$ be the $\C$-algebra of complex matrices with rows and columns indexed by $X$.
The \emph{adjacency matrix} $A_i$ of $i\in \mathcal{I}$
is defined to be the matrix in $M_X(\C)$ whose $(x,y)$ entries are
\[
(A_i)_{xy}=
\begin{cases}
    1 & \text{if $\mathcal{R}(x,y)=i$,}\\   
    0 & \text{otherwise.}
\end{cases}    
\]
Clearly,
\begin{enumerate}[label=$(\mathrm{A}\arabic*)$]
    \item $\sum_{i\in \mathcal{I}} A_i =J_X$, where $J_X$ is the all-ones matrix of $M_X(\C)$. \label{AS:Hadamard}
\end{enumerate}
A triple $\mathfrak{X}=(X,\mathcal{R},\mathcal{I})$ (or simply $(X,\mathcal{R})$)
is called a \emph{commutative association scheme} if
$\mathfrak{X}$ satisfies the following conditions:
\begin{enumerate}[label=$(\mathrm{A}\arabic*)$]
    \setcounter{enumi}{1}
    \item there exists $i_0 \in \mathcal{I}$ such that $A_{i_0}=I_X$, where $I_X$ is the identity matrix of $M_X(\C)$; \label{AS:I}
    \item for each $i\in \mathcal{I}$, there exists $i'\in \mathcal{I}$ such that $A_i^T=A_{i'}$, where $A_i^T$ denotes the transpose of $A_i$; \label{AS:transpose}
    \item for each $i,j\in \mathcal{I}$,
    \[
    A_i A_j = \sum_{k\in \mathcal{I}} p^k_{ij} A_k   
    \]
    holds.
    The constant $p^k_{ij}$ is called the \emph{intersection number}; \label{AS:pijk}
    \item for $i,j,k\in \mathcal{I}$, $p^k_{ij}=p^k_{ji}$ holds; equivalently, $A_i A_j = A_j A_i$.\label{AS:commutative}
\end{enumerate}
If $|\mathcal{I}|=d+1$, then $\mathfrak{X}$ is said to have class $d$.
By relabeling the index set if necessary, one may, and often does, assume that $i_0=0$; we will occasionally use this convention below.
Moreover, if  $\mathfrak{X}$ satisfies 
\begin{enumerate}[label=$(\mathrm{A}\arabic*)$]
    \setcounter{enumi}{5}
\item for each $i\in \mathcal{I}$, $i=i'$ holds, \label{AS:symmetric}
\end{enumerate}
then $\mathfrak{X}$ is called \emph{symmetric}.
Henceforth, when we simply say ``association scheme'', we mean a commutative association scheme.
We also write $\mathfrak{X}=(X,\{A_i\}_{i\in \mathcal{I}})$ when we wish to emphasize the adjacency matrices.

Let $\mathfrak{X}_1=(X_1,\mathcal{R}_1,\mathcal{I}_1)$ and $\mathfrak{X}_2=(X_2,\mathcal{R}_2,\mathcal{I}_2)$ be association schemes.
A pair of maps $(f,g)$, where $f\colon X_1\to X_2$ and $g\colon \mathcal{I}_1\to \mathcal{I}_2$, is called a \emph{homomorphism} from $\mathfrak{X}_1$ to $\mathfrak{X}_2$ if, for any $x,x'\in X_1$,
\[
    g\bigl(\mathcal{R}_1(x,x')\bigr)=\mathcal{R}_2\bigl(f(x),f(x')\bigr)
\]
holds.
That is, the following diagram is commutative:
\[
\begin{tikzcd}
X_1\times X_1 \arrow[r, "\mathcal{R}_1"] \arrow[d, "f\times f"'] & \mathcal{I}_1 \arrow[d, "g"] \\
X_2\times X_2 \arrow[r, "\mathcal{R}_2"'] & \mathcal{I}_2
\end{tikzcd}
\]
Furthermore, if both $f$ and $g$ are bijections, then $(f,g)$ is called an \emph{isomorphism}.
If such an isomorphism $(f,g)$ exists, then $\mathfrak{X}_1$ and $\mathfrak{X}_2$ are said to be \emph{isomorphic}.

Let $\mathcal{I},\mathcal{I}'$ be finite sets of the same cardinality, and let $\varphi\colon \mathcal{I}\to \mathcal{I}'$ be a bijection.
Then, for an association scheme $\mathfrak{X}=(X,\mathcal{R},\mathcal{I})$, if we set
$\mathcal{R}':=\varphi\circ \mathcal{R}$,
then $\mathfrak{X}'=(X,\mathcal{R}',\mathcal{I}')$ is also an association scheme.
In this case, $(\mathrm{id}_X,\varphi)$ is an isomorphism from $\mathfrak{X}$ to $\mathfrak{X}'$.
We call $\mathfrak{X}'$ a \emph{relabeling} of $\mathfrak{X}$ to $\mathcal{I}'$ (with respect to $\varphi$).

Let $\mathfrak{A}=\Span_\C \{ A_i \}_{i\in \mathcal{I}}$.
By \ref{AS:pijk}, $\mathfrak{A}$ becomes a subalgebra of $M_X(\C)$.
The algebra $\mathfrak{A}$ is called the \emph{Bose--Mesner algebra} of $\mathfrak{X}$.
By \ref{AS:Hadamard}, $\{A_i\}_{i\in \mathcal{I}}$ is a basis of $\mathfrak{A}$, and we have $\dim_{\C} \mathfrak{A}=d+1$
if $\mathfrak{X}$ is of class $d$.
By \ref{AS:transpose} and~\ref{AS:commutative},
$\mathfrak{A}$ has another basis  
$\{E_j\}_{j\in \mathcal{J}}$ consisting of the primitive idempotents of $\mathfrak{A}$,
where $\mathcal{J}$ is a finite set.
Since $\{A_i\}_{i\in \mathcal{I}}$ and $\{E_j\}_{j\in \mathcal{J}}$ are bases of $\mathfrak{A}$,
$|\mathcal{I}|=|\mathcal{J}|$ holds.
By~\ref{AS:Hadamard}, $\mathfrak{A}$ is closed under entrywise multiplication; this product is denoted by $\circ$ and called the \emph{Hadamard product}.
Then $\{E_j\}_{j\in \mathcal{J}}$ has the following properties:
\begin{enumerate}[label=$(\mathrm{E}\arabic*)$]
    \item $\sum_{j\in \mathcal{J}}E_j=I_X$; \label{ASE:sum}
    \item there exists $j_0 \in \mathcal{J}$ such that $E_{j_0}=\frac{1}{|X|}J_X$; \label{ASE:J}
    \item for each $j\in \mathcal{J}$, there exists $j^T\in \mathcal{J}$ such that $E_j^T=E_{j^T}$; \label{ASE:transpose}
    \item for each $i,j\in \mathcal{J}$,
    \[
    (|X| E_i)\circ (|X|E_j)=\sum_{k\in \mathcal{J}}q^k_{ij}|X|E_k
    \]
    holds.
    The constant $q^k_{ij}$ is called the \emph{Krein number} of $\mathfrak{X}$; \label{ASE:qijk}
\end{enumerate}

The entries of the \emph{first eigenmatrix} $P:=(P_i(j))_{j\in \mathcal{J}, i\in \mathcal{I}}$
and the \emph{second eigenmatrix} $Q:=(Q_j(i))_{i\in \mathcal{I}, j\in \mathcal{J}}$
are defined by
\[
A_i = \sum_{j\in \mathcal{J}}P_i(j)E_j \qquad
\text{and}\qquad
|X| E_j = \sum_{i\in \mathcal{I}}Q_j(i)A_i,
\]
respectively.

For each $i\in \mathcal{I}$, the intersection number $p^{i_0}_{i\,i^T}$ is called the \emph{valency} of $i$, and is denoted by $k_i$.
By the definition of valency, $k_i$ is equal to the out-degree of each vertex in the digraph whose adjacency matrix is $A_i$.
Then, by the Perron--Frobenius theorem,
\begin{equation}\label{eq:P_ineq}
    \text{$P_i(j_0)=k_i$ and $|P_i(j)|\le k_i$ for any $j\in \mathcal{J}$}
\end{equation}
hold. On the other hand, for each $j\in \mathcal{J}$, 
the rank of $E_j$ is called the \emph{multiplicity} of $j$, and is denoted by $m_j$.
This is equal to the Krein number $q^{j_0}_{j\,j^T}$.
Furthermore,
\begin{equation}\label{eq:Q_ineq}
    \text{$Q_j(i_0)=m_j$ and $|Q_j(i)|\le m_j$ for any $i\in \mathcal{I}$}
\end{equation}
are known to hold. 

We conclude this subsection with a well-known lemma, which we use later.
\begin{lem}
\label{lem:intersection_krein_relation}
Let $\mathfrak{X}=(X,\mathcal{R},\mathcal{I})$ be an association scheme.
Then the following equalities hold for the intersection numbers and Krein numbers of $\mathfrak{X}$:
\begin{enumerate}[label=$(\roman*)$]
    \item \label{item:intersection_numbers} $k_k p^k_{ij}=k_j p^j_{i^T k}=k_i p^i_{k j^T}$ for $i,j,k\in \mathcal{I}$;
    \item \label{item:krein_numbers} $m_k q^k_{ij}=m_j q^j_{i^T k}=m_i q^i_{k j^T}$ for $i,j,k\in \mathcal{J}$.
\end{enumerate}
\end{lem}

\subsection{Imprimitive association schemes}
\label{sec:ImprimitiveAS}
The material in this subsection follows Bannai--Ito~\cite{BI1984}, Zieschang~\cite{Zieschang}, and Curtin~\cite{curtin2008}.
Let $\mathfrak{X}=(X,\mathcal{R},\mathcal{I})$ be a commutative association scheme.
For $i,i'\in\mathcal{I}$, define the \emph{complex product} $ii'$ on $\mathcal{I}$ by
$ii':=\{k\in\mathcal{I}\mid p^k_{ii'}>0\}$.
Further, for subsets $\mathcal{A},\mathcal{B}\subset\mathcal{I}$, define
$\mathcal{A}\mathcal{B}:=\bigcup_{i\in\mathcal{A},\,i'\in\mathcal{B}} ii'$ and 
$\mathcal{A}^T:=\{i^T\mid i\in\mathcal{A}\}$.

\begin{df}
\label{df:closed_subset_zieschang}
A nonempty subset $\mathcal{C}\subset\mathcal{I}$ is called a \emph{closed subset} if
$\mathcal{C}^T\mathcal{C}\subset \mathcal{C}$, that is, if
$ i,i'\in\mathcal{C}$ then $i^T i'\subset\mathcal{C}$ holds.
\end{df}
The subsets $\{i_0\}$ and $\mathcal{I}$ are trivially closed subsets of $\mathcal{I}$. 

\begin{lem}[{cf. \cite[Lemma 2.1.4]{Zieschang}}]\label{lem:closed_subset_equiv}
For $\emptyset\neq\mathcal{C}\subset\mathcal{I}$, the following conditions are equivalent:
\begin{enumerate}[label=$(\roman*)$]
\item \label{item:closed_subset_equiv_i} $\mathcal{C}$ is a closed subset.
\item \label{item:closed_subset_equiv_ii} $\mathcal{R}^{-1}(\mathcal{C})\subset X\times X$ is an equivalence relation on $X$.
\item \label{item:closed_subset_equiv_iii} $\{(i,i')\mid i^T i'\cap \mathcal{C} \neq \emptyset \}\subset \mathcal{I}\times \mathcal{I}$ is an equivalence relation on $\mathcal{I}$.
\end{enumerate}
\end{lem}
Let $\mathcal{C}$ be a closed subset, and denote by $\sim$ the equivalence relation on $X$ determined by Lemma~\ref{lem:closed_subset_equiv}~\ref{item:closed_subset_equiv_ii}.
For $x\in X$, let $x\mathcal{C}:=\{y\in X\mid x\sim y\}$ be the equivalence class of $x$ with respect to $\sim$, and set $X/\mathcal{C}:=\{x\mathcal{C} \mid x\in X\}$.
By Lemma~\ref{lem:closed_subset_equiv}~\ref{item:closed_subset_equiv_iii}, $\mathcal{C}$ also defines an equivalence relation on $\mathcal{I}$.
For $i,i'\in \mathcal{I}$, we denote this relation by $i\equiv i'$.
For $i\in \mathcal{I}$, let $i\mathcal{C}:=\{i'\in \mathcal{I}\mid i\equiv i'\}$ be the equivalence class of $i$ with respect to $\equiv$, and set $\mathcal{I}/\mathcal{C}:=\{i\mathcal{C} \mid i\in \mathcal{I}\}$.

Next, 
in the same manner as Definition~\ref{df:closed_subset_zieschang},
we introduce on $\mathcal{J}$ the complex product defined by the Krein numbers.
For $j,j'\in\mathcal{J}$, define the complex product (Krein product) on $\mathcal{J}$ with respect to the Hadamard product $\circ$ by
$j\circ j':=\{k\in\mathcal{J}\mid q^k_{jj'}>0\}$.
For subsets $\mathcal{A},\mathcal{B}\subset\mathcal{J}$, define
$\mathcal{A}\circ\mathcal{B}:=\bigcup_{j'\in\mathcal{A},\,j\in\mathcal{B}} j'\circ j$.

\begin{df}
\label{df:dual_closed_subset_curtin}
A nonempty subset $\mathcal{C}^\ast\subset\mathcal{J}$ is called a \emph{closed subset with respect to the Hadamard product $\circ$} if
$(\mathcal{C}^\ast)^T\circ \mathcal{C}^\ast\subset \mathcal{C}^\ast$ holds.
\end{df}

\begin{lem}[{cf. \cite[Lemma~3.2]{curtin2008}}]\label{lem:dual_closed_subset_equiv}
For $\mathcal{C}^\ast\subset\mathcal{J}$ with $\mathcal{C}^\ast\neq\emptyset$, the following conditions are equivalent: 
\begin{enumerate}[label=$(\roman*)$]
\item $\mathcal{C}^\ast$ is a closed subset with respect to the Hadamard product $\circ$.
\item $\mathrm{Span}_\C\{E_j\mid j\in\mathcal{C}^\ast\}$ is a subalgebra with respect to the Hadamard product $\circ$.
\end{enumerate}
\end{lem}
A closed subset $\mathcal{C}^\ast$ also determines an equivalence relation on $X$.
For $x,y\in X$, define $x\sim^\ast y$ if and only if
for every $j\in\mathcal{C}^\ast$ and every $w\in X$, one has $(E_j)_{wx}=(E_j)_{wy}$.
Then $\sim^\ast$ becomes an equivalence relation on $X$.
Moreover, $\mathcal{C}^\ast$ also determines an equivalence relation on $\mathcal{J}$. For example,
\[
i\equiv^\ast j
\ \Longleftrightarrow\
q^{\,j}_{i h}\neq 0\ \text{for some}\ h\in\mathcal{C}^\ast
\]
defines an equivalence relation $\equiv^\ast$ on $\mathcal{J}$.
For $j\in \mathcal{J}$, let $j\mathcal{C}^\ast:=\{j'\in \mathcal{J}\mid j\equiv^\ast j'\}$ be the equivalence class of $j$ with respect to $\equiv^\ast$, and set $\mathcal{J}/\mathcal{C}^\ast:=\{j\mathcal{C}^\ast \mid j\in \mathcal{J}\}$.

Since $\mathcal{R}^{-1}(\mathcal{C})$ for a closed subset $\mathcal{C}$ is an equivalence relation on $X$, the matrix $\sum_{i\in \mathcal{C}}A_i$ becomes $I_q\otimes J_p$ after simultaneous permutations of rows and columns, where $p,q$ are positive integers with $pq=|X|$. 
Then $q=|X/\mathcal{C}|$ and $p=|x\mathcal{C}|=\sum_{i\in \mathcal{C}} k_i$ for any $x\in X$.
Also,
\[
M:=\frac{1}{p} \sum_{i\in \mathcal{C}}A_i
\]
is an idempotent; hence there exists a unique subset $\mathcal{C}^\ast \subset \mathcal{J}$ such that
\[
M=\sum_{j\in \mathcal{C}^\ast} E_j.
\]
This $\mathcal{C}^\ast$ is called the \emph{dual closed subset} corresponding to $\mathcal{C}$.
In this case, $\mathcal{C}^\ast$ is a closed subset with respect to $\circ$ in the sense of Definition~\ref{df:dual_closed_subset_curtin}.
Moreover,
\[
q=\mathrm{rank}(M)=\sum_{j\in\mathcal{C}^\ast} \mathrm{rank}(E_j)=\sum_{j\in\mathcal{C}^\ast} m_j =|x_0 \mathcal{C}^\ast|
\quad\text{and}\quad
p=|X/\mathcal{C}^\ast|
\]
also follow.
Furthermore, for any closed subset $\mathcal{C}^\ast$ of $\mathcal{J}$,
by tracing the above argument in reverse,
we obtain a closed subset $\mathcal{C}$ of $\mathcal{I}$ in the sense of Definition~\ref{df:closed_subset_zieschang} corresponding to $\mathcal{C}^\ast$.
This $\mathcal{C}$ is called the \emph{dual closed subset} corresponding to $\mathcal{C}^\ast$.
In this case, the equivalence relations on $X$ determined by $\sim$ and $\sim^\ast$ coincide.

\begin{lem}[cf.~\cite{BI1984,curtin2008,Zieschang}]
\label{lem:closed_subset}
For a closed subset $\mathcal{C}\subset \mathcal{I}$ of a commutative association scheme $\mathfrak{X}=(X,\mathcal{R},\mathcal{I})$ and the dual closed subset $\mathcal{C}^\ast\subset \mathcal{J}$ corresponding to $\mathcal{C}$
(or equivalently, for a closed subset $\mathcal{C}^\ast\subset \mathcal{J}$ and the dual closed subset $\mathcal{C}\subset \mathcal{I}$ corresponding to $\mathcal{C}^\ast$),
the following statements hold:
\begin{enumerate}[label=$(\roman*)$]
    \item \label{lem:closed_subset:mult} For $i\in \mathcal{C}$, $j\in \mathcal{I}$ and $j' \in j \mathcal{C}$, we have $ij' \subset j\mathcal{C}$.
    \item \label{lem:dual_closed_subset:mult} For $i\in \mathcal{C}^\ast$, $j\in \mathcal{J}$ and $j' \in j \mathcal{C}^\ast$, we have $i\circ j' \subset j\mathcal{C}^\ast$.
    \item \label{lem:closed_subset:k} For $i\in \mathcal{C}$ and $j\in \mathcal{C}^\ast$, we have $P_i(j)=k_i$.
    \item \label{lem:closed_subset:m} For $i\in\mathcal{C}$ and $j\in\mathcal{C}^\ast$, we have $Q_j(i)=m_j$.
\end{enumerate}
\end{lem}

A commutative association scheme $\mathfrak{X}=(X,\mathcal{R},\mathcal{I})$ is called \emph{imprimitive} if there exists a closed subset $\mathcal{C}\subset \mathcal{I}$ satisfying $\{i_0\}\subsetneq \mathcal{C}\subsetneq \mathcal{I}$.
Henceforth in this section, we assume that $\mathfrak{X}$ is imprimitive.

Fix $x_0\in X$, and let $x_0\mathcal{C}$ denote the equivalence class of $x_0$ with respect to $\sim$.
For any $x,y\in x_0\mathcal{C}$, we have $\mathcal{R}(x,y)\in \mathcal{C}$.
Thus, we can define the restriction map $\mathcal{R}|_{x_0\mathcal{C}}\colon x_0\mathcal{C}\times x_0\mathcal{C} \to \mathcal{C}$ of $\mathcal{R}$ on $x_0\mathcal{C}$.
Then $\mathfrak{X}_{x_0\mathcal{C}}=( x_0\mathcal{C}, \mathcal{R}|_{x_0\mathcal{C}},\mathcal{C})$ becomes an association scheme, 
which we call the \emph{block scheme} of $\mathfrak{X}$ at $x_0\mathcal{C}$.
Note that from the definition, considering the natural inclusion maps $\iota_X\colon x_0\mathcal{C} \to X$ and $\iota_\mathcal{I}\colon \mathcal{C}\to \mathcal{I}$, 
the pair $(\iota_X,\iota_\mathcal{I})$ gives an injective homomorphism from $\mathfrak{X}_{x_0\mathcal{C}}$ to $\mathfrak{X}$.
Let $\mathfrak{A}_{x_0\mathcal{C}}$ be the Bose--Mesner algebra of $\mathfrak{X}_{x_0\mathcal{C}}$.
Consider the subalgebra $\mathfrak{A}_{\mathcal{C}}:=\mathrm{span}_\C\{A_i\mid i\in\mathcal{C}\}$ of $\mathfrak{A}$ corresponding to the closed subset $\mathcal{C}$. 
It is known that $\mathfrak{A}_{\mathcal{C}}$ is isomorphic to $\mathfrak{A}_{x_0\mathcal{C}}$ as an algebra for every $x_0\in X$.
On the other hand, note that, for $x,y\in X$, $\mathfrak{X}_{x\mathcal{C}}$ and $\mathfrak{X}_{y\mathcal{C}}$ are not necessarily isomorphic as association schemes.
Let $j\mathcal{C}^\ast$ denote the equivalence class of $j\in\mathcal{J}$ with respect to $\equiv^\ast$, 
and write $\mathcal{J}/\mathcal{C}^\ast:=\{j\mathcal{C}^\ast\mid j\in\mathcal{J}\}$.
For an equivalence class $\bm{j}\in\mathcal{J}/\mathcal{C}^\ast$ under $\equiv^\ast$, let
$E_{\bm{j}}:=\sum_{j\in \bm{j}}E_{j}$.
Then the primitive idempotents $\{E_{\bm{j}}\mid \bm{j}\in\mathcal{J}/\mathcal{C}^\ast\}$ give a basis of $\mathfrak{A}_{\mathcal{C}}$.
That is, the index set of the basis consisting of primitive idempotents of $\mathfrak{A}_{x_0\mathcal{C}}$ is also $\mathcal{J}/\mathcal{C}^\ast$.

For $X/\mathcal{C}$, we define
\[
\mathcal{R}/\mathcal{C} \colon X/\mathcal{C} \times X/\mathcal{C} \to \mathcal{I}/\mathcal{C},\quad
(x\mathcal{C},y\mathcal{C}) \mapsto \mathcal{R}(x,y)\mathcal{C}.
\]
Then $\mathcal{R}/\mathcal{C}$ is well-defined, and 
$\mathfrak{X}/\mathcal{C}=( X/\mathcal{C},\mathcal{R}/\mathcal{C},\mathcal{I}/\mathcal{C})$ becomes an association scheme.
This is called the \emph{quotient scheme} of $\mathfrak{X}$ by $\mathcal{C}$.
Note that from the definition, considering the natural projections $\pi_X\colon X \to X/\mathcal{C}$ and $\pi_\mathcal{I}\colon \mathcal{I}\to \mathcal{I}/\mathcal{C}$, 
the pair $(\pi_X,\pi_\mathcal{I})$ gives a surjective homomorphism from $\mathfrak{X}$ to $\mathfrak{X}/\mathcal{C}$.
We denote the Bose--Mesner algebra of $\mathfrak{X}/\mathcal{C}$ by $\mathfrak{A}/\mathcal{C}$.
It is known that $\mathfrak{A}/\mathcal{C}$ is canonically isomorphic to 
the two-sided ideal $\mathfrak{A} M := \{A M \mid A \in \mathfrak{A}\}$ of $\mathfrak{A}$, 
and for $\bm{i}\in \mathcal{I}/\mathcal{C}$, each adjacency matrix $D_{\bm{i}}$ in $\mathfrak{A}/\mathcal{C}$ corresponds to 
\begin{equation}
    \label{eq:MAM}
    \frac{1}{p}\sum_{i\in \bm{i}} A_i M =\frac{1}{p}\sum_{i\in \bm{i}} A_i
\end{equation}
in $\mathfrak{A} M$.
For each $j\in \mathcal{C}^\ast$,
there exists a $q \times q$ matrix $F_j$ such that
$|X| E_j = (q F_j) \otimes J_p$.
A matrix $F_j$ is an idempotent, and $\{F_j\}_{j\in \mathcal{C}^\ast}$ forms a basis of $\mathfrak{A}/\mathcal{C}$ consisting of primitive idempotents.

The first and second eigenmatrices $P^{\mathrm{bl}}$ and $Q^{\mathrm{bl}}$ of the block scheme $\mathfrak{X}_{x_0\mathcal{C}}$ 
and the first and second eigenmatrices $P^{\mathrm{qt}}$ and $Q^{\mathrm{qt}}$ of the quotient scheme $\mathfrak{X}/\mathcal{C}$
can be expressed using the eigenmatrices $P$ and $Q$ of the original $\mathfrak{X}$ as follows:
The first eigenmatrix $P^{\mathrm{bl}}$ and the second eigenmatrix $Q^{\mathrm{bl}}$ of the block scheme $\mathfrak{X}_{x_0\mathcal{C}}$ 
are given, for $i\in\mathcal{C}$, $\bm{j}\in \mathcal{J}/\mathcal{C}^\ast$ and a representative $j\in\bm{j}$, by 
\begin{align}
P^{\mathrm{bl}}_i(\bm{j}) &= P_i(j), \label{eq:block_scheme_P}\\
Q^{\mathrm{bl}}_{\bm{j}}(i) &= \frac{1}{q}\sum_{j\in \bm{j}} Q_{j}(i). \label{eq:block_scheme_Q}
\end{align}
In particular, the right-hand side of \eqref{eq:block_scheme_P} does not depend on the choice of representative $j\in \bm{j}$;
The first eigenmatrix $P^{\mathrm{qt}}$ and the second eigenmatrix $Q^{\mathrm{qt}}$ of the quotient scheme $\mathfrak{X}/\mathcal{C}$ 
are given, for $\bm{i}\in \mathcal{I}/\mathcal{C}$, a representative $i\in\bm{i}$ and $j\in \mathcal{C}^\ast$, by
\begin{align}
P^{\mathrm{qt}}_{\bm{i}}(j) &= \frac{1}{p}\sum_{i\in \bm{i}} P_{i}(j), \label{eq:quotient_P}\\
    Q^{\mathrm{qt}}_{j}(\bm{i}) &= Q_j(i) \label{eq:quotient_Q}.
\end{align}
In particular, the right-hand side of \eqref{eq:quotient_Q} does not depend on the choice of a representative $i\in \bm{i}$.
These are well-known facts.
For example, \eqref{eq:block_scheme_P} and~\eqref{eq:block_scheme_Q} can be found in Section 2.3 of van Dam--Martin--Muzychuk~\cite{vanDamMartinMuzychuk2013}.
For \eqref{eq:quotient_Q} and \eqref{eq:quotient_P}, we refer to \cite[Section 2.9]{BI1984}. 

Since $\mathfrak{A}M$ is a two-sided ideal of $\mathfrak{A}$, the projection $\mathfrak{A}\to \mathfrak{A}M$ is defined.
Then, for any $i\in \mathcal{I}$, the following equality holds:
\begin{equation}
    \label{eq:quotient_scheme_A}
    A_i M=\frac{k_i}{k_{i\mathcal{C}}} \cdot \frac{1}{p}\sum_{i'\in i\mathcal{C}} A_{i'} M=\frac{k_i}{k_{i\mathcal{C}}} \cdot \frac{1}{p}\sum_{i'\in i\mathcal{C}} A_{i'}, 
\end{equation}
where $k_i$ is the valency of $i$ with respect to $\mathfrak{X}$, and $k_{i\mathcal{C}}$ is the valency of $i\mathcal{C}$ with respect to $\mathfrak{X}/\mathcal{C}$.
Moreover, from \eqref{eq:quotient_scheme_A}, the equality 
\begin{equation}
    \label{eq:quotient_scheme_P_and_valency}
    \frac{P_i(j)}{k_i}=\frac{P^{\mathrm{qt}}_{i\mathcal{C}}(j)}{k_{i\mathcal{C}}}
\end{equation}
holds for any $i\in \mathcal{I}$ and $j\in \mathcal{C}^\ast$. 

\subsection{Monomial orders, Gr\"obner bases and elimination orders}
\label{sec:monomial_orders}

In this subsection, we recall the fundamentals of monomial orders and Gr\"obner bases. 
For further details, see, for example, Cox--Little--O'Shea~\cite{Cox}. 

Let $\N^\ell:=\{(n_1,n_2,\ldots ,n_\ell) \mid \text{$n_i$ are nonnegative integers}\}$. We use the following notation throughout this paper: 
\begin{itemize}
    \item $o := (0,0,\ldots ,0)\in \N^\ell$; 
    \item for $i=1,2,\ldots ,\ell$, let $\epsilon_i\in \N^\ell$ denote the $i$-th unit vector, i.e., the vector in which the $i$-th entry is $1$ and the remaining entries are $0$;
    \item for $\alpha=(n_1,n_2, \ldots ,n_\ell)\in \N^\ell$, let $|\alpha|=\sum^\ell_{i=1}n_i$.
\end{itemize}

\begin{df}
    \label{df:monomialorder}
    A total order $\le$ on $\N^\ell$ is called a \emph{monomial order} if 
    \begin{enumerate}[label=$(\roman*)$]
        \item $\alpha \ge o$ for any $\alpha\in \N^\ell$; and
        \item $\alpha+\gamma \le \beta + \gamma$ holds for any $\alpha,\beta,\gamma \in \N^\ell$ with $\alpha \le \beta$. 
    \end{enumerate}
\end{df}
Let $\C[\bm{x}]=\C[x_1,\ldots,x_\ell]$ be a polynomial ring in $\ell$ variables over $\C$. 
For $\alpha=(n_1,n_2,\ldots , n_\ell)\in \N^\ell$, we write the monomial $x_1^{n_1}x_2^{n_2}\cdots x_\ell^{n_\ell}$ of $\C[\bm{x}]$ by $\bm{x}^\alpha$. 
Namely, we can identify $\N^\ell$ with the set of monomials of $\C[\bm{x}]$. Then $\alpha$ is called the \emph{multidegree} of $\bm{x}^\alpha$. 
A monomial order is usually defined on the set of monomials of $\C[x_1,\ldots,x_\ell]$, but by a mild abuse of notation we identify monomials with their exponent vectors and regard it as a total order on $\N^\ell$. 

We recall some typical examples of monomial orders. Let $\alpha, \beta \in \N^\ell$ with $\alpha \neq \beta$.
\begin{itemize}
    \item We define $\alpha \le_{\mathrm{lex}} \beta$ if the leftmost nonzero entry of $\alpha-\beta\in \Z^\ell$ is negative. 
    This is a monomial order on $\N^\ell$ and $\le_{\mathrm{lex}}$ is called the \emph{lexicographic} (or \emph{lex}) order. 
    \item We define $\alpha \le_{\mathrm{grlex}} \beta$ if 
\[    |\alpha|<|\beta|\  \text{or}\   (|\alpha|=|\beta|\   \text{and}\  \alpha \le_{\mathrm{lex}} \beta).   \]
    This $\le_{\mathrm{grlex}}$ is called the \emph{graded lexicographic} (or \emph{grlex}) order.
\end{itemize}

Fix a monomial order $\le$ on $\N^\ell$. For each nonzero polynomial $f=\sum_{\alpha \in \N^\ell} c_\alpha \bm{x}^\alpha \in \C[\bm{x}]$, 
where $c_\alpha=0$ for all but finitely many $\alpha \in \N^\ell$, 
the multidegree (resp. leading term) of $f$, denoted by $\MD(f)$ (resp. $\LT(f)$), is defined as follows: 
\[\MD(f) := \max\{ \alpha \in \N^\ell \mid   c_\alpha \neq 0   \} \;\;\text{and}\;\; \LT(f) := c_{\MD(f)} \bm{x}^{\MD(f)},\]
where the maximum is taken with respect to $\le$.

\begin{df}
Let $I\subset \C[\bm{x}]$ be an ideal. Fix a monomial order on $\C[\bm{x}]$. 
A subset $\mathcal{G}=\{g_1,\dots,g_m\} \subset I$ is called a \emph{Gr\"obner basis} of $I$ with respect to $\le$ 
if the following monomial ideals are equal: 
\[ \langle  \LT(g_1), \LT(g_2),\ldots , \LT(g_m) \rangle =  \langle \LT(f) \mid f\in I  \rangle .\]
Equivalently, $\mathcal{G}$ is a Gr\"obner basis of $I$ if and only if
the leading term of any element of $I$ is divisible by one of the $\LT(g_i)$.
\end{df}
It is well known that every Gr\"obner basis generates the ideal. 
We set the following two subsets of $\N^\ell$: 
\[\MD(I):=\{\MD (f) \mid f\in I\setminus \{0\}\} \;\text{ and }\; \MD(\mathcal{G}):=\{\MD (g) \mid g\in \mathcal{G}\}.\]


\begin{df}\label{def:elim_block}
Let $s$ be an integer with $1 \le s < \ell$. 
We identify $\N^\ell$ with $\N^s \times \N^{\ell-s}$.
\begin{itemize}
    \item We say that a monomial order $\le$ on $\N^\ell$ is \emph{of $s$-elimination type} (see \cite[Chapter 3 \S 1 Exercise 5]{Cox}) 
    if $(\alpha,\beta_1) > (o,\beta_2)$ holds for any $\alpha \in \N^s$ with $\alpha \neq o$ and $\beta_1,\beta_2 \in \N^{\ell - s}$. 
    \item We say that a monomial order $\le$ on $\N^\ell$ is \emph{of $s$-block type} 
    if $(\alpha_1,\beta_1)>(\alpha_2,\beta_2)$ holds whenever $(\alpha_1,o)>(\alpha_2,o)$ for any $\alpha_1,\alpha_2 \in \N^s$ and $\beta_1,\beta_2 \in \N^{\ell - s}$. 
\end{itemize}
By definition, every monomial order of $s$-block type is of $s$-elimination type, but the converse need not hold (see Remark~\ref{rem:elim_block}). 
\end{df}

Typical examples are the lexicographic order for $s$-elimination type and block orders for $s$-block type.
Monomial orders of $s$-elimination type are the standard orders used in the elimination theorem for ideals; see, for example, \cite[Chapter~3]{Cox}. 


We observe the following, which we will use repeatedly; we omit the proof.
\begin{lem}\label{lem:induced_monomial_order}
Let $n,m$ be positive integers with $m\le n$ and let $\le$ be a monomial order on $\N^n$.
For $1\le i_1<\cdots <i_m \le n$,
define a map $\iota:\N^m\to \N^n$ by
\[\iota(a_1,\ldots,a_m)=\sum_{r=1}^m a_r \epsilon_{i_r}.\]
Then the order $\le_\iota$ on $\N^m$ defined by $\alpha \le_\iota \beta$ if  $\iota(\alpha)\le \iota(\beta)$ is a monomial order on $\N^m$.
\end{lem}

\subsection{Multivariate $P$- and $Q$-polynomial association schemes}
\label{sec:Def}

In this subsection, we recall the definitions of multivariate $P$- and $Q$-polynomial association schemes.

A symmetric association scheme $\mathfrak{X}=(X,\mathcal{R}, \mathcal{I})$ of class $d$
is called \emph{$P$-polynomial}
if it satisfies the following conditions:
$\mathcal{I}=\{0,1,\ldots ,d\}$
and there exists a univariate polynomial $v_i$ of degree $i$ such that $A_i=v_i(A_1)$
for each $i\in \{0,1,\ldots ,d\}$. 
Similarly, a symmetric association scheme $\mathfrak{X}=(X,\mathcal{R}, \mathcal{I})$ of class $d$
is called \emph{$Q$-polynomial}
if it satisfies the following conditions:
$\mathcal{J}=\{0,1,\ldots ,d\}$
and there exists a univariate polynomial $v^\ast_j$ of degree $j$ such that $|X|E_j=v^\ast _j(|X| E_1)$
(under the Hadamard product)
for each $j\in \{0,1,\ldots ,d\}$. 
The following is a well-known characterization of the $P$-polynomial property: the three-term recurrence relation 
\[
    A_1 A_i = p^{i-1}_{1i} A_{i-1} +p^{i}_{1i} A_i +p^{i+1}_{1i} A_{i+1}
\]
holds for each $i \in \{0,1,\ldots,d\}$, where $p^{-1}_{10}A_{-1}$ and $p^{d+1}_{1d}A_{d+1}$ are regarded as zero. 
Similarly, the $Q$-polynomial property is characterized by the following three-term recurrence relation: 
\[
    (|X| E_1) \circ (|X| E_i) = q^{i-1}_{1i}|X| E_{i-1} +q^{i}_{1i} |X| E_i +q^{i+1}_{1i} |X| E_{i+1}
\]
holds for each $i\in \{0,1,\ldots ,d\}$, where $q^{-1}_{10}|X|E_{-1}$ and $q^{d+1}_{1d}|X|E_{d+1}$ are regarded as zero.

For non-symmetric commutative association schemes, one can likewise generalize the above conditions to define non-symmetric $P$- or $Q$-polynomial association schemes 
(see, e.g., Damerell~\cite{Damerell1981}, Lam~\cite{Lam1980}, Leonard~\cite{Leonard1991} and Munemasa~\cite{Munemasa1991}). 
However, for non-symmetric $P$- or $Q$-polynomial association schemes, the above three-term recurrence relations do not hold in general.

Multivariate $P$- and $Q$-polynomial association schemes were introduced by Bannai--Kurihara--Zhao--Zhu~\cite{BKZZ}.
\begin{df}
    \label{df:abPpoly}
    Let $\mathcal{D}\subset \N^\ell$ containing $\epsilon_1,\epsilon_2,\ldots ,\epsilon_\ell$, and let $\le$ be a monomial order on $\N^\ell$. 
    A commutative association scheme $\mathfrak{X}=(X,\mathcal{R},\mathcal{I})$ is called \emph{$\ell$-variate $P$-polynomial}
    on the domain $\mathcal{D}$ with respect to $\le$
    if the following three conditions are satisfied:
    \begin{enumerate}[label=$(\roman*)$]
        \item \label{item:abPpoly_i} If $(n_1,n_2,\ldots ,n_\ell)\in \mathcal{D}$
        and $0\le m_i \leq n_i$ for $i=1,2,\ldots ,\ell$,
        then $(m_1,m_2,\ldots ,m_\ell)\in \mathcal{D}$;
        \item \label{item:abPpoly_ii} there exists a relabeling of the adjacency matrices of $\mathfrak{X}$:
        \[
        \{A_i\}_{i\in \mathcal{I}} = \{A_\alpha \}_{\alpha\in \mathcal{D}},   
        \]
        such that, for $\alpha\in \mathcal{D}$,
        \begin{equation*}
            A_\alpha=v_\alpha(A_{\epsilon_1},A_{\epsilon_2},\ldots ,A_{\epsilon_\ell}),    
        \end{equation*}
        where $v_\alpha(\bm{x})$ is an $\ell$-variate polynomial
        of multidegree $\alpha$ with respect to $\le$
        and all monomials $\bm{x}^\beta$ in $v_\alpha(\bm{x})$ satisfy $\beta \in \mathcal{D}$;
        \item \label{item:abPpoly_iii} for $i=1,2,\ldots ,\ell$ and $\alpha=(n_1,n_2,\ldots,n_\ell)\in \mathcal{D}$,
        the product
        $A_{\epsilon_i}\cdot A_{\epsilon_1}^{n_1}A_{\epsilon_2}^{n_2}\cdots A_{\epsilon_\ell}^{n_\ell}$
        is a linear combination of
        \[
         \{A_{\epsilon_1}^{m_1}A_{\epsilon_2}^{m_2}\cdots A_{\epsilon_\ell}^{m_\ell} \mid \beta =(m_1,m_2,\ldots ,m_\ell)\in \mathcal{D},\  \beta \le  \alpha +\epsilon_i\}.   
        \]
    \end{enumerate}

\end{df}

Henceforth, we use the notation $\bm{A}$ for $(A_{\epsilon_1},A_{\epsilon_2},\ldots ,A_{\epsilon_\ell})$. 
Also, for $\alpha=(n_1,n_2,\ldots , n_\ell)\in \N^\ell$, 
we write $A_{\epsilon_1}^{n_1} A_{\epsilon_2}^{n_2} \cdots A_{\epsilon_\ell}^{n_\ell}$ by $\bm{A}^\alpha$.


Multivariate $Q$-polynomial association schemes can also be defined
as in Definition~\ref{df:abPpoly}.
\begin{df}
    \label{df:abQpoly}
    Let $\mathcal{D}^\ast \subset \N^{\ell^\ast}$ containing $\epsilon_1,\epsilon_2,\ldots ,\epsilon_{\ell^\ast}$ 
    and let $\le$ be a monomial order on $\N^{\ell^\ast}$.
    A commutative association scheme $\mathfrak{X}=(X,\mathcal{R})$
    with the primitive idempotents $\{E_j\}_{j\in \mathcal{J}}$ is called \emph{$\ell^\ast$-variate $Q$-polynomial}
    on the domain $\mathcal{D}^\ast$ with respect to $\le$
    if the following three conditions are satisfied:
    \begin{enumerate}[label=$(\roman*)$]
        \item \label{item:abQpoly_i} if $(n_1,n_2,\ldots ,n_{\ell^\ast})\in \mathcal{D}^\ast$
        and $0\le m_i \leq n_i$ for $i=1,2,\ldots ,\ell^\ast$,
        then $(m_1,m_2,\ldots ,m_{\ell^\ast})\in \mathcal{D}^\ast$;
        \item \label{item:abQpoly_ii} there exists a relabeling of the primitive idempotents of $\mathfrak{X}$:
        \[
        \{E_j\}_{j\in \mathcal{J}} = \{E_\alpha \}_{\alpha\in \mathcal{D}^\ast}, 
        \]
        such that, for $\alpha\in \mathcal{D}^\ast$,
        \[
            |X| E_\alpha=v^\ast_\alpha(|X| E_{\epsilon_1},|X| E_{\epsilon_2},\ldots ,|X| E_{\epsilon_{\ell^\ast}})
            \ \text{(under the Hadamard product),}
        \]
        where $v^\ast_\alpha(\bm{x})$ is an $\ell^\ast$-variate polynomial
        of multidegree $\alpha$ with respect to $\le$
        and all monomials $\bm{x}^\beta$ in $v^\ast_\alpha(\bm{x})$ satisfy $\beta \in \mathcal{D}^\ast$;
        \item \label{item:abQpoly_iii} for $i=1,2,\ldots ,\ell^\ast$ and $\alpha=(n_1,n_2,\ldots,n_{\ell^\ast})\in \mathcal{D}^\ast$,
        the product
        $E_{\epsilon_i}\circ E_{\epsilon_1}^{\circ n_1}\circ E_{\epsilon_2}^{\circ n_2}\circ \cdots  \circ E_{\epsilon_{\ell^\ast}}^{\circ n_{\ell^\ast}}$
        is a linear combination of
        \[
            \{E_{\epsilon_1}^{\circ m_1}\circ E_{\epsilon_2}^{\circ m_2}\circ \cdots  \circ E_{\epsilon_{\ell^\ast}}^{\circ m_{\ell^\ast}} \mid \beta=(m_1,m_2,\ldots ,m_{\ell^\ast})\in \mathcal{D}^\ast,\  \beta\le  \alpha+\epsilon_i\}.   
        \]
    \end{enumerate}

\end{df}

\begin{rem}
The following remarks, also mentioned in \cite{BKZZ}, are summarized below.
The following statement concerns only the multivariate $P$-polynomial association scheme, but similar results hold for the multivariate $Q$-polynomial association scheme.
For details, we refer the reader to \cite{BKZZ}.
\begin{enumerate}[label=$(\roman*)$]
    \item $\mathcal{D}$ must contain $o=(0,\ldots ,0)$
    and $v_{o}(\bm{x})=1$ (i.e., $A_{o}=I_X$) holds.
    Moreover, $v_{\epsilon_i}(\bm{x})=x_i$ (i.e., 
    $v_{\epsilon_i}(\bm{A})=A_{\epsilon_i}$) holds for $i=1,2,\ldots ,\ell$.
    \item Every commutative association scheme
$\mathfrak{X}=(X,\mathcal{R})$ of class $d$ is regarded as a $d$-variate $P$-polynomial association scheme on the domain $\mathcal{D}=\{o,\epsilon_1,\epsilon_2,\ldots ,\epsilon_d\}\subset \N^d$ with respect to $\le_{\mathrm{grlex}}$.
Therefore, we usually consider the ``essential'' variate
for $\mathfrak{X}$, i.e., we consider
$\ell =\min \{\ell' \mid \text{$\mathfrak{X}$ is $\ell'$-variate $P$-polynomial}\} $.
\end{enumerate}
\end{rem}

\begin{lem}[{cf. \cite[Proof of Lemma~2.14]{BKZZ}}]
\label{lem:spanvij}
Let $\mathfrak{X}=(X,\mathcal{R},\mathcal{D})$ be an $\ell$-variate $P$-polynomial association scheme on the domain $\mathcal{D}\subset \N^\ell$ with respect to a monomial order $\le$.
For $\alpha\in \mathcal{D}$, we have
\[
\Span\{A_\beta \mid \beta\in \mathcal{D},\   \beta\le \alpha\}
=
\Span\{\bm{A}^\beta \mid \beta\in \mathcal{D},\   \beta\le \alpha\}. 
\]
\end{lem}

\begin{prop}[cf.~\cite{BKZZ}]
\label{prop:P-TFAE}
Let $\mathcal{D}\subset \N^\ell$ containing $\epsilon_1,\epsilon_2,\ldots ,\epsilon_\ell$, and
let $\mathfrak{X} = (X,\mathcal{R}, \mathcal{D})$
be a commutative association scheme.
Then the statements \ref{item:P-TFAE_i} and~\ref{item:P-TFAE_ii} are equivalent:
\begin{enumerate}[label=$(\roman*)$]
    \item \label{item:P-TFAE_i} $\mathfrak{X}$ is
    an $\ell$-variate $P$-polynomial association scheme 
    on $\mathcal{D}$ with respect to a monomial order $\le$;
    \item \label{item:P-TFAE_ii}
    the condition \ref{item:abPpoly_i} of Definition~\ref{df:abPpoly} holds for $\mathcal{D}$
    and the intersection numbers satisfy,
    for each $i=1,2,\ldots ,\ell$ and each $\alpha \in \mathcal{D}$,
    $p^\beta_{\epsilon_i, \alpha} \neq 0$ for $\beta\in \mathcal{D}$ implies $\beta \le \alpha+\epsilon_i$. 
    Moreover, if $\alpha+\epsilon_i\in \mathcal{D}$, then $p^{\alpha+\epsilon_i}_{\epsilon_i, \alpha} \neq 0$ holds.
\end{enumerate}
\end{prop}

The following lemma shows that operations on the index set $\mathcal{D}$ of an association scheme are justified even when they extend beyond $\mathcal{D}$. 
In particular, the statement \ref{item:p_support_ii} extends Proposition~\ref{prop:P-TFAE}~\ref{item:P-TFAE_ii} to arbitrary $\alpha,\beta,\gamma\in \mathcal{D}$.

\begin{lem}
\label{lem:p_support}
Let $\mathfrak{X}=(X,\mathcal{R},\mathcal{D})$ be an $\ell$-variate $P$-polynomial association scheme 
on $\mathcal{D}\subset\N^\ell$ with respect to a monomial order $\le$.
Then the following statements hold:
\begin{enumerate}[label=$(\roman*)$]
    \item \label{item:p_support_i} For any $\gamma\in \N^\ell$, $\bm{A}^\gamma$ is written as a linear combination of
$\{\bm{A}^\xi\mid \xi\in\mathcal{D},\ \xi\le \gamma\}$;
    \item \label{item:p_support_ii} For any $\alpha,\beta\in \mathcal{D}$,
$p^{\gamma}_{\alpha,\beta}\neq 0$ for $\gamma\in\mathcal{D}$ implies $\gamma\le \alpha+\beta$.
Moreover, if $\alpha+\beta\in\mathcal{D}$,
then $p^{\alpha+\beta}_{\alpha,\beta}\neq 0$ holds.
\end{enumerate}
\end{lem}

\begin{proof}
\ref{item:p_support_i}
We use induction on $\gamma \in \N^\ell$ with respect to $\le$.
For $\gamma=o$, this is trivial. For $\gamma\neq o$,
there exists some $i$ such that $\gamma-\epsilon_i\in\N^\ell$.
By the induction hypothesis, $\bm{A}^{\gamma-\epsilon_i}$ is a linear combination of
$\{\bm{A}^{\xi'}\mid \xi'\in\mathcal{D},\ \xi'\le \gamma-\epsilon_i\}$.
By Definition~\ref{df:abPpoly}~\ref{item:abPpoly_iii},
for each $\xi'\in\mathcal{D}$ with $\xi'\le \gamma-\epsilon_i$,
$A_{\epsilon_i}\bm{A}^{\xi'}$ is a linear combination of $\bm{A}^{\xi}$ with $\xi\in \mathcal{D}$ and $\xi\le \xi'+\epsilon_i\le \gamma$.
Hence, $\bm{A}^\gamma=A_{\epsilon_i}\bm{A}^{\gamma-\epsilon_i}$ is a linear combination of
$\{\bm{A}^\xi\mid \xi\in\mathcal{D},\ \xi\le \gamma\}$.

\noindent
\ref{item:p_support_ii}
Take $\alpha,\beta\in\mathcal{D}$. By Definition~\ref{df:abPpoly}~\ref{item:abPpoly_ii},
we can write $A_\alpha=v_\alpha(\bm{A})$ and $A_\beta=v_\beta(\bm{A})$ for some polynomial $v_\alpha$ (resp. $v_\beta$) of multidegree $\alpha$ (resp. $\beta$) with respect to $\le$. 
Thus $A_\alpha A_\beta$ is a linear combination of $\bm{A}^{\xi}$ with $\xi\le \alpha+\beta$.
By \ref{item:p_support_i}, this implies that
$A_\alpha A_\beta$ is a linear combination of
$\{\bm{A}^\xi\mid \xi\in\mathcal{D},\ \xi\le\alpha+\beta\}$.
Moreover, by Lemma~\ref{lem:spanvij}, for each $\xi\in\mathcal{D}$,
$\bm{A}^\xi$ is a linear combination of $\{A_\gamma\mid \gamma\in\mathcal{D},\ \gamma\le\xi\}$.
Thus $A_\alpha A_\beta$ is a linear combination of $\{A_\gamma\mid \gamma\in\mathcal{D},\ \gamma\le\alpha+\beta\}$.
Comparing coefficients in the expansion $A_\alpha A_\beta=\sum_{\gamma\in\mathcal{D}} p^{\gamma}_{\alpha,\beta} A_\gamma$
yields the conclusion.

The latter assertion can be found essentially in Lemma 3.9 of Bernard--Crampé--Vinet--Zaimi--Zhang~\cite{BCVZZ24}, so we refer the reader to that reference.
\end{proof}

We next record the $Q$-polynomial analogues of Proposition~\ref{prop:P-TFAE} and Lemma~\ref{lem:p_support}.
\begin{prop}[cf.~\cite{BKZZ}]
    \label{prop:Q-TFAE}
    Let $\mathcal{D}^\ast\subset \N^{\ell^\ast}$ containing $\epsilon_1,\epsilon_2,\ldots ,\epsilon_{\ell^\ast}$, let $\le$ be a monomial order on $\N^{\ell^\ast}$, and 
    let $\mathfrak{X}$ be a commutative association scheme with the primitive idempotents $\{E_\alpha\}_{\alpha\in \mathcal{D}^\ast}$ indexed by $\mathcal{D}^\ast$.
    The statements \ref{item:Q-TFAE_i} and~\ref{item:Q-TFAE_ii} are equivalent:
    \begin{enumerate}[label=$(\roman*)$]
        \item \label{item:Q-TFAE_i} $\mathfrak{X}$ is
        an $\ell^\ast$-variate $Q$-polynomial association scheme 
        on $\mathcal{D}^\ast$ with respect to $\le$;
        \item \label{item:Q-TFAE_ii}
        the condition \ref{item:abQpoly_i} of Definition~\ref{df:abQpoly} holds for $\mathcal{D}^\ast$
        and the Krein numbers satisfy,
        for each $i=1,2,\ldots ,\ell^\ast$ and each $\alpha \in \mathcal{D}^\ast$,
        $q^\beta_{\epsilon_i, \alpha} \neq 0$ for $\beta\in \mathcal{D}^\ast$ implies $\beta \le \alpha+\epsilon_i$. 
        Moreover, if $\alpha+\epsilon_i\in \mathcal{D}^\ast$, then $q^{\alpha+\epsilon_i}_{\epsilon_i, \alpha} \neq 0$ holds.
    \end{enumerate}
\end{prop}

\begin{lem}\label{lem:q_support}
Let $\mathfrak{X}$ be an $\ell^\ast$-variate $Q$-polynomial association scheme 
on $\mathcal{D}^\ast\subset\N^{\ell^\ast}$ with respect to a monomial order $\le$.
Then the following hold:
\begin{enumerate}[label=$(\roman*)$]
    \item \label{item:q_support_i} For any $\gamma=(\gamma_1,\ldots ,\gamma_{\ell^\ast})\in \N^{\ell^\ast}$, $E_{\epsilon_1}^{\circ \gamma_1} \circ \cdots \circ E_{\epsilon_{\ell^\ast}}^{\circ \gamma_{\ell^\ast}}$ is written as a linear combination of
$E_{\epsilon_1}^{\circ \xi_1} \circ \cdots \circ E_{\epsilon_{\ell^\ast}}^{\circ \xi_{\ell^\ast}}$ with $\xi=(\xi_1,\ldots ,\xi_{\ell^\ast})\in \mathcal{D}^\ast$ and $\xi\le \gamma$ under the Hadamard product;
    \item \label{item:q_support_ii} For any $\alpha,\beta\in \mathcal{D}^\ast$,
$q^{\gamma}_{\alpha,\beta}\neq 0$ for $\gamma\in\mathcal{D}^\ast$ implies $\gamma\le \alpha+\beta$.
Moreover, if $\alpha+\beta\in\mathcal{D}^\ast$,
then $q^{\alpha+\beta}_{\alpha,\beta}\neq 0$ holds.
\end{enumerate}
\end{lem}
\begin{proof}
By the proof of Lemma~\ref{lem:p_support}, the same result follows by replacing matrix products with Hadamard products, adjacency matrices $A_\alpha$ with $E_\alpha$, and intersection numbers $p_{\alpha,\beta}^\gamma$ with Krein numbers $q_{\alpha,\beta}^\gamma$.
\end{proof}

\section{Characterization via the first eigenmatrix of multivariate $P$-polynomial association schemes}\label{sec:eigen_characterization}

In this section, we characterize multivariate $P$-polynomial association schemes in terms of the first eigenmatrix $P$.
In the classical univariate setting, the phrase ``$P$-polynomial'' reflects the fact that the entries of $P$ are governed by a family of orthogonal polynomials.
Theorem~\ref{thm:eigen} shows that, in the multivariate setting, each row of the first eigenmatrix is likewise described by an $\ell$-variate polynomial.

We first describe the relationship between multivariate $P$-polynomial association schemes and ideals.
By Definition~\ref{df:abPpoly} \ref{item:abPpoly_iii},
for $\alpha\in \mathcal{D}$ and $i=1,2,\ldots ,\ell$ with $\alpha+\epsilon_i \notin \mathcal{D}$,
there exists a (unique) polynomial 
\begin{equation}
\label{eq:w_alpha+ei}    
w_{\alpha+\epsilon_i}(\bm{x}):=
\bm{x}^{\alpha+\epsilon_i} + \sum_{\substack{\beta\in \mathcal{D}\\ \beta < \alpha+\epsilon_i}} c_{\beta} \bm{x}^\beta
\end{equation}
of multidegree $\alpha+\epsilon_i$ in $\C [\bm{x}]$
such that $w_{\alpha+\epsilon_i}(\bm{A})=0$.
Let $I$ be the ideal of $\C [\bm{x}]$
generated by 
\begin{equation}
    \label{eq:ideal}
    \mathcal{G}:=
    \{w_{\alpha+\epsilon_i}(\bm{x}) \mid \alpha\in \mathcal{D},\  i=1,2,\ldots ,\ell,\ \alpha + \epsilon_i \notin \mathcal{D}\}.
\end{equation}

\begin{prop}[cf.~\cite{BKZZ}]
    \label{prop:A=C[x]/I}
    Let $\mathfrak{X}$ be an $\ell$-variate $P$-polynomial association scheme on $\mathcal{D}$
    with respect to a monomial order $\le$.
    Then the following statements hold:
    \begin{enumerate}[label=$(\roman*)$]
        \item \label{item:prop_A_CxmodI_i} $\mathcal{G}$ is a Gr\"obner basis of $I$;
        \item \label{item:prop_A_CxmodI_ii} $\MD(I)=\N^\ell \setminus \mathcal{D}$ holds;
        \item \label{item:prop_A_CxmodI_iii} The Bose--Mesner algebra $\mathfrak{A}$ of $\mathfrak{X}$
        is isomorphic to $\C [\bm{x}]/I$ as algebras.
    \end{enumerate}            
\end{prop}

We call the ideal $I$ the defining ideal of $\mathfrak{X}$.

\begin{thm}
\label{thm:eigen}
Let $\mathcal{D}\subset \N^\ell$ contain $\epsilon_1,\epsilon_2,\ldots ,\epsilon_\ell$ and satisfy Definition~\ref{df:abPpoly} \ref{item:abPpoly_i},
let $\le$ be a monomial order on $\N^\ell$, and
let $\mathfrak{X}=(X,\mathcal{R},\mathcal{D})$ be a commutative association scheme with primitive idempotents $\{E_j\}_{j\in \mathcal{J}}$.
Denote by $P$ the first eigenmatrix of $\mathfrak{X}$, and set
$\theta_{i}(j):=P_{\epsilon_i}(j)$
for $i=1,2,\ldots ,\ell$ and $j\in \mathcal{J}$.
Then the following conditions are equivalent:
\begin{enumerate}[label=$(\roman*)$]
\item \label{item:eigen_i} $\mathfrak{X}$ is
an $\ell$-variate $P$-polynomial association scheme on $\mathcal{D}$ with respect to $\le$.
\item \label{item:eigen_ii} The following statements hold:
\begin{enumerate}[label=$(\alph*)$]
\item \label{item:eigen_a}
For $\alpha\in \mathcal{D}$,
there exists an $\ell$-variate polynomial
$v_\alpha(\bm{x})=\sum_{\beta \in \mathcal{D},\  \beta \le \alpha} c_\beta \bm{x}^\beta$
of multidegree $\alpha$ with respect to $\le$
such that $P_\alpha (j)=v_\alpha(\theta_{1}(j),\theta_{2}(j),\ldots ,\theta_{\ell}(j))$ holds for all $j\in \mathcal{J}$.
\item \label{item:eigen_b}
For $\alpha\in \mathcal{D}$ and $i=1,2,\ldots ,\ell$ with $\alpha+\epsilon_i \notin \mathcal{D}$,
there exists an $\ell$-variate monic polynomial
$w_{\alpha+\epsilon_i} (\bm{x})$ of multidegree $\alpha + \epsilon_i$ with respect to $\le$
such that $w_{\alpha+\epsilon_i}(\theta_{1}(j),\theta_{2}(j),\ldots ,\theta_{\ell}(j)) =0$ holds for all $j\in \mathcal{J}$.
\end{enumerate}
\end{enumerate}
\end{thm}

\begin{proof}
\ref{item:eigen_i} $\Longrightarrow$ \ref{item:eigen_ii}:
By Definition~\ref{df:abPpoly}~\ref{item:abPpoly_ii},
for $\alpha\in \mathcal{D}$,
there exists an $\ell$-variate polynomial
$v_\alpha(\bm{x})$ of multidegree $\alpha$
such that $A_\alpha = v_\alpha(\bm{A})$.
For $j\in \mathcal{J}$,
we have $A_\alpha E_j =P_\alpha (j) E_j$ and
\[
A_\alpha E_j
=
v_\alpha(A_{\epsilon_1},A_{\epsilon_2},\ldots ,A_{\epsilon_\ell}) E_j
=
v_\alpha(\theta_{1}(j),\theta_{2}(j),\ldots ,\theta_{\ell}(j)) E_j.
\]
Hence we have $P_\alpha(j)=v_\alpha(\theta_{1}(j),\theta_{2}(j),\ldots ,\theta_{\ell}(j))$.
Next, for $\alpha\in \mathcal{D}$ and $i=1,2,\ldots ,\ell$ with $\alpha+\epsilon_i \notin \mathcal{D}$,
by \eqref{eq:w_alpha+ei},
there exists an $\ell$-variate monic polynomial $w_{\alpha+\epsilon_i}(\bm{x})$ of multidegree $\alpha+\epsilon_i$ such that $w_{\alpha+\epsilon_i}(\bm{A})=0$.
Since for $j\in \mathcal{J}$,
\[
0=w_{\alpha+\epsilon_i}(A_{\epsilon_1},A_{\epsilon_2},\ldots ,A_{\epsilon_\ell}) E_j
=
w_{\alpha+\epsilon_i}(\theta_{1}(j),\theta_{2}(j),\ldots ,\theta_{\ell}(j)) E_j
\]
holds, we have
$w_{\alpha+\epsilon_i}(\theta_{1}(j),\theta_{2}(j),\ldots ,\theta_{\ell}(j))=0$.

\noindent
\ref{item:eigen_ii} $\Longrightarrow$ \ref{item:eigen_i}:
Fix $\alpha\in \mathcal{D}$ and write
\[
v_\alpha(x_1,x_2,\ldots ,x_\ell)=\sum_{\beta=(m_1,m_2,\ldots ,m_\ell) \in \mathcal{D}, \beta \le \alpha} c_\beta x_1^{m_1} x_2^{m_2} \cdots x_\ell^{m_\ell}.
\]
Then
\begin{align*}
A_\alpha &=\sum_{j\in \mathcal{J}} P_\alpha (j) E_j
=
\sum_{j\in \mathcal{J}} v_\alpha(\theta_{1}(j),\theta_{2}(j),\ldots ,\theta_{\ell}(j)) E_j\\
&=
\sum_{j\in \mathcal{J}} \sum_{\beta=(m_1,m_2,\ldots ,m_\ell) \in \mathcal{D}, \beta \le \alpha} c_\beta 
\theta_{1}(j)^{m_1} \theta_{2}(j)^{m_2} \cdots \theta_{\ell}(j)^{m_\ell} E_j\\
&=
\sum_{\beta=(m_1,m_2,\ldots ,m_\ell) \in \mathcal{D}, \beta \le \alpha} c_\beta 
\sum_{j\in \mathcal{J}} 
\theta_{1}(j)^{m_1} \theta_{2}(j)^{m_2} \cdots \theta_{\ell}(j)^{m_\ell} E_j\\
&=
\sum_{\beta=(m_1,m_2,\ldots ,m_\ell) \in \mathcal{D}, \beta \le \alpha} c_\beta 
A_{\epsilon_1}^{m_1} A_{\epsilon_2}^{m_2} \cdots A_{\epsilon_\ell}^{m_\ell} =v_\alpha(\bm{A}).
\end{align*}
Next fix $\alpha\in \mathcal{D}$ and $i=1,2,\ldots ,\ell$ with $\alpha+\epsilon_i \notin \mathcal{D}$.
By the assumption, we have
\[
0=
\sum_{j\in \mathcal{J}}
w_{\alpha+\epsilon_i}(\theta_{1}(j),\theta_{2}(j),\ldots ,\theta_{\ell}(j))E_j
=
w_{\alpha+\epsilon_i}(A_{\epsilon_1},A_{\epsilon_2},\ldots ,A_{\epsilon_\ell}).
\]

From the above, we can conclude that $\mathfrak{X}$ satisfies Definition~\ref{df:abPpoly}~\ref{item:abPpoly_ii} and~\ref{item:abPpoly_iii},
so $\mathfrak{X}$ is an $\ell$-variate $P$-polynomial association scheme on $\mathcal{D}$ with respect to $\le$.
\end{proof}

From Theorem~\ref{thm:eigen}, we obtain the following corollary. 
\begin{cor}
    \label{cor:ideal}
Let $I$ be the ideal defined in \eqref{eq:ideal}.
Then 
\[
I=\{f\in \C[x_1,x_2,\ldots,x_\ell] \mid f(\theta_1(j),\theta_2(j),\ldots,\theta_\ell(j))=0 \text{ for all } j\in \mathcal{J}\}
\]
holds.
That is, $I$ is the vanishing ideal of the finite affine variety with support set $V:=\{(\theta_1(j),\theta_2(j),\ldots,\theta_\ell(j))\mid j\in\mathcal{J}\}\subset\C^\ell$.
Moreover, the set $V(I)=\{\bm{x}\in\C^\ell \mid f(\bm{x})=0 \text{ for all } f\in I\}$ coincides with $V$.
\end{cor}
\begin{proof}
Put
\[
I(V):=\{f\in \C[x_1,x_2,\ldots,x_\ell] \mid f(\theta_1(j),\theta_2(j),\ldots,\theta_\ell(j))=0 \text{ for all } j\in \mathcal{J}\}.
\]
We show that $I=I(V)$.

First we prove $I\subset I(V)$.
For each generator $w_{\alpha+\epsilon_i}(\bm{x})\in \mathcal{G}$,
Theorem~\ref{thm:eigen}~\ref{item:eigen_b} gives
$w_{\alpha+\epsilon_i}(\theta_1(j),\theta_2(j),\ldots,\theta_\ell(j))=0$ for $j\in \mathcal{J}$.
Hence every element of $\mathcal{G}$ belongs to $I(V)$, i.e., $I=\langle \mathcal{G}\rangle \subset I(V)$. 

Next we prove $I(V)\subset I$. Take $f\in I(V)$.
Since $\mathcal{G}$ is a Gr\"obner basis of $I$ by Proposition~\ref{prop:A=C[x]/I}~\ref{item:prop_A_CxmodI_i}, 
we have $f=\sum_{g\in \mathcal{G}} q_g g+r$,
where $q_g\in \C[x_1,x_2,\ldots,x_\ell]$ and the remainder $r$ has no monomial in $\MD(I)$.
We know from Proposition~\ref{prop:A=C[x]/I}~\ref{item:prop_A_CxmodI_ii} that $\MD(I)=\N^\ell\setminus \mathcal{D}$,
so every monomial appearing in $r$ is of the form $\bm{x}^\alpha$ with $\alpha\in \mathcal{D}$. 
Hence we may write
$r(\bm{x})=\sum_{\alpha\in \mathcal{D}} c_\alpha \bm{x}^\alpha$
for some $c_\alpha\in \C$.
Note that $r=f-\sum_{g\in \mathcal{G}} q_g g \in I(V)$ holds since each $g\in \mathcal{G}$ belongs to $I(V)$ and $f\in I(V)$.  

For each $\alpha\in \mathcal{D}$, let $v_\alpha(\bm{x})$ be the polynomial in Theorem~\ref{thm:eigen}~\ref{item:eigen_ii}~\ref{item:eigen_a}.
Since $v_\alpha(\bm{x})$ has multidegree $\alpha$, the transition matrix from
$\{v_\alpha(\bm{x})\mid \alpha\in \mathcal{D}\}$ to
$\{\bm{x}^\alpha\mid \alpha\in \mathcal{D}\}$ is triangular with nonzero diagonal entries.
Therefore,
$\{v_\alpha(\bm{x})\mid \alpha\in \mathcal{D}\}$
is a basis of the vector space spanned by
$\{\bm{x}^\alpha\mid \alpha\in \mathcal{D}\}$.
Thus there exist unique scalars $d_\alpha\in \C$ such that
$r(\bm{x})=\sum_{\alpha\in \mathcal{D}} d_\alpha v_\alpha(\bm{x})$.
Evaluating this identity at $(\theta_1(j),\theta_2(j),\ldots,\theta_\ell(j))$ for $j\in \mathcal{J}$, we obtain
\[
0=r(\theta_1(j),\theta_2(j),\ldots,\theta_\ell(j))
=\sum_{\alpha\in \mathcal{D}} d_\alpha v_\alpha(\theta_1(j),\theta_2(j),\ldots,\theta_\ell(j))
=\sum_{\alpha\in \mathcal{D}} d_\alpha P_\alpha(j),
\]
where the last equality follows from Theorem~\ref{thm:eigen}~\ref{item:eigen_ii}~\ref{item:eigen_a}.
Since $P$ is invertible, it follows that $d_\alpha=0$ for all $\alpha\in \mathcal{D}$.
Hence $r=0$, and therefore $f\in I$.
This proves $I(V)\subset I$.
Consequently, $I=I(V)$.

Finally, we prove that $V(I)=V$.
The inclusion $V\subseteq V(I)$ is immediate from the definition of $I(V)$ and the equality $I=I(V)$.
For the reverse inclusion, take $\bm{a}=(a_1,\ldots,a_\ell)\in V(I)$ and assume that $\bm{a}\notin V$.
For each $\bm{b}=(b_1,\ldots,b_\ell)\in V$, choose an index $t(\bm{b})\in\{1,\ldots,\ell\}$ such that $a_{t(\bm{b})}\neq b_{t(\bm{b})}$, and define
\[
\delta_{\bm{b}}(\bm{x})
:=
\frac{x_{t(\bm{b})}-b_{t(\bm{b})}}{a_{t(\bm{b})}-b_{t(\bm{b})}}.
\]
Then the polynomial $F(\bm{x}):=\prod_{\bm{b}\in V}\delta_{\bm{b}}(\bm{x})$
vanishes on $V$, so $F\in I(V)=I$.
On the other hand, $F(\bm{a})=1$, contradicting $\bm{a}\in V(I)$.
Hence $V(I)\subseteq V$, and therefore $V(I)=V$.
\end{proof}

A parallel statement of Theorem~\ref{thm:eigen} for multivariate $Q$-polynomial association schemes also holds.

\begin{thm}
\label{thm:eigenQ}
Let $\mathcal{D}^\ast\subset \N^{\ell^\ast}$ contain $\epsilon_1,\epsilon_2,\ldots ,\epsilon_{\ell^\ast}$ and satisfy Definition~\ref{df:abQpoly}~\ref{item:abQpoly_i}, 
let $\le$ be a monomial order on $\N^{\ell^\ast}$, and let $\mathfrak{X}=(X,\mathcal{R},\mathcal{I})$ be a commutative association scheme with primitive idempotents $\{E_\alpha\}_{\alpha\in \mathcal{D}^\ast}$.
Denote by $Q$ the second eigenmatrix of $\mathfrak{X}$, and set
$\theta^\ast_{j}(i):=Q_{\epsilon_j}(i)$
for $j=1,2,\ldots ,\ell^\ast$ and $i\in \mathcal{I}$.
Then the following conditions are equivalent:
\begin{enumerate}[label=$(\roman*)$]
\item \label{item:eigenQ_i}  $\mathfrak{X}$ is
an $\ell^\ast$-variate $Q$-polynomial association scheme on $\mathcal{D}^\ast$ with respect to $\le$.
\item \label{item:eigenQ_ii} The following statements hold:
\begin{enumerate}[label=$(\alph*)$]
\item \label{item:eigenQ_a}
For $\alpha\in \mathcal{D}^\ast$,
there exists an $\ell^\ast$-variate polynomial
$v^\ast_\alpha(\bm{x})=\sum_{\beta \in \mathcal{D}^\ast,\  \beta \le \alpha} c_\beta \bm{x}^\beta$
of multidegree $\alpha$ with respect to $\le$
such that $Q_\alpha (i)=v^\ast_\alpha(\theta^\ast_{1}(i),\theta^\ast_{2}(i),\ldots ,\theta^\ast_{\ell^\ast}(i))$ holds for all $i\in \mathcal{I}$.
\item \label{item:eigenQ_b} For $\alpha\in \mathcal{D}^\ast$ and $j=1,2,\ldots ,\ell^\ast$ with $\alpha+\epsilon_j \notin \mathcal{D}^\ast$,
there exists an $\ell^\ast$-variate monic polynomial
$w^\ast_{\alpha+\epsilon_j} (\bm{x})$
of multidegree $\alpha + \epsilon_j$ with respect to $\le$
such that $w^\ast_{\alpha+\epsilon_j}(\theta^\ast_{1}(i),\theta^\ast_{2}(i),\ldots ,\theta^\ast_{\ell^\ast}(i)) =0$ holds for all $i\in \mathcal{I}$.
\end{enumerate}
\end{enumerate}
\end{thm}

\section{Main theorems}\label{sec:MainTheorem}

\subsection{Imprimitive association schemes and elimination orders}

\begin{thm}\label{thm:EliminationThm}
Given an association scheme $\mathfrak{X}$, the following conditions are equivalent: 
\begin{enumerate}[label=$(\roman*)$]
  \item \label{item:EliminationThm_i} $\mathfrak{X}$ is imprimitive; 
  \item \label{item:EliminationThm_ii} there exist $\mathcal{D}\subset \N^\ell$ and a monomial order $\le$ of $s$-elimination type such that 
  $\mathfrak{X}$ is an $\ell$-variate $P$-polynomial association scheme on $\mathcal{D}$ with respect to $\le$; 
  \item \label{item:EliminationThm_iii} there exist $\mathcal{D}^\ast\subset \N^{\ell^\ast}$ and a monomial order $\le^\ast$ of $s^\ast$-elimination type such that 
  $\mathfrak{X}$ is an $\ell^\ast$-variate $Q$-polynomial association scheme on $\mathcal{D}^\ast$ with respect to $\le^\ast$.  
\end{enumerate}            
\end{thm}
\begin{proof}
\ref{item:EliminationThm_i} $\Longrightarrow$ \ref{item:EliminationThm_ii}: Let $\mathfrak{X}$ be an imprimitive association scheme of class $\ell$.
Choose a closed subset $\mathcal{C}\subset \mathcal{I}$ with $|\mathcal{C}|=\ell-s+1$, where $1\le s\le \ell-1$.
Reindex $\mathcal{I}$ by $\mathcal{D}:=\{o,\epsilon_1,\ldots,\epsilon_\ell\}\subset \N^\ell$ so that
$\mathcal{C}=\{o,\epsilon_{s+1},\ldots,\epsilon_\ell\}\subset \mathcal{D}$.
We define a total order on $\N^\ell$ as follows:
for $(\alpha_1,\beta_1),(\alpha_2,\beta_2)\in \N^{s}\times \N^{\ell-s}$, set $(\alpha_1,\beta_1)> (\alpha_2,\beta_2)$ if 
\begin{enumerate}[label=$(\alph*)$]
    \item \label{item:elilex1} $|\alpha_1|> |\alpha_2|$; or 
    \item  \label{item:elilex2} $|\alpha_1|= |\alpha_2|$ and $|\beta_1|> |\beta_2|$; or 
    \item \label{item:elilex3} $|\alpha_1|= |\alpha_2|$ and $|\beta_1|= |\beta_2|$ and $(\alpha_1,\beta_1)>_{\mathrm{lex}} (\alpha_2,\beta_2)$. 
\end{enumerate}
Then we see that $\le$ is a monomial order on $\N^\ell$ of $s$-elimination type. 

In what follows, we prove that $\mathfrak{X}$ is an $\ell$-variate $P$-polynomial association scheme on $\mathcal{D}$ with respect to $\le$ 
by checking the conditions in Definition~\ref{df:abPpoly}. The first two conditions \ref{item:abPpoly_i} and~\ref{item:abPpoly_ii} are trivial. 
We prove that \ref{item:abPpoly_iii} also holds. 
For $1 \le i,j \le \ell$, we have 
    \[A_{\epsilon_i} A_{\epsilon_j}=p^{o}_{\epsilon_i,\epsilon_j}A_{o} + \sum^{\ell}_{k=1} p^{\epsilon_k}_{\epsilon_i,\epsilon_j} A_{\epsilon_k}.\]
\begin{itemize}
    \item 
    If $1\le i,j\le s$, then $\epsilon_i+\epsilon_j> \epsilon_k$ holds for any $1\le k\le \ell$ by definition \ref{item:elilex1} of the above monomial order $\le$. 
    \item For $1 \leq j \leq s < i \leq \ell$, Lemma~\ref{lem:closed_subset}~\ref{lem:closed_subset:mult} leads to $p^{\epsilon_k}_{\epsilon_i,\epsilon_j}=0$ for $s<k\le \ell$.
    By definition \ref{item:elilex2} of $\le$, we have $\epsilon_i+\epsilon_j> \epsilon_k$ for any $1 \le k \le s$. 
    The case of $1 \le i \le s < j \le \ell$ is the same by the commutativity of intersection numbers. 
    \item For $s < i,j \le \ell$, Lemma~\ref{lem:closed_subset}~\ref{lem:closed_subset:mult} leads to $p^{\epsilon_k}_{\epsilon_i,\epsilon_j}=0$ for $1\le k\le s$.
    By definition \ref{item:elilex2} of $\le$, we have $\epsilon_i+\epsilon_j> \epsilon_k$ for any $s < k \le \ell$. 
\end{itemize}
These imply that Definition~\ref{df:abPpoly}~\ref{item:abPpoly_iii} holds. 

\noindent
\ref{item:EliminationThm_ii} $\Longrightarrow$ \ref{item:EliminationThm_i}:
Let $\mathfrak{X}=(X,\{A_\alpha\}_{\alpha\in \mathcal{D}})$ be an $\ell$-variate $P$-polynomial association scheme on $\mathcal{D}$ with respect to a monomial order of $s$-elimination type. 
Let
\[
\mathcal{C}:= \mathcal{D} \cap (\{o\}\times \N^{\ell-s})
= \{\alpha =(\alpha_1,\ldots , \alpha_\ell) \in \mathcal{D} \mid \alpha_1=\cdots=\alpha_s=0\}.
\]
It suffices to show that $\mathcal{C}$ is a closed subset. 

Fix $\alpha,\beta\in \mathcal{C}$ arbitrarily. 
Then, $A_\alpha$ and $A_\beta$ can be expressed as polynomials in $A_{\epsilon_1},\ldots ,A_{\epsilon_\ell}$, so we write $A_\alpha=v_\alpha(\bm{A})$ and $A_\beta=v_\beta(\bm{A})$.
Also, $A_\alpha$ is a normal matrix by the commutativity of $\mathfrak{X}$. 
Thus there exists a polynomial $p(x) \in \C[x]$ such that $\overline{A}_\alpha^T=p(A_\alpha)$ holds. Hence, 
\[A_{\alpha^T} A_\beta =A_\alpha^T A_\beta=\overline{p(A_\alpha)}A_\beta=\overline{p(v_\alpha(\bm{A}))} v_\beta(\bm{A})
=\sum_{\gamma\in \mathcal{C}} c_{\gamma} \bm{A}^{\gamma}.\]
The last equality follows because Lemma~\ref{lem:p_support}~\ref{item:p_support_i} controls the multidegrees of products of monomials, and $\le$ is of $s$-elimination type.
By Lemma~\ref{lem:spanvij}, we have $A_{\alpha^T} A_\beta\in \Span\{A_\gamma \mid \gamma\in \mathcal{C}\}$, i.e., $\alpha^T \beta\subset \mathcal{C}$.
Therefore, $\mathcal{C}$ is a closed subset. 

\noindent
\ref{item:EliminationThm_i} $\Longrightarrow$ \ref{item:EliminationThm_iii}:
If $\mathfrak{X}$ is imprimitive, then there exists a closed subset $\mathcal{C}^\ast\subset \mathcal{J}$ with respect to the Hadamard product $\circ$. 
Let $\ell^\ast=|\mathcal{J}|-1$ and let us reindex $\mathcal{J}$ by $\mathcal{D}^\ast:=\{o,\epsilon_1,\ldots,\epsilon_{\ell^\ast}\}\subset\N^{\ell^\ast}$ so that $\mathcal{C}^\ast=\{o,\epsilon_{s^\ast+1},\ldots,\epsilon_{\ell^\ast}\}\subset \mathcal{D}^\ast$, where $s^\ast$ satisfies $|\mathcal{C}^\ast|=\ell^\ast-s^\ast+1$. 
Applying the same argument as in \ref{item:EliminationThm_i}$\implies$\ref{item:EliminationThm_ii}, we obtain the conclusion. 

\noindent
\ref{item:EliminationThm_iii} $\Longrightarrow$ \ref{item:EliminationThm_i}:
Let $\mathfrak{X}=(X,\{E_\alpha\}_{\alpha\in \mathcal{D}^\ast})$ be an $\ell^\ast$-variate $Q$-polynomial association scheme on $\mathcal{D}^\ast$ with respect to an $s^\ast$-elimination-type monomial order $\le^\ast$ and let 
\[\mathcal{C}^\ast:=\mathcal{D}^\ast \cap (\{o\}\times \N^{\ell^\ast - s^\ast}).\] 

By Lemma~\ref{lem:q_support}~\ref{item:q_support_ii}, we see that for $\alpha,\beta\in \mathcal{C}^\ast$, 
if $q^\gamma_{\alpha\beta}\neq 0$, then $\gamma\le \alpha+\beta$ holds. 
Namely, $E_\alpha\circ E_\beta$ is a linear combination of $E_\gamma$ satisfying $\gamma\le \alpha+\beta$. 
On the other hand, since $\alpha+\beta\in \{o\}\times \N^{\ell^\ast-s^\ast}$ and $\le^\ast$ is of $s^\ast$-elimination type, the first $s^\ast$ entries of $\gamma$ must all be $0$. 
Hence, $\gamma\in \mathcal{C}^\ast$. 
This implies that $\mathrm{span}\{E_\gamma\mid \gamma\in \mathcal{C}^\ast\}$ is closed under Hadamard product. 

Therefore, by Lemma~\ref{lem:dual_closed_subset_equiv}, $\mathcal{C}^\ast$ is a closed subset with respect to $\circ$. 
Since $1\le s^\ast<\ell^\ast$, the subset $\mathcal{C}^\ast$ is nontrivial. Therefore, $\mathfrak{X}$ is imprimitive. 
\end{proof}

\begin{rem}[Relationship with $m$-distance-regular graphs]
\label{rem:mdrg_connection}
The notion of an \emph{$m$-distance-regular graph} was introduced by Bernard--Cramp\'{e}--Vinet--Zaimi--Zhang~\cite{BCVZZ24} as the graph-theoretic counterpart of multivariate $P$-polynomial association schemes.
In the symmetric case, these two notions correspond naturally.
Under this correspondence, the generating matrices $\{A_{\epsilon_1},\ldots,A_{\epsilon_m}\}$ of an $m$-variate $P$-polynomial association scheme are interpreted as the adjacency matrices of a graph $G=(X,E_1\sqcup\cdots\sqcup E_m)$ whose edges in $E_i$ are colored by the $i$th color, and each $A_\alpha$ for $\alpha\in \mathcal{D}$ is interpreted as the matrix of the $\alpha$-distance relation with respect to the chosen monomial order.

Therefore, Theorem~\ref{thm:EliminationThm}~\ref{item:EliminationThm_i}$\Longleftrightarrow$\ref{item:EliminationThm_ii} admits the following graph-theoretic interpretation:
the imprimitivity of $\mathfrak{X}$ is equivalent to the existence of an $s$-elimination-type monomial order governing the associated $m$-distance.
In particular, the closed subset corresponding to $\mathcal{C}$ induces a block partition on $G$, and a monomial order of $s$-elimination type may be viewed as a hierarchical rule that prioritizes moves inside a block over moves between blocks.
\end{rem}

\subsection{Quotient and block schemes and elimination orders}\label{subsec:quot}

Throughout this subsection, let $\mathfrak{X}=(X,\mathcal{R},\mathcal{D})$ be an imprimitive $\ell$-variate $P$-polynomial association scheme on $\mathcal{D}\subset \N^\ell$ with respect to $\le$, and assume that $\le$ is of $s$-block type.
Every monomial order of $s$-block type is, in particular, of $s$-elimination type; see Definition~\ref{def:elim_block}.
Let $\mathcal{C}=\mathcal{D} \cap (\{o\}\times \N^{\ell-s})$. Then $\mathcal{C}$ is a closed subset by the proof of Theorem~\ref{thm:EliminationThm}. 

\begin{lem}
\label{lem:qt_equiv}
Let $(X,\mathcal{R},\mathcal{D})$ be as above. 
Then the following equivalence holds for any $(\alpha_1,\beta_1), (\alpha_2,\beta_2)\in \mathcal{D}\subset \N^{s}\times \N^{\ell - s}$: 
\[(\alpha_1,\beta_1)\equiv (\alpha_2,\beta_2) \Longleftrightarrow \alpha_1=\alpha_2,\]
where $\equiv$ is an equivalence relation on $\mathcal{D}$ appearing in Lemma~\ref{lem:closed_subset_equiv}~\ref{item:closed_subset_equiv_iii}. 
\end{lem}
\begin{proof}
($\Longrightarrow$): Let $(\alpha_1,\beta_1) \equiv (\alpha_2,\beta_2)$, i.e., $(\alpha_1,\beta_1)^T (\alpha_2,\beta_2)\cap \mathcal{C}\neq \emptyset$.
Then there is $\gamma \in \N^{\ell - s}$ such that $(o,\gamma)\in \mathcal{C}$ and $p^{(o,\gamma)}_{(\alpha_1,\beta_1)^T(\alpha_2,\beta_2)}\neq 0$ hold.
By Lemma~\ref{lem:intersection_krein_relation}~\ref{item:intersection_numbers}, this implies
$p^{(\alpha_2,\beta_2)}_{(\alpha_1,\beta_1),(o,\gamma)}\neq 0$. 
By Lemma~\ref{lem:p_support}~\ref{item:p_support_ii}, we have $(\alpha_2,\beta_2) \le (\alpha_1,\beta_1+\gamma)$. 
Since $\le$ is of $s$-block type, we have $(\alpha_2,o) \le (\alpha_1,o)$. 

Similarly, by the symmetry of $\equiv$,
we also have $(\alpha_1,o) \le (\alpha_2,o)$. Thus, $\alpha_1=\alpha_2$. 

\noindent
($\Longleftarrow$): By Lemma~\ref{lem:p_support}~\ref{item:p_support_ii}, 
$p^{(\alpha_1,\beta_1)}_{(\alpha_1,o),(o,\beta_1)}\neq 0$ holds.
By Lemma~\ref{lem:intersection_krein_relation}~\ref{item:intersection_numbers}, this implies
$p^{(o,\beta_1)}_{(\alpha_1,o)^T(\alpha_1,\beta_1)}\neq 0$, i.e., $(\alpha_1,o) \equiv (\alpha_1,\beta_1)$ holds. 
Similarly,$(\alpha_2,o) \equiv (\alpha_2,\beta_2)$ holds. 
Hence, by $\alpha_1=\alpha_2$, we obtain $(\alpha_1,\beta_1) \equiv (\alpha_1, o) = (\alpha_2,o) \equiv (\alpha_2,\beta_2)$. 
\end{proof}

\begin{rem}\label{rem:elim_block}
In the above proof of ($\Longrightarrow$), we use the property of a monomial order of $s$-block type, but this cannot be extended to the case of $s$-elimination type. In fact, 
let $\le$ be a monomial order defined by $\alpha_1 > \alpha_2$ if and only if 
\begin{itemize}
    \item $a_1+b_1>a_2+b_2$ or
    \item $a_1+b_1=a_2+b_2$ and $\alpha_1 >_{\mathrm{grlex}} \alpha_2$ 
\end{itemize}
for $\alpha_1=(a_1,b_1,c_1,d_1),\alpha_2=(a_2,b_2,c_2,d_2)\in \N^2\times \N^2$. 
Then this is of $2$-elimination type, but not of $2$-block type. 
For example, $(1,0,0,0)>(0,1,0,0)$ holds, while $(1,0,0,0) < (0,1,1,0)$ holds. 
\end{rem}

Let $\jmath:\N^s\to \N^s\times\N^{\ell-s}$ be the inclusion defined by $\jmath(\alpha):=(\alpha,o)$, and let $\le_s$ be the monomial order on $\N^s$ induced from $\le$ by $\jmath$ from Lemma~\ref{lem:induced_monomial_order}. 
For $\mathcal{D}$, let 
\[\mathcal{D}_s:=\jmath^{-1}(\mathcal{D})=\{\alpha\in \N^s \mid (\alpha,o) \in \mathcal{D}\}, \]
and for $\alpha\in \mathcal{D}_s$, let 
\[\tilde{v}_{\alpha}(\bm{x})=\sum_{(\alpha,\beta)\in (\alpha,o)\mathcal{C}} v_{(\alpha,\beta)}(\bm{x}) \in \C[x_1,x_2,\ldots ,x_\ell].\]
When we regard $\tilde{v}_{\alpha}$ as a polynomial in $(\C[x_{s+1},x_{s+2},\ldots ,x_\ell])[x_1,x_2,\ldots ,x_s]$, each exponent of $\tilde{v}_{\alpha}$ is less than or equal to $\alpha$ with respect to $\le_s$ and belongs to $\mathcal{D}_s$. 
Namely, since $\le$ is of $s$-block type, we can rewrite it as 
\[
\tilde{v}_{\alpha}(x_1,x_2,\ldots ,x_s)=\sum_{\alpha' \in \mathcal{D}_s, \alpha'\le_s \alpha} c_{\alpha'} (x_{s+1},x_{s+2},\ldots ,x_\ell) \bm{x}^{\alpha'}. 
\]
Let 
\begin{align}u_\alpha(x_1,x_2,\ldots ,x_s)&=\tilde{v}_{\alpha}(x_1,x_2,\ldots ,x_s,k_{s+1},k_{s+2},\ldots ,k_\ell) \notag\\
&=\sum_{\alpha' \in \mathcal{D}_s, \alpha'\le_s \alpha} c_{\alpha'} (k_{s+1},k_{s+2},\ldots ,k_\ell) \bm{x}^{\alpha'},
\label{eq:u_alpha}
\end{align}
where $k_{s+1},k_{s+2},\ldots ,k_\ell$ are the valencies of $\epsilon_{s+1},\epsilon_{s+2},\ldots,\epsilon_\ell$, respectively. 

\begin{lem}\label{lem:leading_coeff_u_alpha}
Work with the same notation as above. Then the leading monomial of $u_{\alpha}$ with respect to $\le_s$ is $\bm{x}^{\alpha}$, and its coefficient is $c_{\alpha}(k_{s+1},k_{s+2},\ldots,k_\ell)\neq 0$. 
\end{lem}
\begin{proof}
Let $\mathfrak{A}$ be the Bose--Mesner algebra of $\mathfrak{X}$ and let 
$M=\frac{1}{p}\sum_{\alpha\in \mathcal{C}}A_\alpha$
(see Section~\ref{sec:ImprimitiveAS}). 
Since $\epsilon_i\in \mathcal{C}$ for $i=s+1,\ldots,\ell$, one has $\epsilon_i\mathcal{C}=\mathcal{C}$ and hence $k_{\epsilon_i\mathcal{C}}=1$.
By \eqref{eq:quotient_scheme_A}, we have $A_{\epsilon_i}M=k_i M$ for $i=s+1,\ldots,\ell$.

Fix $\alpha\in\mathcal{D}_s$ arbitrarily. 
For an equivalence class $\bm{\alpha}:=(\alpha,o)\mathcal{C}\in \mathcal{D}/\mathcal{C}$, put $A_{\bm{\alpha}}:=\sum_{(\alpha,\beta)\in\bm{\alpha}}A_{(\alpha,\beta)}$.
Then we have $A_{\bm{\alpha}}=\sum_{(\alpha,\beta)\in\bm{\alpha}}v_{(\alpha,\beta)}(\bm{A})=\tilde{v}_{\alpha}(\bm{A})$. 
On the other hand, 
the adjacency matrix $D_{\bm{\alpha}}$ of $\mathfrak{X}/\mathcal{C}$
corresponds to the
equation \eqref{eq:MAM} as $\frac{1}{p}A_{\bm{\alpha}}M$ in $\mathfrak{A} M$.
Thus,
\begin{equation}\label{eq:quotient_u_alpha_matrix}
    D_{\bm{\alpha}}\mapsto
    \frac{1}{p}A_{\bm{\alpha}}M =\frac{1}{p}\tilde{v}_{\alpha}(\bm{A})M
    =\frac{1}{p}u_{\alpha}(A_{\epsilon_1},\ldots,A_{\epsilon_s})M. 
\end{equation}
By substituting \eqref{eq:u_alpha} into \eqref{eq:quotient_u_alpha_matrix}, we obtain that 
\begin{equation}\label{eq:quotient_u_alpha_expansion}
D_{\bm{\alpha}}\mapsto
\frac{1}{p}\sum_{\alpha'\in \mathcal{D}_s,\ \alpha'\le_s \alpha}
c_{\alpha'}(k_{s+1},\ldots,k_\ell)\,\bm{A}^{(\alpha',o)}M. 
\end{equation}
Since $\mathfrak{A} M$ is isomorphic to the Bose--Mesner algebra of $\mathfrak{X}/\mathcal{C}$, we have
$\dim \mathfrak{A} M=|\mathcal{D}/\mathcal{C}|=|\mathcal{D}_s|$.
Moreover, \eqref{eq:quotient_u_alpha_expansion} shows that each basis vector $D_{\bm{\alpha}}$ of the quotient Bose--Mesner algebra is a linear combination of $\{\bm{A}^{(\alpha',o)}M\mid \alpha'\in \mathcal{D}_s\}$.
Thus $\{\bm{A}^{(\alpha,o)}M\}_{\alpha\in\mathcal{D}_s}$ spans $\mathfrak{A} M$.
Because this set has cardinality $|\mathcal{D}_s|=\dim \mathfrak{A}M$, it is a basis.
Order the elements of $\mathcal{D}_s$ with respect to $\le_s$.
Then \eqref{eq:quotient_u_alpha_expansion} shows that the transition matrix from the basis $\{D_{\bm{\alpha}}\}_{\alpha\in\mathcal{D}_s}$ to the basis $\{\bm{A}^{(\alpha,o)}M\}_{\alpha\in\mathcal{D}_s}$ is triangular, with diagonal entries $c_{\alpha}(k_{s+1},\ldots,k_\ell)/p$.
A triangular transition matrix between two bases has nonzero diagonal entries, so $c_{\alpha}(k_{s+1},\ldots,k_\ell)\neq 0$.
Therefore, the leading monomial of $u_\alpha(\bm{x})$ with respect to $\le_s$ is $\bm{x}^\alpha$, as required.
\end{proof}

\begin{thm}\label{thm:QuotientScheme}
Work with the same notation as above. 
Then the following hold:
\begin{enumerate}[label=(\roman*)]
\item \label{item:QuotientScheme_P}
The quotient scheme $\mathfrak{X}/\mathcal{C}$ is an $s$-variate $P$-polynomial association scheme on $\mathcal{D}_s$ with respect to $\le_s$, and its associated polynomials are $\{v^{\mathrm{qt}}_{\alpha}\}_{\alpha\in \mathcal{D}_s}$, where 
\[
v^{\mathrm{qt}}_{\alpha}(x_1,\ldots,x_s)
:=\frac{1}{p}u_{\alpha}\left(
\frac{k_{1}}{k_{\epsilon_1\mathcal{C}}}x_1,
\frac{k_{2}}{k_{\epsilon_2\mathcal{C}}}x_2,
\ldots ,
\frac{k_{s}}{k_{\epsilon_s\mathcal{C}}}x_s
\right)
\qquad (\alpha\in\mathcal{D}_s),
\]
and $k_{i}$ is the valency of $\epsilon_i$ in $\mathfrak{X}$ and $k_{\epsilon_i\mathcal{C}}$ is the valency of $\epsilon_i\mathcal{C}$ in $\mathfrak{X}/\mathcal{C}$, namely,
\[
k_{\epsilon_i\mathcal{C}}=\frac{1}{p}\sum_{(\epsilon_i,\beta)\in \epsilon_i\mathcal{C}} k_{(\epsilon_i,\beta)}.
\]

\item \label{item:QuotientScheme_Q}
Let $\mathcal{C}^\ast$ be the dual closed subset of $\mathcal{C}$ and let $\{v^\ast_{(\alpha,\beta)}\}_{(\alpha,\beta)\in\mathcal{D}^\ast}$ be the associated polynomials of 
the $\ell^\ast$-variate $Q$-polynomial structure of $\mathfrak{X}$ on $\mathcal{D}^\ast\subset \N^{s^\ast}\times \N^{\ell^\ast-s^\ast}$ with respect to a monomial order $\le^\ast$ of $s^\ast$-elimination type determined by Theorem~\ref{thm:EliminationThm}~\ref{item:EliminationThm_iii}. 
Let $\iota^\ast:\N^{\ell^\ast-s^\ast}\to\N^{s^\ast}\times \N^{\ell^\ast-s^\ast}$ be defined by $\iota^\ast(\beta):=(o,\beta)$, let 
$\mathcal{D}_{\mathrm{qt}}^\ast:={\iota^\ast}^{-1}(\mathcal{C}^\ast)$, and let 
$\le^\ast_{\mathrm{qt}}$ be the monomial order on $\N^{\ell^\ast-s^\ast}$ induced from $\le^\ast$ by $\iota^\ast$ (Lemma~\ref{lem:induced_monomial_order}). Then the quotient scheme $\mathfrak{X}/\mathcal{C}$ is an $(\ell^\ast-s^\ast)$-variate $Q$-polynomial association scheme on $\mathcal{D}_{\mathrm{qt}}^\ast$ with respect to $\le^\ast_{\mathrm{qt}}$, and its associated polynomials are $\{v^{\ast\mathrm{qt}}_\beta\}_{\beta\in \mathcal{D}_{\mathrm{qt}}^\ast}$, where 
\begin{equation}\label{eq:quotient_Qpoly_polynomial}
v^{\ast\mathrm{qt}}_{\beta}(x_1,\ldots,x_{\ell^\ast-s^\ast})
:=v^\ast_{\iota^\ast(\beta)}(\underbrace{0,\ldots,0}_{s^\ast},x_1,\ldots,x_{\ell^\ast-s^\ast})
\qquad (\beta\in\mathcal{D}_{\mathrm{qt}}^\ast). 
\end{equation}
\end{enumerate}
\end{thm}
\begin{proof}
\ref{item:QuotientScheme_P}:
We verify the criterion in Theorem~\ref{thm:eigen} for the quotient scheme. This shows that $\mathfrak{X}/\mathcal{C}$ is an $s$-variate $P$-polynomial association scheme on $\mathcal{D}_s\subset \N^s$ with respect to $\le_s$.
It is clear that $\mathcal{D}_s$ satisfies Definition~\ref{df:abPpoly}~\ref{item:abPpoly_i}.

Theorem~\ref{thm:eigen}~\ref{item:eigen_a}:
Let $\alpha\in \mathcal{D}_s$ and $\bm{\alpha}:=(\alpha,o)\mathcal{C} \in \mathcal{D}/\mathcal{C}$. 
For $j\in \mathcal{C}^\ast$, by \eqref{eq:quotient_P}, Theorem~\ref{thm:eigen} and $\theta_i(j)=P_{\epsilon_i}(j)=k_i$
($i=s+1,\ldots ,\ell$ and $j\in \mathcal{C}^\ast$),
we have 
\begin{align*}
P^{\mathrm{qt}}_{\bm{\alpha}}(j)
&=\frac{1}{p}\sum_{(\alpha,\beta)\in \bm{\alpha}}P_{(\alpha,\beta)}(j) 
=\frac{1}{p}\sum_{(\alpha,\beta)\in \bm{\alpha}} v_{(\alpha,\beta)}(\theta_{1}(j),\theta_{2}(j),\ldots ,\theta_{\ell}(j))\\
&=\frac{1}{p}\tilde{v}_{\alpha}(\theta_{1}(j),\theta_{2}(j),\ldots ,\theta_{s}(j),k_{s+1},\ldots ,k_{\ell}) 
=\frac{1}{p}u_{\alpha}(\theta_{1}(j),\theta_{2}(j),\ldots ,\theta_{s}(j)). 
\end{align*}
Considering \eqref{eq:quotient_scheme_P_and_valency},
we have
$\theta_{i}(j)=
\frac{k_{i}}{k_{\epsilon_i\mathcal{C}}}
P^{\mathrm{qt}}_{\epsilon_i\mathcal{C}}(j)$
for $i=1,2,\ldots ,s$ and $j\in \mathcal{C}^\ast$.
By letting $\theta'_i(j)=P^{\mathrm{qt}}_{\epsilon_i\mathcal{C}}(j)$, we obtain that 
$P^{\mathrm{qt}}_{\bm{\alpha}}(j)=
v^{\mathrm{qt}}_{\alpha}\bigl(\theta'_{1}(j),\theta'_{2}(j),\ldots,\theta'_{s}(j)\bigr)$. 

Theorem~\ref{thm:eigen}~\ref{item:eigen_b}:
Take $\alpha\in \mathcal{D}_s$ and $i=1,2,\ldots,s$ with $\alpha+\epsilon_i \notin \mathcal{D}_s$. Then $(\alpha+\epsilon_i,o)\notin \mathcal{D}$. Hence, by Theorem~\ref{thm:eigen}~\ref{item:eigen_b} for $\mathfrak{X}$, there exists a polynomial $w_{(\alpha+\epsilon_i,o)}(\bm{x})$ of multidegree $(\alpha+\epsilon_i,o)$ such that
\begin{equation}
\label{eq:hajikko_1}
w_{(\alpha+\epsilon_i,o)}(\theta_1(j),\theta_2(j),\ldots ,\theta_{\ell}(j))=0    
\end{equation}
for $j\in \mathcal{J}$.
By $\epsilon_i\in \mathcal{C}$ for $i=s+1,\ldots ,\ell$ and Lemma~\ref{lem:closed_subset}~\ref{lem:closed_subset:k}, we have
\[
\theta_i(j)=
\begin{cases}
 \frac{k_{i}}{k_{\epsilon_i\mathcal{C}}}\theta'_{i}(j) & \text{if $i=1,\ldots ,s$,}\\
 k_i & \text{if $i=s+1,\ldots ,\ell$}
\end{cases}
\]
for $j\in \mathcal{C}^\ast$.
Substituting this into \eqref{eq:hajikko_1}, we obtain
\begin{equation}
\label{eq:hajikko_2}
w_{(\alpha+\epsilon_i,o)}\left(\frac{k_{1}}{k_{\epsilon_1\mathcal{C}}}\theta'_{1}(j),\frac{k_{2}}{k_{\epsilon_2\mathcal{C}}}\theta'_{2}(j),\ldots ,\frac{k_{s}}{k_{\epsilon_s\mathcal{C}}}\theta'_{s}(j),k_{s+1},k_{s+2},\ldots ,k_{\ell}\right)=0.
\end{equation}
Since $w_{(\alpha+\epsilon_i,o)}$ has multidegree $(\alpha+\epsilon_i,o)$, the polynomial appearing in \eqref{eq:hajikko_2} has multidegree $\alpha+\epsilon_i$ with respect to $\le_s$. Therefore it satisfies the condition of Theorem~\ref{thm:eigen}~\ref{item:eigen_b}.

\ref{item:QuotientScheme_Q}:
We verify the criterion in Theorem~\ref{thm:eigenQ} for the quotient scheme. This shows that $\mathfrak{X}/\mathcal{C}$ is an $(\ell^\ast-s^\ast)$-variate $Q$-polynomial association scheme on $\mathcal{D}_{\mathrm{qt}}^\ast$ with respect to $\le^\ast_{\mathrm{qt}}$.
It is clear that $\mathcal{D}_{\mathrm{qt}}^\ast$ satisfies Definition~\ref{df:abQpoly}~\ref{item:abQpoly_i}.

Theorem~\ref{thm:eigenQ}~\ref{item:eigenQ_a}:
Let $\beta\in\mathcal{D}_{\mathrm{qt}}^\ast$, equivalently $\iota^\ast(\beta)=(o,\beta)\in\mathcal{C}^\ast$. 
By \eqref{eq:quotient_Q}, for any $\bm{i}\in\mathcal{I}/\mathcal{C}$ and $i\in \bm{i}$, we have
$Q^{\mathrm{qt}}_{\beta}(\bm{i})=Q_{\iota^\ast(\beta)}(i)$.
Moreover, by Theorem~\ref{thm:eigenQ}~\ref{item:eigenQ_a},
we have
$Q_{\iota^\ast(\beta)}(i)=v^\ast_{\iota^\ast(\beta)}\bigl(\theta^\ast_{1}(i),\ldots,\theta^\ast_{\ell^\ast}(i)\bigr)$.
Since $\le^\ast$ is of $s^\ast$-elimination type, any monomial appearing in $v^\ast_{\iota^\ast(\beta)}$ does not contain the variables $x_1,\ldots,x_{s^\ast}$. Thus,
\[
v^\ast_{\iota^\ast(\beta)}(x_1,\ldots,x_{\ell^\ast})=v^\ast_{\iota^\ast(\beta)}(\underbrace{0,\ldots,0}_{s^\ast},x_{s^\ast+1},\ldots,x_{\ell^\ast}).
\]
Furthermore, for $t=1,\ldots,\ell^\ast-s^\ast$, we have $\iota^\ast(\epsilon_t)=\epsilon_{s^\ast+t}\in\mathcal{C}^\ast$, and~\eqref{eq:quotient_Q} gives
\[
\theta^\ast_{s^\ast+t}(i)=Q_{\epsilon_{s^\ast+t}}(i)=Q^{\mathrm{qt}}_{\epsilon_t}(\bm{i})=\theta^{\mathrm{qt}\ast}_{t}(\bm{i}).
\]
Therefore,
\[
Q^{\mathrm{qt}}_{\beta}(\bm{i})
=v^\ast_{\iota^\ast(\beta)}(\underbrace{0,\ldots,0}_{s^\ast},\theta^{\mathrm{qt}\ast}_{1}(\bm{i}),\ldots,\theta^{\mathrm{qt}\ast}_{\ell^\ast-s^\ast}(\bm{i}))
=v^{\ast\mathrm{qt}}_{\beta}\bigl(\theta^{\mathrm{qt}\ast}_{1}(\bm{i}),\ldots,\theta^{\mathrm{qt}\ast}_{\ell^\ast-s^\ast}(\bm{i})\bigr).
\]

Theorem~\ref{thm:eigenQ}~\ref{item:eigenQ_b}:
Take $\beta\in \mathcal{D}_{\mathrm{qt}}^\ast$ and $t=1,2,\ldots ,\ell^\ast-s^\ast$ with $\beta+\epsilon_{t} \notin \mathcal{D}_{\mathrm{qt}}^\ast$. 
This implies that $(o,\beta)\in \mathcal{C}^\ast$ and $(o,\beta)+\epsilon_{s^\ast +t}\notin \mathcal{C}^\ast$.
By Theorem~\ref{thm:eigenQ}~\ref{item:eigenQ_b} for $\mathfrak{X}$, 
there exists a polynomial $w^\ast_{(o,\beta)+\epsilon_{s^\ast +t}}(\bm{x})$ of multidegree $(o,\beta)+\epsilon_{s^\ast +t}$
such that
\[w^\ast_{(o,\beta)+\epsilon_{s^\ast +t}}(\theta^\ast_{1}(i),\ldots,\theta^\ast_{\ell^\ast}(i))=0\]
for $i\in \mathcal{I}$.
Since $\le^\ast$ is of $s^\ast$-elimination type, 
any monomial appearing in $w^\ast_{(o,\beta)+\epsilon_{s^\ast +t}}$ does not contain the variables $x_1,\ldots,x_{s^\ast}$. 
Thus, 
\[w^\ast_{(o,\beta)+\epsilon_{s^\ast +t}}(x_1,\ldots,x_{\ell^\ast})=w^\ast_{(o,\beta)+\epsilon_{s^\ast +t}}(\underbrace{0,\ldots,0}_{s^\ast},x_{s^\ast+1},\ldots,x_{\ell^\ast}).\]
For $\bm{i}\in \mathcal{I}/\mathcal{C}$ and $i\in \bm{i}$, we have $\theta^\ast_{s^\ast+t}(i)=\theta^{\mathrm{qt}\ast}_{t}(\bm{i})$.
Substituting these into the above equation, we obtain
\begin{equation}
    \label{eq:hajikko_3}
    w^\ast_{(o,\beta)+\epsilon_{s^\ast +t}}(\underbrace{0,\ldots,0}_{s^\ast},\theta^{\mathrm{qt}\ast}_{1}(\bm{i}),\ldots,\theta^{\mathrm{qt}\ast}_{\ell^\ast-s^\ast}(\bm{i}))=0.
\end{equation}
The polynomial appearing in the above equation has multidegree $\beta+\epsilon_t$ with respect to $\le^\ast_{\mathrm{qt}}$, so it satisfies the condition of Theorem~\ref{thm:eigenQ}~\ref{item:eigenQ_b}.
\end{proof}

In view of the duality in \eqref{eq:block_scheme_P}, \eqref{eq:block_scheme_Q}, \eqref{eq:quotient_P}, and~\eqref{eq:quotient_Q}, analogous multivariate $P$- and $Q$-polynomial structure theorems also hold for block schemes.
The proof proceeds similarly to that of Theorem~\ref{thm:QuotientScheme}, so we omit the details to avoid unnecessary complexity.

\begin{thm}\label{thm:block_scheme_Ppoly}
Work with the same notation as Theorem~\ref{thm:QuotientScheme}, and fix $x_0\in X$.  
Then the following hold:
\begin{enumerate}[label=(\roman*)]
    \item \label{item:block_scheme_Ppoly_i}
    Let $\iota:\N^{\ell-s}\to \N^s\times \N^{\ell-s}$ be defined by $\iota(\beta):=(o,\beta)$, let $\mathcal{D}_{\mathrm{bl}}:=\iota^{-1}(\mathcal{C})$, and let $\le_{\mathrm{bl}}$ be the monomial order on $\N^{\ell-s}$ induced from $\le$ by $\iota$ (Lemma~\ref{lem:induced_monomial_order}). Then the block scheme $\mathfrak{X}_{x_0\mathcal{C}}$ is an $(\ell-s)$-variate $P$-polynomial association scheme on $\mathcal{D}_{\mathrm{bl}}$ with respect to $\le_{\mathrm{bl}}$, and its associated polynomials are $\{v^{\mathrm{bl}}_{\beta}\}_{\beta\in \mathcal{D}_{\mathrm{bl}}}$, where
    \[
    v^{\mathrm{bl}}_{\beta}(x_1,\ldots,x_{\ell-s})
    :=v_{\iota(\beta)}(\underbrace{0,\ldots,0}_{s},x_1,\ldots,x_{\ell-s})
    \qquad (\beta\in \mathcal{D}_{\mathrm{bl}}).
    \] 

\item \label{item:block_scheme_Ppoly_ii} Let $\{v^\ast_{(\alpha,\beta)}\}_{(\alpha,\beta)\in\mathcal{D}^\ast}$ be the associated polynomials of 
the $\ell^\ast$-variate $Q$-polynomial structure of $\mathfrak{X}$ on $\mathcal{D}^\ast\subset \N^{s^\ast}\times \N^{\ell^\ast-s^\ast}$ with respect to a monomial order $\le^\ast$ of $s^\ast$-block type.
Define $\jmath^\ast:\N^{s^\ast}\to\N^{s^\ast}\times\N^{\ell^\ast-s^\ast}$ by $\jmath^\ast(\alpha):=(\alpha,o)$, let
$\mathcal{D}_{\mathrm{bl}}^\ast:={\jmath^\ast}^{-1}(\mathcal{D}^\ast)$, and let $\le^\ast_{\mathrm{bl}}$ be the monomial order on $\N^{s^\ast}$ induced from $\le^\ast$ by $\jmath^\ast$ (Lemma~\ref{lem:induced_monomial_order}).
Then the block scheme $\mathfrak{X}_{x_0\mathcal{C}}$ is an $s^\ast$-variate $Q$-polynomial association scheme on $\mathcal{D}_{\mathrm{bl}}^\ast\subset \N^{s^\ast}$ with respect to $\le^\ast_{\mathrm{bl}}$, and its associated polynomials are $\{v^{\ast\mathrm{bl}}_{\alpha}\}_{\alpha\in \mathcal{D}_{\mathrm{bl}}^\ast}$, where
\[
v^{\ast\mathrm{bl}}_{\alpha}(x_1,\ldots,x_{s^\ast})
:=\frac{1}{q}\sum_{\beta\,:\,(\alpha,\beta)\in\mathcal{D}^\ast}
v^\ast_{(\alpha,\beta)}\left(
\frac{m_{1}}{m_{\epsilon_1\mathcal{C}^\ast}}x_1,
\frac{m_{2}}{m_{\epsilon_2\mathcal{C}^\ast}}x_2,
\ldots ,
\frac{m_{s^\ast}}{m_{\epsilon_{s^\ast}\mathcal{C}^\ast}}x_{s^\ast},
m_{\epsilon_{s^\ast+1}},\ldots,m_{\epsilon_{\ell^\ast}}
\right).
\]
Here $m_i$ denotes the multiplicity of $\epsilon_i\in \mathcal{D}^\ast$ in $\mathfrak{X}$ and $m_{\epsilon_i\mathcal{C}^\ast}$ denotes
the multiplicity of the primitive idempotent class
$\epsilon_i\mathcal{C}^\ast \in \mathcal{J}/\mathcal{C}^\ast$ in the block scheme $\mathfrak{X}_{x_0\mathcal{C}}$.
\end{enumerate}
\end{thm}


\begin{thm}\label{thm:ideal_block_quotient}
Work with the same notation as in Theorems~\ref{thm:QuotientScheme} and~\ref{thm:block_scheme_Ppoly}.
Let $\theta_i(j)=P_{\epsilon_i}(j)$ be the $(\epsilon_i,j)$-entry of the first eigenmatrix of $\mathfrak{X}$ for $i=1,\ldots,\ell$ and $j\in \mathcal{J}$.
Let $I\subset \C[x_1,\ldots,x_\ell]$ be the ideal associated with $\mathfrak{X}$, i.e., the ideal generated by the polynomials in \eqref{eq:ideal}. 
Let $I^{\mathrm{qt}}\subset \C[x_1,\ldots,x_s]$ (resp. $I^{\mathrm{bl}}\subset \C[x_{s+1},\ldots,x_\ell]$) be the ideal associated with the quotient scheme $\mathfrak{X}/\mathcal{C}$ (resp. block scheme $\mathfrak{X}_{x_0\mathcal{C}}$).
\begin{enumerate}[label=(\roman*)]
\item \label{item:ideal_block_quotient_i} Put 
$J^{\mathrm{qt}}:=(I+\langle x_{s+1}-k_{s+1},\ldots,x_\ell-k_\ell\rangle)\cap \C[x_1,\ldots,x_s]$.
Then
\begin{equation}
    \label{eq:ideal_quotient_elim}
    J^{\mathrm{qt}}
=\{f\in \C[x_1,\ldots,x_s] \mid f(\theta_1(j),\ldots,\theta_s(j))=0,\  j\in \mathcal{C}^\ast\}
\end{equation}
holds. Moreover, let $\sigma_{\mathrm{qt}}:\C[x_1,\ldots,x_s]\to \C[x_1,\ldots,x_s]$ be the algebra automorphism defined by
\[
(\sigma_{\mathrm{qt}}f)(x_1,\ldots,x_s)
:=
f\left(
\frac{k_{1}}{k_{\epsilon_1\mathcal{C}}}x_1,
\ldots,
\frac{k_{s}}{k_{\epsilon_s\mathcal{C}}}x_s
\right).
\]
Then we have
\begin{align}
I^{\mathrm{qt}}
&=\sigma_{\mathrm{qt}}(J^{\mathrm{qt}}) \notag\\
&=\Bigl\{f\in \C[x_1,\ldots,x_s] \mid 
f\left(\frac{k_{\epsilon_1\mathcal{C}}}{k_1}\theta_1(j),
\ldots,\frac{k_{\epsilon_s\mathcal{C}}}{k_s}\theta_s(j)\right)=0,\   
j\in \mathcal{C}^\ast\Bigr\}. 
\label{eq:ideal_quotient_rescaling}
\end{align}
In particular,
$\mathfrak{A}(\mathfrak{X}/\mathcal{C})
\simeq
\C[x_1,\ldots,x_s]/I^{\mathrm{qt}}
\simeq
\C[x_1,\ldots,x_s]/J^{\mathrm{qt}}$.

\item \label{item:ideal_block_quotient_ii}
We have 
\begin{equation}
    \label{eq:ideal_block_elim}
I^{\mathrm{bl}}=I\cap \C[x_{s+1},\ldots,x_\ell],
\end{equation}
that is, $I^{\mathrm{bl}}$ is the elimination ideal with respect to $\le$ of $s$-elimination type.
In particular, $\mathfrak{A}(\mathfrak{X}_{x_0\mathcal{C}})\ \simeq\ \C[x_{s+1},\ldots,x_{\ell}]/I^{\mathrm{bl}}$. 
\end{enumerate}
\end{thm}

\begin{proof}
For $j\in \mathcal{J}$, put
\[
\Theta(j):=(\theta_1(j),\theta_2(j),\ldots,\theta_\ell(j)).
\]
We first show that
\begin{equation}
\label{eq:Cstar_by_last_coordinates}
j\in \mathcal{C}^\ast
\iff
(\theta_{s+1}(j),\theta_{s+2}(j),\ldots,\theta_\ell(j))=(k_{s+1},k_{s+2},\ldots,k_\ell).
\end{equation}
The implication ``$\Rightarrow$'' follows from Lemma~\ref{lem:closed_subset}~\ref{lem:closed_subset:k}. 
For the converse, assume that
\[
(\theta_{s+1}(j),\theta_{s+2}(j),\ldots,\theta_\ell(j))=(k_{s+1},k_{s+2},\ldots,k_\ell)
=(\theta_{s+1}(j_0),\theta_{s+2}(j_0),\ldots,\theta_\ell(j_0)).
\]
Take $(o,\beta)\in \mathcal{C}=\mathcal{D}\cap (\{o\}\times \N^{\ell-s})$. 
Since $\le$ is of $s$-elimination type and $(o,\beta)\in \{o\}\times \N^{\ell-s}$, the associated polynomial $v_{(o,\beta)}(\bm{x})$ involves only $x_{s+1},\ldots,x_\ell$, and therefore
\[
P_{(o,\beta)}(j)
=v_{(o,\beta)}(\Theta(j))
=v_{(o,\beta)}(\Theta(j_0))
=P_{(o,\beta)}(j_0)
=k_{(o,\beta)}.
\]
Now, using $M=\frac{1}{p}\sum_{(o,\beta)\in \mathcal{C}}A_{(o,\beta)}=\sum_{h\in \mathcal{C}^\ast}E_h$
and $p=\sum_{(o,\beta)\in \mathcal{C}}k_{(o,\beta)}$, we obtain
\[
ME_j
=\frac{1}{p}\sum_{(o,\beta)\in \mathcal{C}}A_{(o,\beta)}E_j
=\frac{1}{p}\sum_{(o,\beta)\in \mathcal{C}}P_{(o,\beta)}(j)E_j
=\frac{1}{p}\sum_{(o,\beta)\in \mathcal{C}}k_{(o,\beta)}E_j
=E_j.
\]
Thus $j\in \mathcal{C}^\ast$, proving \eqref{eq:Cstar_by_last_coordinates}.

\medskip
\noindent
\ref{item:ideal_block_quotient_i}.
First we prove \eqref{eq:ideal_quotient_elim}.
Let us denote the right-hand side of \eqref{eq:ideal_quotient_elim} by $V^{\mathrm{qt}}$. That is,
\[
V^{\mathrm{qt}}
:=
\{f\in \C[x_1,\ldots,x_s] \mid f(\theta_1(j),\ldots,\theta_s(j))=0 \text{ for all } j\in \mathcal{C}^\ast\}.
\]

$J^{\mathrm{qt}}\subset V^{\mathrm{qt}}$:
Take $f\in J^{\mathrm{qt}}$. Then there exist $g\in I$ and $h_{s+1},\ldots,h_\ell\in \C[x_1,\ldots,x_\ell]$ such that
$f=g+\sum_{t=s+1}^{\ell} h_t(x_t-k_t)$.
For $j\in \mathcal{C}^\ast$, Corollary~\ref{cor:ideal} gives $g(\Theta(j))=0$, and~\eqref{eq:Cstar_by_last_coordinates} gives $\theta_t(j)=k_t$ for $t=s+1,\ldots,\ell$. 
Therefore
\[
f(\theta_1(j),\ldots,\theta_s(j))
=g(\Theta(j))+\sum_{t=s+1}^{\ell} h_t(\Theta(j))(\theta_t(j)-k_t)
=0,
\]
so $f\in V^{\mathrm{qt}}$.

$V^{\mathrm{qt}}\subset J^{\mathrm{qt}}$:
Take $f\in V^{\mathrm{qt}}$ and put
$T:=\{(\theta_{s+1}(j),\theta_{s+2}(j),\ldots,\theta_\ell(j)) \mid j\in \mathcal{J}\}\subset \C^{\ell-s}$.
By the definition of $T$, we have $\bm{k}:=(k_{s+1},k_{s+2},\ldots,k_\ell)\in T$.
By \eqref{eq:Cstar_by_last_coordinates}, the point $\bm{k}$ corresponds exactly to the indices in $\mathcal{C}^\ast$. 
For each $\bm{b}=(b_{s+1},\ldots,b_\ell)\in T\setminus\{\bm{k}\}$, choose $t(\bm{b})\in\{s+1,\ldots,\ell\}$ such that $b_{t(\bm{b})}\neq k_{t(\bm{b})}$, and define
\[
e(x_{s+1},\ldots,x_\ell)
:=
\prod_{\bm{b}\in T\setminus\{\bm{k}\}}
\frac{x_{t(\bm{b})}-b_{t(\bm{b})}}{k_{t(\bm{b})}-b_{t(\bm{b})}}
\in \C[x_{s+1},\ldots,x_\ell].
\]
Then $e(\bm{k})=1$ and $e(\bm{b})=0$ for all $\bm{b}\in T\setminus\{\bm{k}\}$, so by \eqref{eq:Cstar_by_last_coordinates} we have
\[
e(\theta_{s+1}(j),\ldots,\theta_\ell(j))=
\begin{cases}
1 & \text{if } j\in \mathcal{C}^\ast,\\
0 & \text{if } j\notin \mathcal{C}^\ast.
\end{cases}
\]
Now define
\[
F(\bm{x}):=f(x_1,\ldots,x_s)e(x_{s+1},\ldots,x_\ell)
\in \C[x_1,\ldots,x_\ell].
\]
If $j\in \mathcal{C}^\ast$, then $f(\theta_1(j),\ldots,\theta_s(j))=0$ by assumption. 
If $j\notin \mathcal{C}^\ast$, then $e(\theta_{s+1}(j),\ldots,\theta_\ell(j))=0$. 
Hence $F(\Theta(j))=0$ for all $j\in \mathcal{J}$. 
By Corollary~\ref{cor:ideal}, we obtain $F\in I$.
Since $e(\bm{k})=1$, the polynomial $e-1$ vanishes at $(k_{s+1},\ldots,k_\ell)$, and therefore
$e-1\in \langle x_{s+1}-k_{s+1},\ldots,x_\ell-k_\ell\rangle$.
Consequently,
\[
f
=F+f(1-e)
\in I+\langle x_{s+1}-k_{s+1},\ldots,x_\ell-k_\ell\rangle.
\]
Since $f\in \C[x_1,\ldots,x_s]$, it follows that $f\in J^{\mathrm{qt}}$. 
Thus $J^{\mathrm{qt}}=V^{\mathrm{qt}}$, proving \eqref{eq:ideal_quotient_elim}.

Next, for $i=1,2,\ldots,s$ and $j\in \mathcal{C}^\ast$, put $\theta'_i(j):=P_{\epsilon_i\mathcal{C}}(j)=\frac{k_{\epsilon_i\mathcal{C}}}{k_i}\theta_i(j)$,
as in the proof of Theorem~\ref{thm:QuotientScheme}~\ref{item:QuotientScheme_P}.
Applying Corollary~\ref{cor:ideal} to the quotient scheme $\mathfrak{X}/\mathcal{C}$, we obtain
\[
I^{\mathrm{qt}}
=
\{h\in \C[x_1,\ldots,x_s] \mid h(\theta'_1(j),\ldots,\theta'_s(j))=0 \text{ for all } j\in \mathcal{C}^\ast\}.
\]
On the other hand, for any $g\in \C[x_1,\ldots,x_s]$ and any $j\in \mathcal{C}^\ast$, we have
\[
(\sigma_{\mathrm{qt}}g)(\theta'_1(j),\ldots,\theta'_s(j))
=g\left(\frac{k_1}{k_{\epsilon_1\mathcal{C}}}\theta'_1(j),\frac{k_2}{k_{\epsilon_2\mathcal{C}}}\theta'_2(j),\ldots,\frac{k_s}{k_{\epsilon_s\mathcal{C}}}\theta'_s(j)\right)
=g(\theta_1(j),\ldots,\theta_s(j)).
\]
Therefore, $g\in J^{\mathrm{qt}}$ if and only if $\sigma_{\mathrm{qt}}g\in I^{\mathrm{qt}}$.
Hence $I^{\mathrm{qt}}=\sigma_{\mathrm{qt}}(J^{\mathrm{qt}})$, which proves \eqref{eq:ideal_quotient_rescaling}.

By Proposition~\ref{prop:A=C[x]/I}~\ref{item:prop_A_CxmodI_iii}, we have 
$\mathfrak{A}(\mathfrak{X}/\mathcal{C})
\simeq
\C[x_1,\ldots,x_s]/I^{\mathrm{qt}}$.
Moreover, the algebra automorphism $\sigma_{\mathrm{qt}}$ induces an algebra isomorphism
$\C[x_1,\ldots,x_s]/I^{\mathrm{qt}}
\simeq
\C[x_1,\ldots,x_s]/J^{\mathrm{qt}}$.

\medskip
\noindent
\ref{item:ideal_block_quotient_ii}.
Let
\[
V^{\mathrm{bl}}
:=
\{f\in \C[x_{s+1},\ldots,x_\ell] \mid f(\theta_{s+1}(j),\ldots,\theta_\ell(j))=0 \text{ for all } \bm{j}\in \mathcal{J}/\mathcal{C}^\ast,\ j\in \bm{j}\}.
\]
By \eqref{eq:block_scheme_P}, the values $\theta_{s+1}(j),\ldots,\theta_\ell(j)$ do not depend on the choice of representative $j\in \bm{j}$. 
Applying Corollary~\ref{cor:ideal} to the block scheme $\mathfrak{X}_{x_0\mathcal{C}}$, we obtain $I^{\mathrm{bl}}=V^{\mathrm{bl}}$.
We now prove that $I\cap \C[x_{s+1},\ldots,x_\ell]=V^{\mathrm{bl}}$.
If $f\in I\cap \C[x_{s+1},\ldots,x_\ell]$, then Corollary~\ref{cor:ideal} for $\mathfrak{X}$ gives
$f(\theta_{s+1}(j),\ldots,\theta_\ell(j))=0$ for $j\in \mathcal{J}$,
so in particular $f\in V^{\mathrm{bl}}$.
Conversely, take $f\in V^{\mathrm{bl}}$ and regard it as a polynomial in $\C[x_1,\ldots,x_\ell]$ independent of $x_1,\ldots,x_s$.
Since every $j\in \mathcal{J}$ belongs to some class $\bm{j}\in \mathcal{J}/\mathcal{C}^\ast$, the assumption implies
$f(\Theta(j))=f(\theta_{s+1}(j),\ldots,\theta_\ell(j))=0$
for $j\in \mathcal{J}$.
Hence Corollary~\ref{cor:ideal} yields $f\in I$, and therefore $f\in I\cap \C[x_{s+1},\ldots,x_\ell]$. 
Thus
\[
I^{\mathrm{bl}}=V^{\mathrm{bl}}=I\cap \C[x_{s+1},\ldots,x_\ell],
\]
which proves \eqref{eq:ideal_block_elim}. 
The final algebra isomorphism follows from Proposition~\ref{prop:A=C[x]/I}~\ref{item:prop_A_CxmodI_iii} applied to the block scheme.
\end{proof}

\section{Examples}\label{sec:examples}

\subsection{Commutative association schemes of class $2$}
It is well known that a commutative association scheme of class $2$ is imprimitive if and only if $\mathfrak{X}$ is the association scheme of a complete multipartite graph (cf.~\cite{BCN1989}). In this direction, the ``if part'' is trivial. We explain the ``only if part'' from the viewpoint of multivariate $P$-polynomial association schemes. Let $\mathfrak{X}=(X,\{A_0,A_1,A_2\})$ be an imprimitive commutative association scheme of class $2$.
Under the present definition, the notion of elimination-type order is vacuous in one variable, so we impose a bivariate $P$-polynomial structure on $\mathfrak{X}$ with respect to $\N^2$.
We relabel $A_0,A_1,A_2$ by $A_0=A_{00}, A_1=A_{10}, A_2=A_{01}$, so that $\mathcal{D}=\{o,(1,0),(0,1)\}\subset \N^2$.
By
\begin{align*}
    A_{10}^2 &= p^{00}_{10,10} A_{00} + p^{10}_{10,10} A_{10} + p^{01}_{10,10} A_{01},\\
    A_{10} A_{01} &= p^{10}_{10,01} A_{10} + p^{01}_{10,01} A_{01}, \;\text{ and}\\
    A_{01}^2 &= p^{00}_{01,01} A_{00} + p^{10}_{01,01} A_{10} + p^{01}_{01,01} A_{01},
\end{align*}
we obtain
\begin{align*}
    w_{20}(x,y)&=x^2 - p^{10}_{10,10} x - p^{01}_{10,10} y - p^{00}_{10,10},\\
    w_{11}(x,y)&=xy - p^{10}_{10,01} x - p^{01}_{10,01} y, \;\text{ and}\\
    w_{02}(x,y)&=y^2 - p^{10}_{01,01} x - p^{01}_{01,01} y - p^{00}_{01,01}.
\end{align*}
Hence $\mathfrak{X}$ is a bivariate $P$-polynomial association scheme with respect to a
$1$-elimination order on $\N^2$ if and only if either $p^{01}_{10,10}=0$ or $p^{10}_{01,01}=0$ holds.
Without loss of generality, assume $p^{10}_{01,01}=0$.
Then $\mathcal{C}=\{o,(0,1)\}=\mathcal{D}\cap (\{o\}\times \N)$
is the corresponding closed subset.
Since $\mathcal{R}^{-1}(\mathcal{C})$ is an equivalence relation on $X$,
$A_{01}$ is the adjacency matrix of a disjoint
union of complete graphs; write this graph as the disjoint union of $m$ copies of $K_r$.
The complementary relation $A_{10}$ is then the complete $m$-partite graph
with parts of size $r$.
Therefore, $\mathfrak{X}$ is the association scheme of the complete multipartite graph.

In this case, the above polynomials become
\begin{align*}
    w_{20}(x,y)&=x^2 -(m-2)r x - (m-1)r y - (m-1)r,\\
    w_{11}(x,y)&=xy - (r-1) x,\\
    w_{02}(x,y)&=y^2  - (r-2) y - (r-1).
\end{align*}
Let $I=\langle w_{20}(x,y),w_{11}(x,y),w_{02}(x,y)\rangle\subset \C[x,y]$
be the defining ideal of $\mathfrak{X}$.
By Corollary~\ref{cor:ideal}, $I$ is the vanishing ideal of the three eigenvalue points
$V=\{(r(m-1),r-1), (-r,r-1), (0,-1)\}$,
and hence the first eigenmatrix (with columns indexed by $A_{00},A_{10},A_{01}$) is
\[
P=
\begin{pmatrix}
1 & r(m-1) & r-1 \\
1 & -r & r-1 \\
1 & 0 & -1
\end{pmatrix}.
\]

\smallskip
\noindent
\emph{Block scheme:}
Fix $x_0\in X$. Since $\mathcal{C}$ records ``being in the same $K_r$-block'',
the block scheme $\mathfrak{X}_{x_0\mathcal{C}}$ is the one-class association scheme on $r$
vertices, namely the complete graph $K_r$.
Theorem~\ref{thm:block_scheme_Ppoly}~\ref{item:block_scheme_Ppoly_i} gives its associated polynomial as
$v^{\mathrm{bl}}_1(t)=v_{(0,1)}(0,t)=t$.
Moreover, Theorem~\ref{thm:ideal_block_quotient} yields
$I^{\mathrm{bl}}
=
I\cap \C[y]
=
\langle y^2-(r-2)y-(r-1)\rangle$,
which is exactly the defining ideal of the univariate $P$-polynomial structure of $K_r$.

\smallskip
\noindent
\emph{Quotient scheme:}
The quotient classes are the $m$ blocks, so the quotient scheme $\mathfrak{X}/\mathcal{C}$ is
the one-class association scheme on $m$ vertices, namely the complete graph $K_m$.
On the level of ideals,
$J^{\mathrm{qt}}
=
(I+\langle y-(r-1)\rangle)\cap \C[x]
=
\langle x^2-(m-2)r x-(m-1)r^2\rangle$ holds.
Since the original valency of $A_{10}$ is $k_{10}=r(m-1)$ and the quotient valency is
$k_{\epsilon_1\mathcal{C}}=m-1$, the rescaling automorphism in
Theorem~\ref{thm:ideal_block_quotient} is $\sigma_{\mathrm{qt}}(f)(t)=f(rt)$.
Accordingly,
\begin{align*}
    \frac{1}{r^2}w_{20}(r t,r-1)&=t^2-(m-2)t-(m-1),\\
    w_{11}(r t,r-1)&=w_{02}(r t,r-1)=0,
\end{align*}
and therefore
$I^{\mathrm{qt}}
=
\sigma_{\mathrm{qt}}(J^{\mathrm{qt}})
=
\langle t^2-(m-2)t-(m-1)\rangle\subset \C[t]$ holds.
This is exactly the defining ideal of the univariate $P$-polynomial structure of $K_m$.

\subsection{Nonbinary Johnson schemes}
Let $J_r(k,n)$ be the nonbinary Johnson scheme (assume $r\ge 3$ and $1\le k<n$) on
$X:=\{\bm{x}\in\{0,1,\ldots,r-1\}^n \mid w(\bm{x})=k\}$,
and let $\mathcal{R}\colon X\times X \to \mathcal{D}_{\mathrm{nbJ}}$, $\mathcal{R}(\bm{x},\bm{y})=(i,j)$ be defined by
\[
\mathcal{D}_{\mathrm{nbJ}}
:=
\{(i,j)\in \N^2 \mid 0\le i\le \min(k,n-k),\ 0\le j\le k-i\}
\]
and
\[
|\operatorname{supp}(\bm{x})\cap \operatorname{supp}(\bm{y})|=k-i,
\qquad
|\{t \mid x_t=y_t\neq 0\}|=k-i-j,
\]
where $\operatorname{supp}(\bm{x})=\{t \mid x_t \neq 0\}$ denotes the support of $\bm{x}$.
It is known that $J_r(k,n)$ is a bivariate $P$- and $Q$-polynomial association scheme with respect to a lexicographic order on $\N^2$ (that is, a $1$-elimination-type order),
and the associated bivariate polynomials of the bivariate $P$-polynomial structure are given by
\[
v_{(i,j)}(x,y)=(r-1)^i K_j(x,k-i,r-1) E_i(y,n-x,k-x),
\]
where $K_j$ is the Krawtchouk polynomial and $E_i$ is the Eberlein polynomial (\cite{bi,BKZZ2}).

Now consider
$\mathcal{C}
:=
\mathcal{D}_{\mathrm{nbJ}}\cap (\{o\}\times \N)
=
\{(0,j)\mid 0\le j\le k\}$.
The condition $i=0$ means $\operatorname{supp}(\bm{x})=\operatorname{supp}(\bm{y})$.
Hence $\mathcal{R}^{-1}(\mathcal{C})$ is exactly the equivalence relation
having the same support, so $\mathcal{C}$ is a nontrivial closed subset.
Therefore $J_r(k,n)$ is imprimitive, and the quotient-and-block formalism of
Section~\ref{subsec:quot} applies.

\smallskip
\noindent
\emph{Block scheme:}
Fix $\bm{x}_0\in X$ and write $S:=\operatorname{supp}(\bm{x}_0)\subset \{1,\ldots,n\}$.
Then the block $X_S:=\bm{x}_0\mathcal{C}$ consists of all words with support $S$,
so $|X_S|=(r-1)^k$.
For $\bm{x},\bm{y}\in X_S$, the relation $(0,j)$ records exactly the number of coordinates
of $S$ on which $\bm{x}$ and $\bm{y}$ differ.
Consequently the block scheme $\mathfrak{X}_{\bm{x}_0\mathcal{C}}$ is naturally isomorphic to the
Hamming scheme $H(k,r-1)$.
Theorem~\ref{thm:block_scheme_Ppoly}~\ref{item:block_scheme_Ppoly_i} recovers its univariate $P$-polynomials by
\[
v_j^{\mathrm{bl}}(t)
=
v_{(0,j)}(0,t)
=
K_j(t,k,r-1)
\qquad (0\le j\le k),
\]
since $E_0(\cdot,\cdot,\cdot)=1$.

\smallskip
\noindent
\emph{Quotient scheme:}
The quotient classes $X/\mathcal{C}$ are indexed by the $k$-subsets of $\{1,\ldots,n\}$.
If $\operatorname{supp}(\bm{x})=S$ and $\operatorname{supp}(\bm{y})=T$, then
$i=k-|S\cap T|$ depends only on $S$ and $T$.
Hence the quotient scheme $J_r(k,n)/\mathcal{C}$ is naturally isomorphic to the ordinary
Johnson scheme $J(n,k)$.
At the polynomial level, Theorem~\ref{thm:QuotientScheme}~\ref{item:QuotientScheme_P} starts from the sums
\[
\widetilde v_i(x_1,x_2):=\sum_{j=0}^{k-i} v_{(i,j)}(x_1,x_2)
\qquad
(0\le i\le \min(k,n-k)).
\]
After eliminating the block coordinate $x_2$ and applying the rescaling from Theorem~\ref{thm:QuotientScheme}~\ref{item:QuotientScheme_P}, one obtains the usual univariate $P$-polynomials of the Johnson scheme, namely the Eberlein polynomials
$E_i(\,\cdot\,,n,k)$ for $0\le i\le \min(k,n-k)$.

\subsection{Imprimitive distance-regular graphs}
Let $\Gamma$ be a distance-regular graph of diameter $d \geq 2$, and let $\mathfrak{X}=(X,\{A_i\}_{i=0}^d)$,
where $A_i$ is the distance-$i$ matrix of $\Gamma$,
be the associated distance scheme. 
By the three-term recurrence for distance-regular graphs 
\[
A_1A_i=b_{i-1}A_{i-1}+a_iA_i+c_{i+1}A_{i+1}
\;\; (0\le i\le d), 
\]
where $b_{-1}=c_{d+1}=0$, we introduce a system of polynomials $\{v_i(x)\}_{i=0}^d$ in one variable (i.e., a sequence of distance polynomials) as follows:  
\[
v_0(x)=1, \;\; v_1(x)=x, \;\; 
c_{i+1}v_{i+1}(x)=xv_i(x)-a_iv_i(x)-b_{i-1}v_{i-1}(x)\ (1\le i\le d-1). 
\]
Then $A_i=v_i(A_1)$ holds for $0\le i\le d$. This is a system of $P$-polynomials in one variable associated with $\Gamma$. 
In particular, we have 
\begin{equation}
\label{eq:DRG_v2}
v_2(x)=\frac{x^2-a_1x-k}{c_2}, 
\end{equation}
where $k:=b_0=k_1$. 

It is known that the distance scheme $\mathfrak{X}$ is imprimitive if and only if $\Gamma$ is either \emph{bipartite} or \emph{antipodal} (cf.~\cite{BCN1989}).
In what follows, we rephrase those properties in terms of monomial orders and describe an explicit system of polynomials in two variables induced by $\{v_i\}$ and associated with a bivariate $P$-polynomial structure. 
Our point is that the closed subset appearing in each case recovers,
through Theorems~\ref{thm:QuotientScheme} and~\ref{thm:block_scheme_Ppoly},
the familiar halved/folded constructions together with the corresponding univariate distance polynomials.

\begin{enumerate}[label=$(\roman*)$]
\item \textbf{The bipartite case.}
Assume that $\Gamma$ is bipartite. Then
\[
\mathcal{C}_{\mathrm{even}}
:=
\{0,2,4,\ldots,2\lfloor d/2\rfloor\}
\subset \{0,1,\ldots,d\}
\]
is a closed subset, and $\mathcal{R}^{-1}(\mathcal{C}_{\mathrm{even}})$ is the equivalence relation ``belonging to the same side of the bipartition''. Put
$m:=\Bigl\lfloor \frac{d}{2}\Bigr\rfloor$ and $m':=\Bigl\lfloor \frac{d-1}{2}\Bigr\rfloor$.
Now set
\[
\mathcal{D}_{\mathrm{bi}}
:=
\{(0,j)\mid 0\le j\le m\}
\ \sqcup\
\{(1,j)\mid 0\le j\le m'\}
\subset \N^2
\]
and relabel the distance matrices by
\[
A_{(0,j)}:=A_{2j}\ (0\le j\le m),
\qquad
A_{(1,j)}:=A_{2j+1}\ (0\le j\le m').
\]
In particular, $A_{\epsilon_1}=A_{(1,0)}=A_1$ and $A_{\epsilon_2}=A_{(0,1)}=A_2$.
Since $\Gamma$ is bipartite, we have $a_i=0$ for all $0\le i\le d$. Hence the univariate distance polynomials satisfy
\[
v_i(-x)=(-1)^i v_i(x)
\qquad (0\le i\le d),
\]
so for each $j$ there exist univariate polynomials $f_j(t),g_j(t)\in\C[t]$ such that
\begin{equation}
\label{eq:bipartite_evenodd}
v_{2j}(x)=f_j(v_2(x)),
\qquad
v_{2j+1}(x)=x\,g_j(v_2(x)).
\end{equation}
Here $v_2$ is given by \eqref{eq:DRG_v2}. More concretely, \eqref{eq:bipartite_evenodd} is determined recursively by
$f_0(t)=1$,
$g_0(t)=1$,
$f_1(t)=t$
and for $j\ge 1$,
\[
(k+c_2t)g_{j-1}(t)=b_{2j-2}f_{j-1}(t)+c_{2j}f_j(t),
\qquad
f_j(t)=b_{2j-1}g_{j-1}(t)+c_{2j+1}g_j(t).
\]
These are exactly the even/odd parts of the three-term recurrence under the condition $a_i=0$.

Define
\[
v_{(0,j)}(x_1,x_2):=f_j(x_2),
\qquad
v_{(1,j)}(x_1,x_2):=x_1\,g_j(x_2)
\qquad (0\le j\le m \text{ or } m').
\]
Since $A_2=v_2(A_1)$, we obtain
$A_{(0,j)}=v_{(0,j)}(A_1,A_2)$
and
$A_{(1,j)}=v_{(1,j)}(A_1,A_2)$.
Thus the original univariate polynomials $\{v_i\}$ produce a concrete bivariate $P$-polynomial family $\{v_{(0,j)},v_{(1,j)}\}$ in the bipartite case.

Let $\le$ be the lexicographic order with $x_1>x_2$. This order is of $1$-block type, and
\[
\mathcal{D}_{\mathrm{bi}}\cap (\{o\}\times \N)=\{(0,j)\mid 0\le j\le m\}
\]
corresponds exactly to $\mathcal{C}_{\mathrm{even}}$. 
One checks directly that $\mathfrak{X}$ becomes a bivariate $P$-polynomial association scheme with respect to $\le$.

\noindent\emph{Block scheme and quotient scheme.}
Fix $x_0\in X$ and put $X_0:=x_0\mathcal{C}_{\mathrm{even}}$. Then $X_0$ is one side of the bipartition. The block scheme
$\mathfrak{X}_{x_0\mathcal{C}_{\mathrm{even}}}
=
\bigl(X_0,\{R_{2j}\cap (X_0\times X_0)\}_{j=0}^m\bigr)$
is the distance scheme of the halved graph on $X_0$.
In particular, $\{v_{(0,j)}=f_j\}_{j=0}^m$ is the univariate $P$-polynomial system of the block scheme. 

On the other hand, the quotient scheme $\mathfrak{X}/\mathcal{C}_{\mathrm{even}}$ has two equivalence classes, namely the two sides of the bipartition. Hence $X/\mathcal{C}_{\mathrm{even}}$ is a $2$-point set.
Therefore $\mathfrak{X}/\mathcal{C}_{\mathrm{even}}$ is the trivial class-$1$ scheme on two vertices, i.e., the complete graph $K_2$. Its univariate $P$-polynomial system is just $u_0(x)=1$ and $u_1(x)=x$.

\item \textbf{The antipodal case.}
Assume that $\Gamma$ is antipodal, that is, distance $d$ defines an equivalence relation. Then
$\mathcal{C}_{\mathrm{anti}}:=\{0,d\}$
is a closed subset and $\mathcal{R}^{-1}(\mathcal{C}_{\mathrm{anti}})$ is the equivalence relation given by antipodal classes. 
Again put $m:=\Bigl\lfloor \frac{d}{2}\Bigr\rfloor$ and $m':=\Bigl\lfloor \frac{d-1}{2}\Bigr\rfloor$.
Now define
\[
\mathcal{D}_{\mathrm{anti}}
:=
\{(j,0)\mid 0\le j\le m\}
\ \sqcup\
\{(j,1)\mid 0\le j\le m'\}
\subset \N^2
\]
and relabel the distance matrices by
\[
A_{(j,0)}:=A_j\ (0\le j\le m),
\qquad
A_{(j,1)}:=A_{d-j}\ (0\le j\le m').
\]
In particular, $A_{\epsilon_1}=A_{(1,0)}=A_1$ and $A_{\epsilon_2}=A_{(0,1)}=A_d$.
With the same $1$-block lexicographic order $\le$ as above, we have
\[
\mathcal{D}_{\mathrm{anti}}\cap (\{o\}\times \N)=\{o,(0,1)\},
\]
which corresponds to $\mathcal{C}_{\mathrm{anti}}$. 
Using the original univariate distance polynomials, we first set
$v_{(j,0)}(x_1,x_2):=v_j(x_1)$ for $0\le j\le m$,
so that $A_{(j,0)}=v_{(j,0)}(A_1,A_d)$. To recover $A_{(j,1)}=A_{d-j}$ from $(A_1,A_d)$, apply the three-term recurrence backwards:
\begin{align*}
A_1A_{d}&=b_{d-1}A_{d-1}+a_dA_d,\\
A_1A_{d-j+1}&=b_{d-j}A_{d-j}+a_{d-j+1}A_{d-j+1}+c_{d-j+2}A_{d-j+2}
\quad (j\ge 1).
\end{align*}
This suggests the recursive definition
$v_{(0,1)}(x_1,x_2):=x_2$ and
$v_{(-1,1)}(x_1,x_2):=0$
and for $j\ge 1$,
\begin{equation}
\label{eq:antipodal_rec}
v_{(j,1)}(x_1,x_2)
:=
\frac{1}{b_{d-j}}
\Bigl(
 x_1v_{(j-1,1)}(x_1,x_2)
-a_{d-j+1}v_{(j-1,1)}(x_1,x_2)
-c_{d-j+2}v_{(j-2,1)}(x_1,x_2)
\Bigr).
\end{equation}
Then, inductively, $A_{(j,1)}=A_{d-j}=v_{(j,1)}(A_1,A_d)$ for $0\le j\le m'$.
Moreover, \eqref{eq:antipodal_rec} shows that the leading monomial of $v_{(j,1)}$ is $x_1^j x_2$. 
One checks directly that $\mathfrak{X}$ becomes a bivariate $P$-polynomial association scheme with respect to $\le$.

\noindent\emph{Block scheme and quotient scheme.}
Fix $x_0\in X$, let $X_0:=x_0\mathcal{C}_{\mathrm{anti}}$, and write $r:=|X_0|$. Then $X_0$ is the antipodal class containing $x_0$.
The corresponding block scheme is the class-$1$ scheme given by the complete graph $K_r$, whose univariate $P$-polynomial system is $w_0(x)=1$ and $w_1(x)=x$.

The quotient scheme $\mathfrak{X}/\mathcal{C}_{\mathrm{anti}}$ is the antipodal quotient, or folded graph. For $1\le i<m$ one has $i\mathcal{C}_{\mathrm{anti}}=\{(i,0),(i,1)\}=\{i,d-i\}$, so the quotient adjacency matrices are
$A_{\bm{i}}=A_{(i,0)}+A_{(i,1)}=A_i+A_{d-i}$ for $1\le i<m$,
and if $d$ is even then the middle class satisfies $A_{\bm{d/2}}=A_{(d/2,0)}=A_{d/2}$.
Hence the quotient scheme is a distance scheme of class $m$, so it has a univariate $P$-polynomial system $\{v^{\mathrm{qt}}_i\}_{i=0}^m$.
\end{enumerate}

\section{Multivariate polynomial association schemes, direct products, and crested products}
\label{sec:multpoly_products}
\subsection{Direct products and crested products}
In this subsection, we study two basic product constructions from the viewpoint of multivariate polynomiality: the \emph{direct product} and the \emph{crested product}.
On the $P$-side, we fix multivariate polynomial structures on the factors together with monomial orders adapted to closed subsets, and show that the crested product itself naturally carries an $(\ell_1+\ell_2)$-variate $P$-polynomial structure.
Analogous statements hold on the $Q$-side: multivariate $Q$-polynomial structures compatible with dual closed subsets also exhibit additivity of variables.

The \emph{direct product} of two association schemes $\mathfrak{X}^{(r)}=(X^{(r)},\mathcal{R}^{(r)},\mathcal{I}_r)$ $(r=1,2)$ is defined as follows.
Let $X:=X^{(1)}\times X^{(2)}$ and define a relation $\mathcal{R}$ on $X\times X$ by
\[
\mathcal{R}\bigl((x_1,x_2),(y_1,y_2)\bigr)
=
\bigl(\mathcal{R}^{(1)}(x_1,y_1),\mathcal{R}^{(2)}(x_2,y_2)\bigr)
\]
for
$x_1,y_1\in X^{(1)}$ and $x_2,y_2\in X^{(2)}$.
Then $\mathfrak{X}^{(1)}\otimes\mathfrak{X}^{(2)}:=(X,\mathcal{R},\mathcal{I}_1\times\mathcal{I}_2)$ is a commutative association scheme.
The adjacency matrices of the direct product are given by $A_{(i,j)} = A^{(1)}_i \otimes A^{(2)}_j$
for $(i,j)\in \mathcal{I}_1\times\mathcal{I}_2$,
where 
$A^{(1)}_i$ and $A^{(2)}_j$ denote the adjacency matrices of $\mathfrak{X}^{(1)}$ and $\mathfrak{X}^{(2)}$ respectively
and
$\otimes$ denotes the Kronecker (tensor) product of matrices.

Let $r=1,2$, and write
$\mathfrak{X}^{(r)}=(X^{(r)},\{A^{(r)}_{\xi}\}_{\xi\in\mathcal{D}_r})$
for an $\ell_r$-variate $P$-polynomial association scheme on $\mathcal{D}_r\subset\N^{\ell_r}$,
and let $\le_r$ denote its monomial order.
Also define an order $\le_{\otimes}$ on $\N^{\ell_1+\ell_2}=\N^{\ell_1}\times\N^{\ell_2}$ by
\begin{equation}
\label{eq:product_order_multpoly}
(\alpha,\gamma)\le_{\otimes}(\beta,\delta)
\iff
\bigl(\alpha<_1\beta\bigr)
\ \text{or}\ 
\bigl(\alpha=\beta\ \text{and}\ \gamma\le_2\delta\bigr)
\qquad
(\alpha,\beta\in\N^{\ell_1},\ \gamma,\delta\in\N^{\ell_2})
\end{equation}
This is a monomial order.
For direct products, multivariate $P$-polynomial structures lift directly with respect to the product order.

\begin{prop}
\label{prop:direct_prod_multivar}
Under the notation above, the direct product $\mathfrak{X}^{(1)}\otimes \mathfrak{X}^{(2)}$ is
an $(\ell_1+\ell_2)$-variate $P$-polynomial association scheme on $\mathcal{D}_{\otimes}:=\mathcal{D}_1\times\mathcal{D}_2\subset\N^{\ell_1+\ell_2}$,
with monomial order $\le_{\otimes}$ given by \eqref{eq:product_order_multpoly}.
\end{prop}

\begin{proof}
Label adjacency matrices of $\mathfrak{X}^{(1)}\otimes \mathfrak{X}^{(2)}$ by
$A_{(\alpha,\gamma)}:=A^{(1)}_{\alpha}\otimes A^{(2)}_{\gamma}$ for
$(\alpha,\gamma)\in\mathcal{D}_{\otimes}$
and apply Proposition~\ref{prop:P-TFAE}.

First, it is clear that $\mathcal{D}_{\otimes}$ satisfies Definition~\ref{df:abPpoly}~\ref{item:abPpoly_i}.
Next, for $1\le i\le \ell_1$ and $(\alpha,\gamma)\in\mathcal{D}_{\otimes}$, we have
$A_{(\epsilon_i,o)}A_{(\alpha,\gamma)}
=(A^{(1)}_{\epsilon_i}A^{(1)}_{\alpha})\otimes A^{(2)}_{\gamma}$.
Applying Lemma~\ref{lem:p_support}~\ref{item:p_support_ii} to $\mathfrak{X}^{(1)}$,
every term $A^{(1)}_{\alpha'}\otimes A^{(2)}_{\gamma}$ appearing on the right-hand side
satisfies $\alpha'\le_1 \alpha+\epsilon_i$, hence its index $(\alpha',\gamma)$ is at most
$(\alpha+\epsilon_i,\gamma)$.
Moreover, if $(\alpha+\epsilon_i,\gamma)\in\mathcal{D}_{\otimes}$, then
$p^{\alpha+\epsilon_i}_{\epsilon_i,\alpha}\neq0$, so the coefficient of
$A_{(\alpha+\epsilon_i,\gamma)}$ is nonzero.

Similarly, for $1\le j\le \ell_2$ and $(\alpha,\gamma)\in\mathcal{D}_{\otimes}$, we have
$A_{(o,\epsilon_j)}A_{(\alpha,\gamma)}
=A^{(1)}_{\alpha}\otimes (A^{(2)}_{\epsilon_j}A^{(2)}_{\gamma})$
and applying Lemma~\ref{lem:p_support}~\ref{item:p_support_ii} to $\mathfrak{X}^{(2)}$,
all indices appearing on the right are at most $(\alpha,\gamma+\epsilon_j)$,
with nonzero top coefficient whenever $(\alpha,\gamma+\epsilon_j)\in\mathcal{D}_{\otimes}$.

Therefore, by Proposition~\ref{prop:P-TFAE},
$\mathfrak{X}^{(1)}\otimes \mathfrak{X}^{(2)}$ is an $(\ell_1+\ell_2)$-variate
$P$-polynomial association scheme on $\mathcal{D}_{\otimes}$.
\end{proof}

For direct products of univariate $P$- or $Q$-polynomial association schemes, \cite{BKZZ} shows that multivariate polynomial structures exist for arbitrary monomial orders.
We will later use Proposition~\ref{prop:direct_prod_P/Q-poly} in exactly that form.
Here, $\mathcal{D}\subset \N^\ell$ is said to be of \emph{rectangular type} if there exist natural numbers $d_1,\ldots,d_\ell$ such that $\mathcal{D}=\{(\alpha_1,\ldots,\alpha_\ell)\in\N^\ell \mid \alpha_i\le d_i \text{ for all } i\}$.

\begin{prop}[cf.~\cite{BKZZ}]
\label{prop:direct_prod_P/Q-poly}
    The following hold:
    \begin{enumerate}[label=$(\roman*)$]
    \item \label{item:direct_prod_P} if $\{\mathfrak{X}^{(k)}\}^\ell_{k=1}$ are $P$-polynomial,
    then $\bigotimes^{\ell}_{k=1}\mathfrak{X}^{(k)}$ is an $\ell$-variate $P$-polynomial association scheme
    on a domain $\mathcal{D}$ of rectangular type with respect to any monomial order $\le$;
    \item \label{item:direct_prod_Q} if $\{\mathfrak{X}^{(k)}\}^{\ell^\ast}_{k=1}$ are $Q$-polynomial,
    then $\bigotimes^{\ell^\ast}_{k=1}\mathfrak{X}^{(k)}$ is an $\ell^\ast$-variate $Q$-polynomial association scheme
    on a domain $\mathcal{D}^\ast$ of rectangular type  with respect to any monomial order $\le$.
    \end{enumerate}
\end{prop}

On the other hand, the crested product was introduced by Bailey--Cameron~\cite{baileycameron2005}.
It contains both the direct product and the wreath product (``nesting'') as special cases.
We do not discuss the wreath product here; see \cite{baileycameron2005} for details.
Moreover, crested products are always imprimitive, so they fit well with the main results of this paper.
We first recall the definition.

\begin{df}[Crested Product]\label{df:crested_product}
Let $\mathfrak{X}^{(r)}=(X^{(r)},\mathcal{R}^{(r)},\mathcal{I}_r)$ be commutative association schemes ($r=1,2$),
with adjacency matrices $\{A^{(r)}_i\}_{i\in\mathcal{I}_r}$.
Take closed subsets $\mathcal{C}_r\subset\mathcal{I}_r$ ($r=1,2$).
Define the family of $01$-matrices $\{A_k\}_{k\in\mathcal{K}}$ on $X:=X^{(1)}\times X^{(2)}$ by
\begin{align*}
  &\mathcal{K}:=(\mathcal{C}_1\times \mathcal{I}_2)\ \sqcup\ \bigl((\mathcal{I}_1\setminus\mathcal{C}_1)\times (\mathcal{I}_2/\mathcal{C}_2)\bigr),\\
  &A_{(i,j)}:=A^{(1)}_i\otimes A^{(2)}_j\qquad (i\in\mathcal{C}_1,\ j\in\mathcal{I}_2),\\
  &A_{(i,\bm{j})}:=A^{(1)}_i\otimes A^{(2)}_{\bm{j}}\qquad (i\in\mathcal{I}_1\setminus\mathcal{C}_1,\ \bm{j}\in\mathcal{I}_2/\mathcal{C}_2),
\end{align*}
where for $\bm{j}\in\mathcal{I}_2/\mathcal{C}_2$ we set
$A^{(2)}_{\bm{j}}:=\sum_{j\in\bm{j}}A^{(2)}_j$ (see Section~\ref{sec:ImprimitiveAS}).
The association scheme generated by $\{A_k\}_{k\in\mathcal{K}}$ is called
the \emph{crested product of $\mathfrak{X}^{(1)}$ and $\mathfrak{X}^{(2)}$
with respect to the closed subsets $\mathcal{C}_1,\mathcal{C}_2$}
(see Bailey--Cameron~\cite{baileycameron2005}).
\end{df}

\begin{rem}\label{rem:crested_special_cases}
Definition~\ref{df:crested_product} includes the direct and wreath products in the following sense:
\begin{enumerate}[label=$(\roman*)$]
\item If $\mathcal{C}_1=\mathcal{I}_1$ or $\mathcal{C}_2=\{0\}$
(equivalently, $\mathfrak{X}^{(2)}/\mathcal{C}_2\cong \mathfrak{X}^{(2)}$),
the crested product coincides with the direct product $\mathfrak{X}^{(1)}\otimes \mathfrak{X}^{(2)}$.
\item If $\mathcal{C}_1=\{0\}$ and $\mathcal{C}_2=\mathcal{I}_2$
(equivalently, $\mathfrak{X}^{(2)}/\mathcal{C}_2$ is the trivial scheme),
the crested product coincides with the wreath product (nesting).
\end{enumerate}
\end{rem}

From now on, assume further that $\le_1$ is of $s_1$-elimination type and $\le_2$ is of $s_2$-block type, and that, for $r=1,2$, under the given labeling,
\[
\mathcal{D}_r\subset\N^{s_r}\times\N^{\ell_r-s_r}
\qquad\text{and}\qquad
\mathcal{C}_r=\mathcal{D}_r\cap (\{o\}\times \N^{\ell_r-s_r}).
\]
Here $s_r=0$ or $s_r=\ell_r$ is also allowed; in these cases, $\le_r$ can be regarded as an ordinary monomial order.
Also set
$\jmath_2:\N^{s_2}\to\N^{s_2}\times\N^{\ell_2-s_2}$,
$\jmath_2(\gamma):=(\gamma,o)$,
$\mathcal{D}_{2,\mathrm{qt}}:=\jmath_2^{-1}(\mathcal{D}_2)$
and let $\le_{2,\mathrm{qt}}$ be the monomial order on $\N^{s_2}$ induced from $\le_2$ (Lemma~\ref{lem:induced_monomial_order}).
Since $\le_2$ is of $s_2$-block type, equivalence classes in the quotient of $\mathfrak{X}^{(2)}$ by $\mathcal{C}_2$
are determined by the first $s_2$ components. For each $\gamma\in\mathcal{D}_{2,\mathrm{qt}}$, write
$\bm{\gamma}:=(\gamma,o)\mathcal{C}_2\in \mathcal{D}_2/\mathcal{C}_2$ and
$A^{(2)}_{\bm{\gamma}}:=\sum_{(\gamma,\delta)\in\bm{\gamma}}A^{(2)}_{(\gamma,\delta)}$
for the corresponding class and matrix.
Let $p_2:=\sum_{\xi\in\mathcal{C}_2}k^{(2)}_{\xi}$
and
$M_{\mathcal{C}_2}:=\frac{1}{p_2}\sum_{\xi\in\mathcal{C}_2}A^{(2)}_{\xi}=\frac{1}{p_2}A^{(2)}_{\bm{o}}$,
where $k^{(2)}_{\xi}$ denotes the valency in $\mathfrak{X}^{(2)}$.
Then the index set of the crested product is
\begin{equation}
\label{eq:crested_domain_multivar}
\mathcal{D}_{\mathrm{cr}}
:=
\mathcal{D}^1_{\mathrm{cr}}\sqcup\mathcal{D}^2_{\mathrm{cr}}
\subset\N^{\ell_1+\ell_2},
\end{equation}
where 
\begin{align*}
\mathcal{D}^1_{\mathrm{cr}}&:=\{(o,\beta,\gamma,\delta)\mid (o,\beta)\in\mathcal{C}_1,\ (\gamma,\delta)\in\mathcal{D}_2\},\\
\mathcal{D}^2_{\mathrm{cr}}&:=\{(\alpha,\beta,\gamma,o)\mid (\alpha,\beta)\in\mathcal{D}_1,\ \alpha \neq o,\ \gamma\in\mathcal{D}_{2,\mathrm{qt}}\},   
\end{align*}
and the corresponding adjacency matrices are
\begin{align}
B_{(o,\beta,\gamma,\delta)}&:=A^{(1)}_{(o,\beta)}\otimes A^{(2)}_{(\gamma,\delta)}
\qquad \text{for}\ (o,\beta,\gamma,\delta)\in\mathcal{D}^1_{\mathrm{cr}},
\label{eq:crested_basis_first}\\
B_{(\alpha,\beta,\gamma,o)}&:=A^{(1)}_{(\alpha,\beta)}\otimes A^{(2)}_{\bm{\gamma}}
\qquad \text{for}\ (\alpha,\beta,\gamma,o)\in\mathcal{D}^2_{\mathrm{cr}}.
\label{eq:crested_basis_second}
\end{align}

\begin{thm}
\label{thm:crested_multivar}
Under the assumptions above,
the crested product of $\mathfrak{X}^{(1)}$ and $\mathfrak{X}^{(2)}$ with respect to $(\mathcal{C}_1,\mathcal{C}_2)$
is an $(\ell_1+\ell_2)$-variate $P$-polynomial association scheme on
\eqref{eq:crested_domain_multivar}, and its monomial order is
$\le_{\otimes}$ defined by \eqref{eq:product_order_multpoly}.
\end{thm}

\begin{proof}
Let $\mathfrak{X}$ be this crested product, with labeling
\eqref{eq:crested_basis_first}, \eqref{eq:crested_basis_second}.
We again use Proposition~\ref{prop:P-TFAE}.

First, we show that $\mathcal{D}_{\mathrm{cr}}$ satisfies
Definition~\ref{df:abPpoly}~\ref{item:abPpoly_i}.
Take $(\alpha,\beta,\gamma,\delta)\in\mathcal{D}_{\mathrm{cr}}$ and assume
$o\le \alpha'\le \alpha$, $o\le \beta'\le \beta$, $o\le \gamma'\le \gamma$, and $o\le \delta'\le \delta$
(componentwise inequalities).
If $\alpha=o$, then $\alpha'=o$, and since $(o,\beta')\in\mathcal{C}_1$ and
$(\gamma',\delta')\in\mathcal{D}_2$, we have
$(o,\beta',\gamma',\delta')\in\mathcal{D}^1_{\mathrm{cr}}$.
If $\alpha\ne o$, then $\delta=o$, hence $\delta'=o$.
Also, $(\gamma,o)\in\mathcal{D}_2$ implies $(\gamma',o)\in\mathcal{D}_2$,
namely $\gamma'\in\mathcal{D}_{2,\mathrm{qt}}$.
Therefore, when $\alpha'\ne o$ we have $(\alpha',\beta',\gamma',o)\in\mathcal{D}^2_{\mathrm{cr}}$.
Hence Definition~\ref{df:abPpoly}~\ref{item:abPpoly_i} holds.

Next, the $\ell_1+\ell_2$ generating relations corresponding to $\mathcal{D}_{\mathrm{cr}}$ are
\begin{align*}
B_{(\epsilon_i,o,o,o)}&=A^{(1)}_{(\epsilon_i,o)}\otimes A^{(2)}_{\bm{o}}
&& (1\le i\le s_1),\\
B_{(o,\epsilon_j,o,o)}&=A^{(1)}_{(o,\epsilon_j)}\otimes I
&& (1\le j\le \ell_1-s_1),\\
B_{(o,o,\epsilon_u,o)}&=I\otimes A^{(2)}_{(\epsilon_u,o)}
&& (1\le u\le s_2),\\
B_{(o,o,o,\epsilon_v)}&=I\otimes A^{(2)}_{(o,\epsilon_v)}
&& (1\le v\le \ell_2-s_2)
\end{align*}
We verify the dominance conditions for products by these generators.

For $(\alpha,\beta,\gamma,\delta)\in\mathcal{D}_{\mathrm{cr}}$, we compute products with
the generators above and check the non-vanishing conditions of intersection numbers in $\mathfrak{X}$.

\smallskip
\noindent
(a) $(\alpha,\beta,\gamma,\delta)\in\mathcal{D}^1_{\mathrm{cr}}$, i.e., $\alpha=o$:
First, for $1\le i\le s_1$,
\[
B_{(\epsilon_i,o,o,o)}B_{(o,\beta,\gamma,\delta)}
=\bigl(A^{(1)}_{(\epsilon_i,o)}A^{(1)}_{(o,\beta)}\bigr)
\otimes
\bigl(A^{(2)}_{\bm{o}}A^{(2)}_{(\gamma,\delta)}\bigr).
\]
Applying \eqref{eq:quotient_scheme_A} to $\mathfrak{X}^{(2)}$, we get
$A^{(2)}_{\bm{o}}A^{(2)}_{(\gamma,\delta)}
=\frac{k_{(\gamma,\delta)}}{k_{\bm{\gamma}}}A^{(2)}_{\bm{\gamma}}$.
On the other hand, applying Lemma~\ref{lem:p_support}~\ref{item:p_support_ii} to $\mathfrak{X}^{(1)}$,
$A^{(1)}_{(\epsilon_i,o)}A^{(1)}_{(o,\beta)}$
is a linear combination of $A^{(1)}_{(\alpha',\beta')}$ with
$(\alpha',\beta')\le_1(\epsilon_i,\beta)$.
Hence the product $B_{(\epsilon_i,o,o,o)}B_{(o,\beta,\gamma,\delta)}$ is a linear combination of terms of the form
\[
B_{(\alpha',\beta',\gamma,o)}\quad \text{if }(\alpha',\beta')\notin\mathcal{C}_1,
\qquad
\sum_{(\gamma,\delta')\in\bm{\gamma}} B_{(o,\beta',\gamma,\delta')}\quad \text{if }(\alpha',\beta')=(o,\beta')\in\mathcal{C}_1.
\]
In either case, all indices are bounded above by $(\epsilon_i,\beta,\gamma,\delta)$.
Moreover, if $(\epsilon_i,\beta,\gamma,\delta)\in\mathcal{D}_{\mathrm{cr}}$, then
$(\epsilon_i,\beta,\gamma,\delta)\in\mathcal{D}^2_{\mathrm{cr}}$, so $\delta=o$.
The coefficient of $B_{(\epsilon_i,\beta,\gamma,o)}$ can be written, using
the intersection number $p^{(\epsilon_i,\beta)}_{(\epsilon_i,o),(o,\beta)} \neq 0$ of $\mathfrak{X}^{(1)}$
and the constant $\frac{k_{(\gamma,o)}}{k_{\bm{\gamma}}}\neq 0$ from $\mathfrak{X}^{(2)}$, as
\[
p^{(\epsilon_i,\beta,\gamma,o)}_{(\epsilon_i,o,o,o),(o,\beta,\gamma,o)}=p^{(\epsilon_i,\beta)}_{(\epsilon_i,o),(o,\beta)}\cdot \frac{k_{(\gamma,o)}}{k_{\bm{\gamma}}}\neq 0.
\]

Next, for $1\le j\le \ell_1-s_1$, we have
\[
B_{(o,\epsilon_j,o,o)}B_{(o,\beta,\gamma,\delta)}
=\bigl(A^{(1)}_{(o,\epsilon_j)}A^{(1)}_{(o,\beta)}\bigr)\otimes A^{(2)}_{(\gamma,\delta)}.
\]
By Lemma~\ref{lem:p_support}~\ref{item:p_support_ii}, all indices are at most
$(o,\beta+\epsilon_j,\gamma,\delta)$.
If $(o,\beta+\epsilon_j,\gamma,\delta)\in\mathcal{D}_{\mathrm{cr}}$, then the top coefficient
of $B_{(o,\beta+\epsilon_j,\gamma,\delta)}$ is nonzero.

Furthermore, for $1\le u\le s_2$ and $1\le v\le \ell_2-s_2$,
\begin{align*}
B_{(o,o,\epsilon_u,o)}B_{(o,\beta,\gamma,\delta)}
&=A^{(1)}_{(o,\beta)}\otimes \bigl(A^{(2)}_{(\epsilon_u,o)}A^{(2)}_{(\gamma,\delta)}\bigr),\\
B_{(o,o,o,\epsilon_v)}B_{(o,\beta,\gamma,\delta)}
&=A^{(1)}_{(o,\beta)}\otimes \bigl(A^{(2)}_{(o,\epsilon_v)}A^{(2)}_{(\gamma,\delta)}\bigr),
\end{align*}
and applying Lemma~\ref{lem:p_support}~\ref{item:p_support_ii} to $\mathfrak{X}^{(2)}$,
the indices are bounded by $(o,\beta,\gamma+\epsilon_u,\delta)$ and
$(o,\beta,\gamma,\delta+\epsilon_v)$, respectively.
If these indices belong to $\mathcal{D}_{\mathrm{cr}}$, the corresponding top coefficients are nonzero.

\smallskip
\noindent
(b) $(\alpha,\beta,\gamma,\delta)\in\mathcal{D}^2_{\mathrm{cr}}$, i.e., $\alpha\neq o$:
Take $(\alpha,\beta,\gamma,o)\in\mathcal{D}_{\mathrm{cr}}$ with $\alpha\ne o$.
First, for $1\le i\le s_1$,
\[
B_{(\epsilon_i,o,o,o)}B_{(\alpha,\beta,\gamma,o)}
=\bigl(A^{(1)}_{(\epsilon_i,o)}A^{(1)}_{(\alpha,\beta)}\bigr)
\otimes
\bigl(A^{(2)}_{\bm{o}}A^{(2)}_{\bm{\gamma}}\bigr)
=\bigl(A^{(1)}_{(\epsilon_i,o)}A^{(1)}_{(\alpha,\beta)}\bigr)\otimes (p_2 A^{(2)}_{\bm{\gamma}}).
\]
Applying Lemma~\ref{lem:p_support}~\ref{item:p_support_ii} to the first factor,
the right-hand side is a linear combination of terms with
$(\alpha',\beta')\le_1(\alpha+\epsilon_i,\beta)$.
Terms with $(\alpha',\beta')\notin\mathcal{C}_1$ are simply $B_{(\alpha',\beta',\gamma,o)}$,
while terms with $(\alpha',\beta')=(o,\beta')\in\mathcal{C}_1$ can be written as
$A^{(1)}_{(o,\beta')}\otimes A^{(2)}_{\bm{\gamma}}
=\sum_{(\gamma,\delta)\in\bm{\gamma}} B_{(o,\beta',\gamma,\delta)}$.
In either case, indices are bounded by $(\alpha+\epsilon_i,\beta,\gamma,o)$.
If $(\alpha+\epsilon_i,\beta,\gamma,o)\in\mathcal{D}_{\mathrm{cr}}$, then
the coefficient of $B_{(\alpha+\epsilon_i,\beta,\gamma,o)}$ is nonzero.

Similarly, for $1\le j\le \ell_1-s_1$,
\[
B_{(o,\epsilon_j,o,o)}B_{(\alpha,\beta,\gamma,o)}
=\bigl(A^{(1)}_{(o,\epsilon_j)}A^{(1)}_{(\alpha,\beta)}\bigr)\otimes A^{(2)}_{\bm{\gamma}}
\]
and indices are controlled by $(\alpha,\beta+\epsilon_j,\gamma,o)$.
If $(\alpha,\beta+\epsilon_j,\gamma,o)\in\mathcal{D}_{\mathrm{cr}}$, then the top coefficient is nonzero.

Next, for $1\le u\le s_2$, using $A^{(2)}_{\bm{\gamma}}=A^{(2)}_{\bm{\gamma}}M_{\mathcal{C}_2}$, we have
\[
A^{(2)}_{(\epsilon_u,o)}A^{(2)}_{\bm{\gamma}}
=\frac{k_{(\epsilon_u,o)}}{k_{\epsilon_u\mathcal{C}_2}}\cdot \frac{1}{p_2}
A^{(2)}_{\bm{\epsilon_u}}A^{(2)}_{\bm{\gamma}}.
\]
Here $\mathfrak{X}^{(2)}/\mathcal{C}_2$ is an $s_2$-variate $P$-polynomial association scheme
on $\mathcal{D}_{2,\mathrm{qt}}$ by Theorem~\ref{thm:QuotientScheme}~\ref{item:QuotientScheme_P},
so the right-hand side is a linear combination of $A^{(2)}_{\bm{\eta}}$ with
$\eta\le_{2,\mathrm{qt}}\gamma+\epsilon_u$.
Therefore,
\[
B_{(o,o,\epsilon_u,o)}B_{(\alpha,\beta,\gamma,o)}
\in
\Span\{B_{(\alpha,\beta,\eta,o)}\mid \eta\le_{2,\mathrm{qt}}\gamma+\epsilon_u\}\subset
\Span\{B_{\zeta}\mid \zeta\le_{\otimes}(\alpha,\beta,\gamma+\epsilon_u,o)\}.
\]
Moreover, if $(\alpha,\beta,\gamma+\epsilon_u,o)\in\mathcal{D}_{\mathrm{cr}}$, then
the coefficient of $B_{(\alpha,\beta,\gamma+\epsilon_u,o)}$ is nonzero.

Finally, for $1\le v\le \ell_2-s_2$, since $(o,\epsilon_v)\in\mathcal{C}_2$, equation
\eqref{eq:quotient_scheme_A} gives
$A^{(2)}_{(o,\epsilon_v)}M_{\mathcal{C}_2}=k_{(o,\epsilon_v)}M_{\mathcal{C}_2}$.
Hence
\[
B_{(o,o,o,\epsilon_v)}B_{(\alpha,\beta,\gamma,o)}
=A^{(1)}_{(\alpha,\beta)}\otimes \bigl(A^{(2)}_{(o,\epsilon_v)}A^{(2)}_{\bm{\gamma}}\bigr)
=A^{(1)}_{(\alpha,\beta)}\otimes \bigl(k_{(o,\epsilon_v)}A^{(2)}_{\bm{\gamma}}\bigr)
=k_{(o,\epsilon_v)}B_{(\alpha,\beta,\gamma,o)}.
\]
Thus the index is bounded by $(\alpha,\beta,\gamma,\epsilon_v)$.

Therefore, $\mathcal{D}_{\mathrm{cr}}$ and $\le_{\otimes}$ satisfy the conditions of
Proposition~\ref{prop:P-TFAE}, and the crested product is an
$(\ell_1+\ell_2)$-variate $P$-polynomial association scheme.
\end{proof}

Next we discuss imprimitivity of the crested product.
First, we note a statement that holds without assuming multivariate $P$-polynomiality on the factors.
\begin{prop}\label{prop:crested_imprimitive}
Let $\mathfrak{X}^{(1)},\mathfrak{X}^{(2)}$ be commutative association schemes,
and let $\mathfrak{X}$ be their crested product.
Then $\mathcal{C}:=\{(0,j)\mid j\in\mathcal{I}_2\}\subset\mathcal{K}$
is a closed subset of $\mathfrak{X}$, so $\mathfrak{X}$ is imprimitive.
Moreover, the following hold.
\begin{enumerate}[label=$(\roman*)$]
\item \label{item:prop:crested_imprimitive_1} For any $(x_1,x_2)\in X^{(1)}\times X^{(2)}$, the block determined by $\mathcal{C}$ is $\{x_1\}\times X^{(2)}$, and the block scheme on this block, $\mathfrak{X}_{\{x_1\}\times X^{(2)}}$, is isomorphic to $\mathfrak{X}^{(2)}$;
\item \label{item:prop:crested_imprimitive_2} the quotient scheme $\mathfrak{X}/\mathcal{C}$ is isomorphic to $\mathfrak{X}^{(1)}$.
\end{enumerate}
\end{prop}

\begin{proof}
\ref{item:prop:crested_imprimitive_1}
That $\mathcal{C}$ is closed follows from Lemma~\ref{lem:closed_subset_equiv}, since
$\sum_{j\in\mathcal{I}_2}A_{(0,j)}=I_{X^{(1)}}\otimes\sum_{j\in\mathcal{I}_2}A^{(2)}_j=I_{X^{(1)}}\otimes J_{X^{(2)}}$
induces an equivalence relation $\sim$ on $X^{(1)}\times X^{(2)}$.
Hence $\mathfrak{X}$ is imprimitive.
Also, $(x_1,x_2)\sim (y_1,y_2)$ is equivalent to $x_1=y_1$, so each block is
$\{x_1\}\times X^{(2)}$ ($x_1\in X^{(1)}$).
On this block, only matrices
$\{A_{(0,j)}\vert_{\{x_1\}\times X^{(2)}}\}_{j\in\mathcal{I}_2}$ appear.
Since $A_{(0,j)}=I_{X^{(1)}}\otimes A^{(2)}_j$,
$A_{(0,j)}\vert_{\{x_1\}\times X^{(2)}}$ coincides exactly with $A^{(2)}_j$.
Therefore the block scheme is isomorphic to $\mathfrak{X}^{(2)}$.

\ref{item:prop:crested_imprimitive_2}
The quotient set $(X^{(1)}\times X^{(2)})/\mathcal{C}$ is the set of blocks, and it has the natural bijection
$\pi:(X^{(1)}\times X^{(2)})/\mathcal{C}\to X^{(1)},\ (x_1,x_2)\mathcal{C}\mapsto x_1$.
By Definition~\ref{df:crested_product}, for any $x=(x_1,x_2),y=(y_1,y_2)$,
the first component of $\mathcal{R}(x,y)$ is always $\mathcal{R}^{(1)}(x_1,y_1)$.
Hence relations between blocks are determined only by the first component,
and $\pi$ yields an isomorphism between $\mathfrak{X}/\mathcal{C}$ and $\mathfrak{X}^{(1)}$.
\end{proof}

Under the same assumptions as Theorem~\ref{thm:crested_multivar},
it is straightforward to verify that the monomial order $\le_{\otimes}$ on $\N^{\ell_1+\ell_2}$
is of $\ell_1$-block type, so we obtain the following corollary.

\begin{cor}
\label{cor:crested_elimination}

The closed subset in Proposition~\ref{prop:crested_imprimitive} corresponds to
\[
\mathcal{C}:=\{(o,o,\gamma,\delta)\mid (\gamma,\delta)\in\mathcal{D}_2\}\subset\mathcal{D}_{\mathrm{cr}}
\]
and the resulting decomposition into the quotient scheme $\mathfrak{X}^{(1)}$ and the block scheme $\mathfrak{X}^{(2)}$
is described exactly by this elimination-type monomial order $\le_{\otimes}$.
\end{cor}

Finally, on the $Q$-side of crested products, we record that the number of variables is also additive
with respect to suitable dual closed subsets.
Let $r=1,2$, and let $\mathcal{J}_r$ denote the index set of primitive idempotents of $\mathfrak{X}^{(r)}$.
Let $\mathcal{C}_r^\ast\subset\mathcal{J}_r$ be the dual closed subset corresponding to $\mathcal{C}_r$.
Also let $\{E^{(r)}_{\xi}\}_{\xi\in\mathcal{D}_r^\ast}$ be the labeling of primitive idempotents
for an $\ell_r^\ast$-variate $Q$-polynomial structure of $\mathfrak{X}^{(r)}$.
Assume that
$\mathcal{D}_1^\ast\subset\N^{s_1^\ast}\times\N^{\ell_1^\ast-s_1^\ast}$ and
$\mathcal{C}_1^\ast=\mathcal{D}_1^\ast\cap (\{o\}\times \N^{\ell_1^\ast-s_1^\ast})$
for $\mathfrak{X}^{(1)}$ with respect to a monomial order $\le_1^\ast$ of $s_1^\ast$-block type, and
$\mathcal{D}_2^\ast\subset\N^{s_2^\ast}\times\N^{\ell_2^\ast-s_2^\ast}$ and
$\mathcal{C}_2^\ast=\mathcal{D}_2^\ast\cap (\{o\}\times \N^{\ell_2^\ast-s_2^\ast})$
for $\mathfrak{X}^{(2)}$ with respect to a monomial order $\le_2^\ast$ of $s_2^\ast$-elimination type.
Set
\[
\jmath_1^\ast:\N^{s_1^\ast}\to \N^{s_1^\ast}\times\N^{\ell_1^\ast-s_1^\ast},
\qquad
\jmath_1^\ast(\alpha):=(\alpha,o),
\qquad
\mathcal{D}_{1,\mathrm{bl}}^\ast:=(\jmath_1^\ast)^{-1}(\mathcal{D}_1^\ast)
\]
and
\[
\iota_2^\ast:\N^{\ell_2^\ast-s_2^\ast}\to \N^{s_2^\ast}\times\N^{\ell_2^\ast-s_2^\ast},
\qquad
\iota_2^\ast(\delta):=(o,\delta),
\qquad
\mathcal{D}_{2,\mathrm{qt}}^\ast:=(\iota_2^\ast)^{-1}(\mathcal{C}_2^\ast)
\]
and define an order $\le_{\mathrm{cr}}^\ast$ on $\N^{\ell_1^\ast+\ell_2^\ast}$ by
\begin{equation}
\label{eq:product_order_crested_Q}
(\alpha,\beta,\gamma,\delta)\le_{\mathrm{cr}}^\ast(\alpha',\beta',\gamma',\delta')
\iff
\bigl((\gamma,\delta)<_2^\ast(\gamma',\delta')\bigr)
\ \text{or}\ 
\bigl((\gamma,\delta)=(\gamma',\delta')\ \text{and}\ (\alpha,\beta)\le_1^\ast(\alpha',\beta')\bigr).
\end{equation}
This is again a monomial order.
Now set
\begin{equation}
\label{eq:crested_domain_multivar_Q}
\mathcal{D}_{\mathrm{cr}}^\ast
:=
\{(\alpha,\beta,o,\delta)\mid (\alpha,\beta)\in\mathcal{D}_1^\ast,\ \delta\in\mathcal{D}_{2,\mathrm{qt}}^\ast\}
\ \sqcup\
\{(\alpha,o,\gamma,\delta)\mid \alpha\in\mathcal{D}_{1,\mathrm{bl}}^\ast,\ (\gamma,\delta)\in\mathcal{D}_2^\ast\setminus\mathcal{C}_2^\ast\}
\subset\N^{\ell_1^\ast+\ell_2^\ast}
\end{equation}
and for $\alpha\in\mathcal{D}_{1,\mathrm{bl}}^\ast$ define
\[
E^{(1)}_{\bm{\alpha}}
:=
\sum_{\beta\,:\,(\alpha,\beta)\in\mathcal{D}_1^\ast}
E^{(1)}_{(\alpha,\beta)}.
\]
By Theorem~\ref{thm:block_scheme_Ppoly}~\ref{item:block_scheme_Ppoly_ii},
these are precisely the lifts inside $\mathfrak{X}^{(1)}$ of primitive idempotents
$\bm{\alpha}\in \mathcal{J}_1/\mathcal{C}_1^\ast$ of the block scheme with respect to $\mathcal{C}_1$.
Bailey--Cameron~\cite{baileycameron2005} showed that, under the notation above,
the primitive idempotents of the crested product $\mathfrak{X}$ are
\begin{align}
F_{(\alpha,\beta,o,\delta)}
&:=
E^{(1)}_{(\alpha,\beta)}\otimes E^{(2)}_{(o,\delta)}
&&\bigl((\alpha,\beta)\in\mathcal{D}_1^\ast,\ \delta\in\mathcal{D}_{2,\mathrm{qt}}^\ast\bigr),
\label{eq:crested_dual_basis_first}\\
F_{(\alpha,o,\gamma,\delta)}
&:=
E^{(1)}_{\bm{\alpha}}\otimes E^{(2)}_{(\gamma,\delta)}
&&\bigl(\alpha\in\mathcal{D}_{1,\mathrm{bl}}^\ast,\ (\gamma,\delta)\in\mathcal{D}_2^\ast\setminus\mathcal{C}_2^\ast\bigr).
\label{eq:crested_dual_basis_second}
\end{align}
Thus the primitive idempotents of the crested product are naturally indexed by
\eqref{eq:crested_domain_multivar_Q}.
In the proof of Theorem~\ref{thm:crested_multivar}, if we replace adjacency matrices
with primitive idempotents, ordinary matrix products with Hadamard products,
and Lemma~\ref{lem:p_support} with Lemma~\ref{lem:q_support},
the dominance conditions of Proposition~\ref{prop:Q-TFAE} follow in the same way.
Therefore we obtain the following.
\begin{thm}
\label{thm:crested_multivar_Q}
Under the assumptions above, the crested product of
$\mathfrak{X}^{(1)}$ and $\mathfrak{X}^{(2)}$ with respect to $(\mathcal{C}_1,\mathcal{C}_2)$
is an $(\ell_1^\ast+\ell_2^\ast)$-variate $Q$-polynomial association scheme on
\eqref{eq:crested_domain_multivar_Q}, and its monomial order is
$\le_{\mathrm{cr}}^\ast$ defined by \eqref{eq:product_order_crested_Q}.
\end{thm}

\subsection{Direct products and arbitrary monomial orders}
In this subsection, we prove converses to Proposition~\ref{prop:direct_prod_P/Q-poly} by combining Theorems~\ref{thm:EliminationThm} and~\ref{thm:QuotientScheme}.
More precisely, the converse to Proposition~\ref{prop:direct_prod_P/Q-poly}~\ref{item:direct_prod_P} yields Theorem~\ref{thm:directproduct_anyorder},
while the converse to Proposition~\ref{prop:direct_prod_P/Q-poly}~\ref{item:direct_prod_Q} yields Theorem~\ref{thm:directproduct_anyorder_Q}.

\begin{thm}
\label{thm:directproduct_anyorder}
Fix $\ell\in\N$, and let $\mathfrak{X}=(X,\mathcal{R})$ be a commutative association scheme.
Then the following are equivalent.
\begin{enumerate}[label=$(\roman*)$]
    \item \label{item:directproduct_anyorder_i} $\mathfrak{X}$ is isomorphic to a direct product of $\ell$ commutative univariate $P$-polynomial association schemes.
    \item \label{item:directproduct_anyorder_ii} There exist $\mathcal{D}\subset\N^\ell$ and a labeling of the adjacency matrices $\{A_\alpha\}_{\alpha\in\mathcal{D}}$ such that $\mathfrak{X}=(X,\{A_\alpha\}_{\alpha\in\mathcal{D}})$ is an $\ell$-variate $P$-polynomial association scheme on $\mathcal{D}$ with respect to any monomial order $\le$.
\end{enumerate}
\end{thm}

\begin{proof}
The implication
\ref{item:directproduct_anyorder_i}$\Longrightarrow$\ref{item:directproduct_anyorder_ii}
is already proved in Proposition~\ref{prop:direct_prod_P/Q-poly}, so we omit it.
We prove
\ref{item:directproduct_anyorder_ii}$\Longrightarrow$\ref{item:directproduct_anyorder_i}.
Assume \ref{item:directproduct_anyorder_ii}. Then there exist
$\mathcal{D}\subset\N^\ell$ and a labeling of the adjacency matrices
$\{A_\alpha\}_{\alpha\in\mathcal{D}}$ such that
$\mathfrak{X}=(X,\{A_\alpha\}_{\alpha\in\mathcal{D}})$,
and for every monomial order, $\mathfrak{X}$ is an
$\ell$-variate $P$-polynomial association scheme on $\mathcal{D}$.
We proceed by induction on $\ell$.
The case $\ell=1$ is trivial, so assume $\ell\ge 2$ and that the claim holds for $\ell-1$.

First, let $\le'$ be any monomial order on $x_1,\ldots,x_{\ell-1}$.
Define a monomial order $\le_1$ on $\N^\ell$ by
\[
(\alpha,a_\ell)<_1(\beta,b_\ell)
\quad\Longleftrightarrow\quad
\bigl(\alpha<' \beta\bigr)\ \text{or}\ \bigl(\alpha=\beta\ \text{and}\ a_\ell<b_\ell\bigr)
\qquad (\alpha,\beta\in\N^{\ell-1}).
\]
Then $\le_1$ is a monomial order of $(\ell-1)$-block type.
By assumption \ref{item:directproduct_anyorder_ii},
$\mathfrak{X}$ is an $\ell$-variate $P$-polynomial association scheme with respect to $\le_1$. Set
\[
\mathcal{C}_1:=\mathcal{D}\cap (\{o\}\times \N)
=\{(a_1,\ldots ,a_\ell)\in\mathcal{D}\mid a_1=a_2=\cdots =a_{\ell -1}=0\}.
\]
Then Theorem~\ref{thm:QuotientScheme}
implies that $\mathfrak{X}/\mathcal{C}_1$ is an $(\ell-1)$-variate
$P$-polynomial association scheme with respect to $\le'$.
As $\le'$ is arbitrary, $\mathfrak{X}/\mathcal{C}_1$ satisfies
the $(\ell-1)$-variable version of \ref{item:directproduct_anyorder_ii}.

Next, let $\le_2$ be the lexicographic order with
$x_1<x_2<\cdots<x_\ell$, viewed as $1$-block type.
Again by \ref{item:directproduct_anyorder_ii},
$\mathfrak{X}$ is an $\ell$-variate $P$-polynomial association scheme with respect to $\le_2$.
Set
\[
\mathcal{C}_2:=
\mathcal{D}\cap (\N^{\ell -1}\times \{o\})
=\{(a_1,\ldots ,a_\ell)\in\mathcal{D}\mid a_\ell=0\}.
\]
Then, $\mathfrak{X}/\mathcal{C}_2$ is a univariate $P$-polynomial association scheme.

Let $\sim_r$ denote the equivalence relation corresponding to $\mathcal{C}_r$,
and write $X/\mathcal{C}_r=\{x\mathcal{C}_r\mid x\in X\}$ ($r=1,2$).
Consider the map
\[
f\colon X\longrightarrow (X/\mathcal{C}_1)\times (X/\mathcal{C}_2),
\qquad
x\longmapsto \bigl(x\mathcal{C}_1,x\mathcal{C}_2\bigr).
\]
We show that $f$ is bijective.
First, $f$ is injective.
Since $\mathcal{C}_1\cap\mathcal{C}_2=\{o\}$, for any $x,y\in X$,
\begin{align*}
\bigl(x\mathcal{C}_1,x\mathcal{C}_2\bigr)=(y\mathcal{C}_1,y\mathcal{C}_2)
&\Longleftrightarrow\ 
x\sim_1 y\ \text{and}\ x\sim_2 y\\
&\Longleftrightarrow\ 
\mathcal{R}(x,y)\in \mathcal{C}_1\ \text{and}\ \mathcal{R}(x,y)\in \mathcal{C}_2\\
&\Longleftrightarrow\ 
\mathcal{R}(x,y)\in \mathcal{C}_1\cap\mathcal{C}_2=\{o\}\\
&\Longleftrightarrow\ 
x=y
\end{align*}
holds, so $f$ is injective.
Next, we show surjectivity.
Take arbitrary $x\mathcal{C}_1\in X/\mathcal{C}_1$ and $y\mathcal{C}_2\in X/\mathcal{C}_2$,
and choose representatives $x$ and $y$.
Write $\mathcal{R}(x,y)=(\alpha,a)\in\N^{\ell-1}\times\N$.
By Lemma~\ref{lem:p_support}, $p^{(\alpha, a)}_{(o, a),(\alpha, 0)}\neq 0$.
Hence there exists $z\in X$ such that
$\mathcal{R}(x,z)=(o,a)$ and $\mathcal{R}(z,y)=(\alpha,0)$.
This implies $x\sim_1 z$ and $z\sim_2 y$,
so $f(z)=(x\mathcal{C}_1,y\mathcal{C}_2)$.
Thus $f$ is surjective.
Similarly, consider the map on relation indices
\[
g\colon \mathcal{D}\longrightarrow (\mathcal{D}/\mathcal{C}_1)\times(\mathcal{D}/\mathcal{C}_2),
\qquad
\alpha\longmapsto \bigl(\alpha\mathcal{C}_1,\alpha\mathcal{C}_2\bigr)
\]
where $\equiv_r$ is the equivalence relation on $\mathcal{D}$ induced by $\mathcal{C}_r$.
By Lemma~\ref{lem:qt_equiv},
\[
(a_1,\ldots ,a_{\ell-1},a_\ell)\equiv_1(b_1,\ldots ,b_{\ell-1},b_\ell)
\iff a_\ell=b_\ell,
\]
\[
(a_1,\ldots ,a_{\ell-1},a_\ell)\equiv_2(b_1,\ldots ,b_{\ell-1},b_\ell)
\iff (a_1,\ldots ,a_{\ell-1})=(b_1,\ldots ,b_{\ell-1})
\]
so $g$ is bijective.
Moreover, by definition of quotient schemes, for any $x,x'\in X$,
$(\mathcal{R}/\mathcal{C}_r)\bigl(x\mathcal{C}_r,x'\mathcal{C}_r\bigr)
=\mathcal{R}(x,x')\mathcal{C}_r$ ($r=1,2$)
holds. Therefore,
\[
g\bigl(\mathcal{R}(x,x')\bigr)
=
\bigl(
(\mathcal{R}/\mathcal{C}_1)\bigl(x\mathcal{C}_1,x'\mathcal{C}_1\bigr),
(\mathcal{R}/\mathcal{C}_2)\bigl(x\mathcal{C}_2,x'\mathcal{C}_2\bigr)
\bigr)
=
\mathcal{R}_{\mathfrak{X}/\mathcal{C}_1\otimes \mathfrak{X}/\mathcal{C}_2}\bigl(f(x),f(x')\bigr)
\]
holds, and $(f,g)$ is an isomorphism from $\mathfrak{X}$ to
$\mathfrak{X}/\mathcal{C}_1\otimes \mathfrak{X}/\mathcal{C}_2$.
That is,
\begin{equation}
\label{eq:directproduct_anyorder}
\mathfrak{X}\cong \mathfrak{X}/\mathcal{C}_1\otimes \mathfrak{X}/\mathcal{C}_2.    
\end{equation}

Since $\mathfrak{X}/\mathcal{C}_1$ satisfies the $(\ell-1)$-variable version of
\ref{item:directproduct_anyorder_ii}, the induction hypothesis implies that
$\mathfrak{X}/\mathcal{C}_1$ is isomorphic to a direct product of $\ell-1$
univariate $P$-polynomial association schemes.
Hence, by \eqref{eq:directproduct_anyorder},
$\mathfrak{X}$ is isomorphic to a direct product of $\ell$ univariate
$P$-polynomial association schemes.
This proves
\ref{item:directproduct_anyorder_ii}$\implies$\ref{item:directproduct_anyorder_i}.
\end{proof}

The $Q$-polynomial version of Theorem~\ref{thm:directproduct_anyorder}
is proved similarly.
Since the proof follows the same structure, we omit it.
\begin{thm}\label{thm:directproduct_anyorder_Q}
Fix $\ell^\ast\in\N$, and let
$\mathfrak{X}=(X,\mathcal{R})$ be a commutative association scheme.
Then the following are equivalent.
\begin{enumerate}[label=$(\roman*)$]
 \item \label{item:directproduct_anyorder_Q_i} $\mathfrak{X}$ is isomorphic to a direct product of $\ell^\ast$ (commutative) $Q$-polynomial association schemes.
 \item \label{item:directproduct_anyorder_Q_ii} There exist $\mathcal{D}^\ast\subset \N^{\ell^\ast}$ and a labeling of the primitive idempotents $\{E_\alpha\}_{\alpha\in\mathcal{D}^\ast}$ such that $\mathfrak{X}$ with $\{E_\alpha\}_{\alpha\in\mathcal{D}^\ast}$ is an $\ell^\ast$-variate $Q$-polynomial association scheme on $\mathcal{D}^\ast$ with respect to any monomial order $\le^\ast$.
\end{enumerate}
\end{thm}

\section{Further topics}\label{sec:other_topics}

\subsection{Formal duality}
\label{sec:curtin}

In this subsection, we show that Curtin's formal duality~\cite{curtin2008} is inherited by quotient and block schemes in the imprimitive case, and hence that the corresponding multivariate polynomial structures are inherited as well.

Let $\mathfrak{X}=(X,\mathcal{R},\mathcal{I})$ and $\mathfrak{X}^\vee=(X^\vee,\mathcal{R}^\vee,\mathcal{J})$ be commutative association schemes with $|X|=|X^\vee|$.
Denote their adjacency matrices by $\{A_i\}_{i\in\mathcal{I}}$ and $\{A^\vee_j \}_{j\in\mathcal{J}}$, and their Bose--Mesner algebras by $\mathfrak{A}$ and $\mathfrak{A}^\vee$.

\begin{df}[formal duality]
\label{df:formal-duality}
Let $n:=|X|=|X^\vee|$. A linear isomorphism $\Phi:\mathfrak{A}\to\mathfrak{A}^\vee$ is called a \emph{formal duality} from $\mathfrak{A}$ to $\mathfrak{A}^\vee$ if it satisfies the following:
\begin{enumerate}[label=$(\roman*)$]
\item For any $A,B\in\mathfrak{A}$, $\Phi(AB)=\Phi(A)\circ\Phi(B)$.
\item For any $A,B\in\mathfrak{A}$, $\Phi(A\circ B)=n^{-1}\Phi(A)\Phi(B)$.
\end{enumerate}
In this case, we say that $\mathfrak{X}$ and $\mathfrak{X}^\vee$ are \emph{formally dual} (in the sense of their Bose--Mesner algebras).
\end{df}

\medskip
\noindent
Below, we fix a formal duality $\Phi$. In this case, by appropriately indexing the primitive idempotents of $\mathfrak{A},\mathfrak{A}^\vee$ as $\{E_j\}_{j\in\mathcal{J}}, \{E_i^\vee\}_{i\in\mathcal{I}}$, we can arrange the indices of $\{A_i\}_{i\in\mathcal{I}}$, $\{E_j\}_{j\in\mathcal{J}}$ and $\{A_j^\vee\}_{j\in\mathcal{J}}$, $\{E_i^\vee\}_{i\in\mathcal{I}}$ to be \emph{standard} with respect to $\Phi$, that is,
\begin{equation}
\label{eq:formal-duality-standard}
\Phi(E_j)=A_j^\vee\ (j\in\mathcal{J}),
\qquad
\Phi(A_i)=|X|\,(E_i^\vee)^T\ (i\in\mathcal{I})
\end{equation}
hold.
From this, it immediately follows that the intersection numbers of $\mathfrak{X}$ and the Krein numbers of $\mathfrak{X}^\vee$ are equal, and vice versa. Therefore, we have the following lemma:
\begin{lem}
\label{lem:formal_duality_multivariate}
Assume that $\mathfrak{X}$ and $\mathfrak{X}^\vee$ are formally dual and that the indexing is standard as in \eqref{eq:formal-duality-standard}.
Then the following hold.
\begin{enumerate}[label=$(\roman*)$]
\item \label{item:formal_duality_multivariate_i}
If $\mathfrak{X}$ is an $\ell$-variate $P$-polynomial association scheme on $\mathcal{D}\subset \N^\ell$ with respect to a monomial order $\le$, then $\mathfrak{X}^\vee$ is an $\ell$-variate $Q$-polynomial association scheme on the same domain $\mathcal{D}$ with respect to the same monomial order $\le$.
\item \label{item:formal_duality_multivariate_ii}
If $\mathfrak{X}$ is an $\ell^\ast$-variate $Q$-polynomial association scheme on $\mathcal{D}^\ast\subset \N^{\ell^\ast}$ with respect to a monomial order $\le^\ast$, then $\mathfrak{X}^\vee$ is an $\ell^\ast$-variate $P$-polynomial association scheme on the same domain $\mathcal{D}^\ast$ with respect to the same monomial order $\le^\ast$.
\end{enumerate}
\end{lem}

We restate Curtin~\cite{curtin2008}'s result in the notation of this paper.

\begin{thm}[cf.~Theorem~7.1 in Curtin~\cite{curtin2008}]
\label{thm:curtin71}
Fix a formal duality $\Phi:\mathfrak{A}\to\mathfrak{A}^\vee$ satisfying \eqref{eq:formal-duality-standard}.
Suppose $\mathfrak{X}$ is imprimitive, and let $\mathcal{C}$ be a closed subset with $\{0\}\subsetneq \mathcal{C}\subsetneq \mathcal{I}$, and let $\mathcal{C}^\ast\subset\mathcal{J}$ be the dual closed subset corresponding to $\mathcal{C}$.
Then the following hold:
\begin{enumerate}[label=$(\roman*)$]
\item \label{item:curtin71_i} $\mathfrak{X}^\vee$ is imprimitive with respect to the closed subset $\mathcal{C}^\ast\subset\mathcal{J}$, and $\mathcal{C}$ is its dual closed subset (i.e., $(\mathcal{C}^\ast)^\ast=\mathcal{C}$).
\item \label{item:curtin71_ii} For any $x_0\in X$, the block scheme $\mathfrak{X}_{x_0\mathcal{C}}$ on the block (equivalence class) $x_0\mathcal{C}$ determined by $\mathcal{C}$, and the quotient scheme $\mathfrak{X}^\vee/\mathcal{C}^\ast$ are formally dual.
\item \label{item:curtin71_iii} The quotient scheme $\mathfrak{X}/\mathcal{C}$ and any block scheme of $\mathfrak{X}^\vee$ with respect to $\mathcal{C}^\ast$ are formally dual.
\end{enumerate}
\end{thm}

Therefore, from Lemma~\ref{lem:formal_duality_multivariate} and Theorem~\ref{thm:curtin71}, we obtain the following corollary.
\begin{cor}
\label{cor:curtin_multivariate}
Under the assumptions of Theorem~\ref{thm:curtin71}, formal duality propagates the reduced multivariate polynomial structure to the associated block and quotient schemes.
More precisely, the following hold.
\begin{enumerate}[label=$(\roman*)$]
\item \label{item:curtin_multivariate_i}
Assume that $\mathfrak{X}$ is an imprimitive $\ell$-variate $P$-polynomial association scheme on $\mathcal{D}\subset \N^s\times \N^{\ell-s}$ with respect to a monomial order of $s$-block type, and that
\[
\mathcal{C}=\mathcal{D}\cap (\{o\}\times \N^{\ell-s}).
\]
Let $\mathcal{C}^\ast\subset\mathcal{J}$ be the dual closed subset corresponding to $\mathcal{C}$;
\begin{enumerate}[label=$(\alph*)$]
\item \label{item:curtin_multivariate_i_a}
For any $x_0\in X$, the block scheme $\mathfrak{X}_{x_0\mathcal{C}}$ is an $(\ell-s)$-variate $P$-polynomial association scheme, the quotient scheme $\mathfrak{X}^\vee/\mathcal{C}^\ast$ is an $(\ell-s)$-variate $Q$-polynomial association scheme, and these two schemes are formally dual;
\item \label{item:curtin_multivariate_i_b}
The quotient scheme $\mathfrak{X}/\mathcal{C}$ is an $s$-variate $P$-polynomial association scheme, every block scheme of $\mathfrak{X}^\vee$ with respect to $\mathcal{C}^\ast$ is an $s$-variate $Q$-polynomial association scheme, and these schemes are formally dual;
\end{enumerate}
\item \label{item:curtin_multivariate_ii}
Assume that $\mathfrak{X}$ is an imprimitive $\ell^\ast$-variate $Q$-polynomial association scheme on $\mathcal{D}^\ast\subset \N^{s^\ast}\times \N^{\ell^\ast-s^\ast}$ with respect to a monomial order of $s^\ast$-block type, and that
\[
\mathcal{C}^\ast=\mathcal{D}^\ast\cap (\{o\}\times \N^{\ell^\ast-s^\ast}).
\]
Let $\mathcal{C}\subset\mathcal{I}$ be the dual closed subset corresponding to $\mathcal{C}^\ast$.
\begin{enumerate}[label=$(\alph*)$]
\item \label{item:curtin_multivariate_ii_a}
For any $x_0\in X$, the block scheme $\mathfrak{X}_{x_0\mathcal{C}}$ is an $s^\ast$-variate $Q$-polynomial association scheme, the quotient scheme $\mathfrak{X}^\vee/\mathcal{C}^\ast$ is an $s^\ast$-variate $P$-polynomial association scheme, and these two schemes are formally dual;
\item \label{item:curtin_multivariate_ii_b}
The quotient scheme $\mathfrak{X}/\mathcal{C}$ is an $(\ell^\ast-s^\ast)$-variate $Q$-polynomial association scheme, every block scheme of $\mathfrak{X}^\vee$ with respect to $\mathcal{C}^\ast$ is an $(\ell^\ast-s^\ast)$-variate $P$-polynomial association scheme, and these schemes are formally dual.
\end{enumerate}
\end{enumerate}
\end{cor}

\begin{proof}
\ref{item:curtin_multivariate_i}~\ref{item:curtin_multivariate_i_a}:
By Theorem~\ref{thm:block_scheme_Ppoly}~\ref{item:block_scheme_Ppoly_i}, the block scheme $\mathfrak{X}_{x_0\mathcal{C}}$ is an $(\ell-s)$-variate $P$-polynomial association scheme.
On the other hand, Theorem~\ref{thm:curtin71}~\ref{item:curtin71_ii} shows that $\mathfrak{X}_{x_0\mathcal{C}}$ and $\mathfrak{X}^\vee/\mathcal{C}^\ast$ are formally dual.
Therefore Lemma~\ref{lem:formal_duality_multivariate}~\ref{item:formal_duality_multivariate_i} implies that $\mathfrak{X}^\vee/\mathcal{C}^\ast$ is an $(\ell-s)$-variate $Q$-polynomial association scheme.

The remaining three claims are obtained similarly by combining the corresponding parts of Theorem~\ref{thm:curtin71} and Lemma~\ref{lem:formal_duality_multivariate}.
\end{proof}


\subsection{Composition series of quotient schemes, sequences of block schemes, and elimination/extension theory}
\label{sec:quotient_composition}
In this subsection, we reinterpret chains of closed subsets through the lens of elimination theory.
Theorem~\ref{thm:EliminationThm} shows that each imprimitive step may be viewed as a single elimination step, while the corresponding block scheme records the reverse process, namely extension.
This gives a convenient way to organize composition factors arising from chains of closed subsets together with the associated sequence of block schemes.

More concretely, let $\mathfrak{X}=(X,\mathcal{R},\mathcal{I})$ be a commutative association scheme, fix $x_0\in X$, and let $\mathcal{C}'\subset \mathcal{C}\subset \mathcal{I}$ be closed subsets.
Then $\mathcal{C}'$ is a closed subset of $\mathfrak{X}_{x_0\mathcal{C}}$, so one may form the quotient scheme $\mathfrak{X}_{x_0\mathcal{C}}/\mathcal{C}'$.
Likewise, since $x_0\mathcal{C}'\subset x_0\mathcal{C}$, the scheme $\mathfrak{X}_{x_0\mathcal{C}'}$ may be regarded as a block scheme of $\mathfrak{X}_{x_0\mathcal{C}}$.

\begin{df}\label{df:quotient_composition_series}
For a commutative association scheme $\mathfrak{X}=(X,\mathcal{R},\mathcal{I})$, a strictly decreasing chain of closed subsets
\begin{equation}\label{eq:quotient_series_chain}
\mathcal{I}=\mathcal{C}_0 \supsetneq \mathcal{C}_1 \supsetneq \cdots \supsetneq \mathcal{C}_t = \{i_0\}
\end{equation}
is called a \emph{chain of closed subsets} of $\mathfrak{X}$.
Fixing $x_0\in X$, for each $r=0,1,\ldots,t-1$, define
\[
\mathfrak{X}_{\mathcal{C}_r/\mathcal{C}_{r+1}}
:= \mathfrak{X}_{x_0\mathcal{C}_r}/\mathcal{C}_{r+1}
\]
and call this the \emph{composition factor} corresponding to $(\mathcal{C}_r,\mathcal{C}_{r+1})$.
If this chain of closed subsets cannot be further refined (in terms of inclusion), i.e., if for each $r$ there is no closed subset strictly between $\mathcal{C}_{r+1}$ and $\mathcal{C}_r$, then \eqref{eq:quotient_series_chain} is called a \emph{composition series}.

Furthermore, the sequence of block schemes obtained from the chain of closed subsets,
\[
\mathfrak{X}_{x_0\mathcal{C}_t}\hookrightarrow \mathfrak{X}_{x_0\mathcal{C}_{t-1}}\hookrightarrow \cdots \hookrightarrow \mathfrak{X}_{x_0\mathcal{C}_0}=\mathfrak{X},
\]
is called the \emph{block sequence} associated with this composition series.
\end{df}

Let $\mathcal{I}=\mathcal{C}_0 \supsetneq \mathcal{C}_1 \supsetneq \cdots \supsetneq \mathcal{C}_t=\{i_0\}$
and $\mathcal{I}=\mathcal{D}_0 \supsetneq \mathcal{D}_1 \supsetneq \cdots \supsetneq \mathcal{D}_u=\{i_0\}$
be two composition series of a commutative association scheme $\mathfrak{X}$.
It is known that $t=u$, and there exists a permutation $\sigma$ of $\{0,1,\ldots,t-1\}$ such that
the Bose--Mesner algebras of $\mathfrak{X}_{\mathcal{C}_r/\mathcal{C}_{r+1}}$ and $\mathfrak{X}_{\mathcal{D}_{\sigma(r)}/\mathcal{D}_{\sigma(r)+1}}$
are isomorphic as algebras for each $r=0,1,\ldots,t-1$.
For more details on composition series in the general (not necessarily commutative) setting, see Rassy--Zieschang~\cite{RassyZieschang1998}; for the commutative background, compare also Section~2.9 of Bannai--Ito~\cite{BI1984}.

If we repeatedly apply Theorems \ref{thm:QuotientScheme}, \ref{thm:block_scheme_Ppoly}, and \ref{thm:ideal_block_quotient} to the block schemes $\mathfrak{X}_{x_0\mathcal{C}_r}$ arising from a composition series, then the structures of the Bose--Mesner algebras of the sequence of composition factors
$\mathfrak{X}_{\mathcal{C}_0/\mathcal{C}_1},\mathfrak{X}_{\mathcal{C}_1/\mathcal{C}_2}, \ldots,\mathfrak{X}_{\mathcal{C}_{t-1}/\mathcal{C}_t}$
and of the block sequence
$\mathfrak{X}_{x_0\mathcal{C}_t}, \mathfrak{X}_{x_0\mathcal{C}_{t-1}}, \ldots , \mathfrak{X}_{x_0\mathcal{C}_0}$
can be computed explicitly.
On the quotient side, if a monomial order adapted to $\mathcal{C}_{r+1}$ is given for $\mathfrak{X}_{x_0\mathcal{C}_r}$, then by performing ``linear substitution + elimination'' as in \eqref{eq:ideal_quotient_elim}, \eqref{eq:ideal_quotient_rescaling} on a Gr\"obner basis of the defining ideal $I_r$ of $\mathfrak{X}_{x_0\mathcal{C}_r}$, one obtains the defining ideal of the composition factor $\mathfrak{X}_{\mathcal{C}_r/\mathcal{C}_{r+1}}$.
That is, each composition factor is exactly a single elimination step.
On the block side, from a Gr\"obner basis of the defining ideal $J_r$ of $\mathfrak{X}_{x_0\mathcal{C}_{r}}$, one can directly read off the defining ideal $J_{r+1}$ of $\mathfrak{X}_{x_0\mathcal{C}_{r+1}}$ as $J_{r+1}=J_r\cap \C[y_{s_r+1},\ldots,y_{\ell_r}]$.
Conversely, tracing the block sequence upward can be understood as reconstructing $J_r$ by extending $J_{r+1}$ and adding new variables.
In this sense, the composition series simultaneously provides elimination theory for the composition factors and extension theory for the block sequence.

The next proposition gives a lower bound on the number of composition factors in a composition series of an $\ell$-variate $P$-polynomial association scheme determined by the lexicographic order.
\begin{prop}\label{prop:lexicographic_composition_length}
Assume that $\mathfrak{X}=(X,\mathcal{R},\mathcal{D})$ is an $\ell$-variate $P$-polynomial association scheme on $\mathcal{D}\subset\N^\ell$ with respect to the lexicographic order determined by $x_1>x_2>\cdots>x_\ell$.
For each $r=0,1,\ldots,\ell$, define
\[
\mathcal{C}^{\mathrm{lex}}_r:=\{\alpha=(a_1,\ldots,a_\ell)\in\mathcal{D}\mid a_1=\cdots=a_r=0\}.
\]
Then
\[
\mathcal{D}=\mathcal{C}^{\mathrm{lex}}_0\supsetneq \mathcal{C}^{\mathrm{lex}}_1\supsetneq \cdots \supsetneq \mathcal{C}^{\mathrm{lex}}_\ell=\{o\}
\]
is a chain of $\ell+1$ closed subsets. Consequently, every composition series of $\mathfrak{X}$ has at least $\ell$ composition factors.
\end{prop}

\begin{proof}
For each $r=1,2,\ldots,\ell-1$, the lexicographic order with $x_1>x_2>\cdots>x_\ell$ is of $r$-elimination type.
Hence, by the argument in the proof of Theorem~\ref{thm:EliminationThm}~\ref{item:EliminationThm_ii}$\Longrightarrow$\ref{item:EliminationThm_i},
$\mathcal{C}^{\mathrm{lex}}_r$ is a closed subset of $\mathcal{D}$.
The cases $r=0$ and $r=\ell$ are trivial.
Thus all $\mathcal{C}^{\mathrm{lex}}_r$ are closed.
The inclusions are strict because $\epsilon_r\in\mathcal{D}$ for every $r=1,2,\ldots,\ell$ by Definition~\ref{df:abPpoly}, and
$\epsilon_r\in \mathcal{C}^{\mathrm{lex}}_{r-1}\setminus \mathcal{C}^{\mathrm{lex}}_r$.
Therefore $\mathcal{D}=\mathcal{C}^{\mathrm{lex}}_0\supsetneq \mathcal{C}^{\mathrm{lex}}_1\supsetneq \cdots \supsetneq \mathcal{C}^{\mathrm{lex}}_\ell=\{o\}$
is a chain of $\ell+1$ closed subsets.

Since the set of closed subsets is finite, this chain extends to a maximal chain of closed subsets, hence to a composition series.
That composition series has at least $\ell$ composition factors.
All composition series of $\mathfrak{X}$ have the same number of composition factors.
Therefore every composition series has at least $\ell$ composition factors, equivalently $t\ge \ell$.
\end{proof}

\section{Open problems and future directions}
\label{sec:open_problems}

Theorem~\ref{thm:EliminationThm} shows that the imprimitivity of an association scheme $\mathfrak{X}$ is equivalent to the existence of a multivariate $P$-polynomial or $Q$-polynomial structure with respect to an elimination-type monomial order.
This equivalence provides a dictionary that translates classical questions about imprimitive schemes into the language of elimination theory and Gr\"obner bases, and it suggests new computational approaches to several natural problems.

We close by listing a few open problems closely related to imprimitivity and by indicating how they may be reinterpreted within the framework developed in this paper.

\subsection{Schurity testing and extensibility}

\begin{prob}[Schurity Testing Problem]
\label{prob:schurity}
Given a (commutative) association scheme $\mathfrak{X}$, determine whether it is \emph{Schurian} (i.e., arises from the orbit decomposition of some permutation group $G\le \mathrm{Sym}(X)$).
\end{prob}

The Schurity testing problem is classical and computationally important, and the existence of a polynomial-time algorithm in full generality remains open.
Arora--Zieschang~\cite{aroraZieschang2012}, building on Smith's characterization of Schurity and prescheme extensibility~\cite{smith1994,smith2007}, introduced the notion of height-$t$ preschemes ($t$-extensions) and maximal height $t_{\max}$, and also gave an algorithmic approach to testing extensibility.

From the point of view of this paper, Theorem~\ref{thm:EliminationThm} suggests that, in the imprimitive case, one should try to describe the Bose--Mesner algebra as a quotient of a polynomial ring adapted to an elimination order, explicitly identify the quotient-and-block structure arising from closed subsets as elimination data, and then decompose the global testing problem into block and quotient pieces.
In particular, it would be very interesting to understand how extensibility, or more concretely $t_{\max}$, behaves under block and quotient operations.

\begin{prob}[Decomposition Law for Extensibility in the Imprimitive Case]
\label{prob:extensibility_factor}
Let $\mathfrak{X}$ be an imprimitive scheme, and consider the block partition by a closed subset $\mathcal{C}$, with block scheme $\mathfrak{X}_{x_0\mathcal{C}}$ and quotient scheme $\mathfrak{X}/\mathcal{C}$. To what extent can $t_{\max}(\mathfrak{X})$ (or $t$-extensibility) be determined from $t_{\max}(\mathfrak{X}_{x_0\mathcal{C}})$ and $t_{\max}(\mathfrak{X}/\mathcal{C})$?
\end{prob}

\subsection{Existence problem for imprimitive non-symmetric 3-class schemes (Type 2)}

\begin{prob}[J{\o}rgensen's Open Case for Type 2]
\label{prob:jorgensen_type2}
Regarding imprimitive non-symmetric 3-class schemes, J{\o}rgensen organized type 2 (arising from doubly regular $(m,r)$-team tournaments), and in particular, posed the existence problem for $4\le r<m$ as ``the most interesting open problem'' (see \cite{jorgensen2009}). The smallest feasible case is $(r,m)=(4,10)$ (order $40$). Determine whether such a type 2 scheme exists.
\end{prob}

Problems of this kind are fundamentally existence problems for highly constrained combinatorial structures, and direct search quickly becomes infeasible.
On the other hand, Theorem~\ref{thm:EliminationThm} shows that any such scheme, if it exists, must admit a multivariate $P$-polynomial realization for some $\ell$, some domain $\mathcal{D}\subset\N^\ell$, and some elimination-type monomial order.
This reduces the existence question to the solvability of a zero-dimensional ideal equipped with an elimination Gr\"obner basis, and hence makes it accessible to computational algebra.
In the particularly small case of three classes, one may hope to eliminate variables step by step until either a contradiction appears or an explicit construction emerges.

\subsection{Systematic elimination of imprimitive parameter sets}

Recently there has been renewed progress on ruling out parameter sets of imprimitive association schemes that satisfy all currently known feasibility conditions but are not realized, or are realized uniquely, by using spherical embeddings of eigenspaces.
For example, Vidali~\cite{vidali2025} obtained new non-existence and uniqueness results for several open parameter sets arising from quotient-polynomial graphs.

From the perspective of this paper, an open parameter set should, whenever realizable, admit a polynomial realization compatible with an elimination order by Theorem~\ref{thm:EliminationThm}.
This realization is encoded by a zero-dimensional ideal $I$; see \eqref{eq:ideal}.
Accordingly, one may hope to detect contradictions already at the level of an elimination ideal such as $I\cap\C[x_{s+1},\ldots,x_\ell]$, thereby obtaining computational-algebraic non-existence proofs.
The spherical-embedding method and the elimination-ideal approach appear to impose genuinely different kinds of constraints, so combining them could lead to significantly stronger feasibility tests.

\vspace{2ex}

\noindent
{\bf{Acknowledgments}}

\vspace{2ex}

\noindent
We would like to thank Hajime Tanaka for his helpful comments.
The first named author is supported by JSPS KAKENHI Grant Numbers JP24K00521. 
The second named author is supported by JSPS KAKENHI Grant Numbers JP24K06830.


\end{document}